\documentclass[12pt]{article}
\usepackage[utf8]{inputenc}
\usepackage{lmodern}
\usepackage{amsmath,amsfonts,amsthm,amssymb}
\usepackage{bm,bbm}
\usepackage[margin=1.8cm]{geometry}
\usepackage{tikz}
\usepackage{csquotes}
\usepackage[american]{babel}
\usepackage{verbatim}
\usepackage[titletoc,title]{appendix}
\usepackage[ruled]{algorithm2e}
    \SetKwInput{KwData}{Input}
    \SetKwInput{KwResult}{Output}
\usepackage{authblk}
\usepackage{accents}
\usepackage{tkz-euclide}

\usepackage{microtype}
\DisableLigatures[f]{encoding = *}
\usepackage{tabularx} 
\usepackage{subfig} 
\usepackage{caption}
\usepackage[section]{placeins}
\usepackage{graphicx} 
\usepackage{epstopdf}
\usepackage{psfrag} 
\usepackage{standalone}
\usepackage[backend=biber, maxbibnames=99, giveninits=true, style=numeric-comp, isbn=false, doi=true, url=false]{biblatex}
\usepackage{xcolor}
\definecolor{ForestGreen}{HTML}{009B55}
\definecolor{NavyBlue}{HTML}{006EB8}
\definecolor{halfgray}{gray}{0.55}
\definecolor{webgreen}{rgb}{0,.5,0}
\definecolor{webbrown}{rgb}{.6,0,0}
\definecolor{dartmouthgreen}{rgb}{0.05, 0.5, 0.06}
\usepackage[colorlinks]{hyperref}
\hypersetup{urlcolor=NavyBlue,citecolor=ForestGreen}

\usepackage[capitalise]{cleveref}

\makeatletter
\@addtoreset{equation}{section}
\makeatother

\makeatletter
\@addtoreset{enunciato}{section}
\makeatother

\newcommand{\ie}{i.e.,~}

\newcommand{\N}{\mathbb N}
\newcommand{\R}{\mathbb R}

\renewcommand{\b}{\beta}

\newcommand{\h}{\eta}
\newcommand{\z}{\zeta}
\renewcommand{\o}{\omega}

\newcommand{\s}{\sigma}
\newcommand{\g}{\gamma}

\renewcommand{\L}{\Lambda}

\newcommand{\cC}{\mathcal{C}}

\newcommand{\cG}{\mathcal{G}}

\newcommand{\cV}{\mathcal{V}}
\newcommand{\cX}{\mathcal{X}}

\newcommand{\cO}{\mathcal{O}}
\newcommand{\cR}{\mathcal{R}}
\newcommand{\cS}{\mathcal{S}}

\newcommand{\Ve}{V_\ee}
\newcommand{\Vo}{V_\oo}
\newcommand{\xtbb}{\{X_{t}^\b\}_{t \in \N}}
\renewcommand{\ss}{\cX^s}

\newcommand{\ee}{\mathrm{\mathbf{e}}}
\newcommand{\oo}{\mathrm{\mathbf{o}}}
\newcommand{\pap}{\partial^+}

\newcommand{\D}{\Delta}

\newtheorem{thm}{Theorem}[section]
\newtheorem{cor}[thm]{Corollary}
\newtheorem{lem}[thm]{Lemma}
\newtheorem{prop}[thm]{Proposition}
\newtheorem{defi}[thm]{Definition}

\definecolor{cssgreen}{rgb}{0.0, 0.5, 0.0}

\title{Critical configurations of the hard-core model\\on square grid graphs}
\author[1,2]{Simone Baldassarri}
\affil[1]{Universit\`{a} degli Studi di Firenze, Firenze, Italy}
\affil[2]{Aix-Marseille Université, Marseille, France}
\author[3]{Vanessa Jacquier}
\affil[3]{University of Utrecht, Utrecht, The Netherlands}
\author[4]{Alessandro Zocca}
\affil[4]{Vrije Universiteit Amsterdam, The Netherlands}
\date{\today}
\date{\today}

\begin{document}
\maketitle

\begin{abstract}
We consider the hard-core model on a finite square grid graph with stochastic Glauber dynamics parametrized by the inverse temperature $\beta$. We investigate how the transition between its two maximum-occupancy configurations takes place in the low-temperature regime $\beta\to\infty$ in the case of periodic boundary conditions. The hard-core constraints and the grid symmetry make the structure of the critical configurations, also known as essential saddles, for this transition very rich and complex. We provide a comprehensive geometrical characterization of the set of critical configurations that are asymptotically visited with probability one. In particular, we develop a novel isoperimetric inequality for hard-core configurations with a fixed number of particles and we show how not only their size but also their shape determines the characterization of the saddles.
\end{abstract}



\medskip
\noindent
\textit{Acknowledgements:} S.B and V.J.~are grateful for the support of ``Gruppo Nazionale per l'Analisi Matematica, la Probabilità e le loro Applicazioni" (GNAMPA-INdAM).
The authors are grateful to Francesca Nardi, Julien Sohier, Gianmarco Bet, and Tommaso Monni for useful and fruitful discussions at the early stage of this work.

\section{Introduction} 
We consider a stochastic model, known in the literature as \textit{hard-core lattice gas model} \cite{gaunt1965hard,van1994percolation}, where particles have a non-negligible radius and therefore cannot overlap. Assuming a finite volume, the hard-core constraints are modeled with a finite undirected graph $\Lambda$. More specifically, particles can reside on the sites of $\Lambda$ and edges connect the pairs of sites in $\Lambda$ that cannot be simultaneously occupied. In other words, any hard-core configuration is an independent set of $\Lambda$. In this paper, we take $\Lambda$ to be a square grid graph with periodic boundary conditions. The resulting hard-core particle configurations are then those whose occupied sites have all the corresponding four neighboring sites empty, see~\cref{fig:configurazione} for an example of such configurations. 

This interacting particle system evolves according to a stochastic dynamics which is fully characterized by the \textit{Hamiltonian} or \textit{energy function} in \eqref{eq:energyhardcore} and is parametrized by the inverse temperature $\beta$. In particular, the appearance and disappearance of particles are modeled via a Glauber-type update Markov chain $\{X_t\}_{t\in\mathbb{N}}$ with Metropolis transition probabilities induced by the Hamiltonian, see \eqref{eq:metropolistransitionprobabilities} later for more details. The stochastic process is reversible with respect to the corresponding Gibbs measure, cf.~\eqref{eq:gibbs}, which is then its equilibrium distribution. 

In the low-temperature regime (i.e., $\beta \to \infty$), the most likely states for this interacting particle system, which we refer to as \textit{stable states}, are those with a maximum number of particles, namely. On the square grid graph $\Lambda$ of even length, there are two stable states, corresponding to the two chessboard-like patterns. When $\beta$ grows large, it takes the system a very long time to move from one stable state to the other since such a transition involves visiting intermediate configurations which are very unlikely. Such transitions become thus rare events and this is a central issue in the framework of metastability for interacting particle systems, which represents a thriving area in mathematical physics that is full of challenges. As a consequence, the stochastic process takes also a very long time to converge to stationarity, exhibiting so-called \textit{slow/torpid mixing} \cite{Galvin2008,Zocca2015}.

The asymptotic behavior of the first hitting times between the maximum-occupancy configurations of this model in the low-temperature regime has been studied in \cite{NZB15}. In particular, the authors showed how the order-of-magnitude of this first hitting time depends on the grid sizes and on the boundary conditions by means of an extension of the setting in \cite{Manzo2004}. Instead of leveraging directly the general strategy proposed in \cite{Manzo2004}, which allows us to derive the asymptotic behavior of the transition time together with a characterization of the critical configurations, the authors of \cite{NZB15} adopted a novel combinatorial method to estimate the energy barrier between the two stable states of the model, which is disentangled with respect to the description of the critical droplets. 

The main motivation of the present paper is to fill this knowledge gap. Indeed, the geometrical characterization of the \textit{essential gates} is a relevant goal both from a probabilistic and a physical point of view since it provides insightful details of the dynamical behavior of the system. This represents a crucial point in describing the typical trajectories, namely, those typically followed by the system during the transition from a stable state to the other. We remark that in several models analyzed in the context of Freidlin-Wentzell Markov chains evolving under Glauber dynamics, the essential gate was unique \cite{Apollonio2022} but, in general, there may exist many minimal sets that are crossed with high probability during the phase transition, either distinct or overlapping (see e.g. \cite{Baldassarri2022weak,Baldassarri2022strong} for this description in the case of the conservative Kawasaki dynamics). Interestingly, this is what happens also for our model despite it evolves under the non-conservative Glauber dynamics. Such a peculiar feature rests on the hard-core constraints and on the specific symmetry of the system, i.e., we are analyzing the tunneling transition between two stable states. Indeed, the fact that particles cannot appear in any site and the starting and target configurations have the same energy forces the system to visit many critical configurations before reaching the cycle of the target stable state. This is indeed also what happens for the Ising and Potts model evolving with the Glauber dynamics when there is no external magnetic field (see \cite{Bet2021} for instance), while when the symmetry of the system is broken, namely, an external magnetic field is present, the situation drastically changes. This different behavior has a major impact on the geometrical structure of the essential gates, which indeed turns out to be much richer than in the other cases and deserves a careful and detailed analysis.

In order to geometrically characterize the critical configurations, with each cluster of particles we associate its contour, that is a union of edges on the dual graph of $\Lambda$. To this end, we provide some results concerning the model-dependent isoperimetric inequality. In particular, we show that for a fixed area the unique clusters that minimize the perimeter have a \textit{rhomboidal} shape. However, the energy landscape is much more complex as the periodic boundary conditions give rise to other types of clusters with minimal perimeter for a given area, such as the configurations having a column containing a fixed number of particles.

In this paper, we adopt the framework of the \textit{pathwise approach}, introduced in \cite{Cassandro1984}, later developed in \cite{Olivieri1995,Olivieri1996}, and summarized in the monograph \cite{Olivieri2005}. A modern version of this approach can be found in \cite{NZB15,Cirillo2013,Cirillo2015,Manzo2004}. The pathwise approach has been widely adopted to the low-temperature behavior of finite-volume models with single-spin-flip Glauber dynamics, e.g. \cite{Apollonio2022,baldassarri2023ising,https://doi.org/10.48550/arxiv.2208.11869,https://doi.org/10.48550/arxiv.2108.04011,Bet2021,Betnegative,cirillo2023homogeneous,Nardi2019,Zocca2018,Zocca2018bis}, with Kawasaki dynamics, e.g. \cite{baldassarri2023metastability,Baldassarri2021,Baldassarri2022weak,Baldassarri2022strong,denHollander2000,Nardi2005}, and with parallel dynamics, e.g. \cite{Cirillo2008,Cirillo2008bis,Cirillo2022}. The more involved infinite-volume limit at low temperature was studied via this approach in \cite{denHollander2000,Gaudilliere2009,baldassarri2023droplet}.
Another method to study the metastability is the so-called \textit{potential-theoretic approach}, initiated in \cite{Bovier2002} and later summarized in the monograph \cite{Bovier2015} (see for instance \cite{Bovier2005,Bovier2010,Nardi2012} for the application of this approach to specific models both in finite and infinite volume). Since these two approaches rely on different definitions of metastable states, they are not completely equivalent. The situation is particularly delicate for infinite-volume systems, irreversible systems, and degenerate systems, as discussed in \cite{Cirillo2013,Cirillo2015,Bet2021PCA}. More recent approaches are developed in \cite{Beltrn2010,Beltrn2014,Bianchi2016,Bianchi2020,landim2021resolvent,kim2022approximation}. \medskip

The paper is organized as follows. In \cref{sec:model}, we provide a detailed model description and state our main result regarding the geometric features of the critical configurations, \cref{thm:saddles}. The rest of the paper is then devoted to the proof of this result. First, \cref{sec:defaux} provides some preliminary definitions and auxiliary results and then finally the proof of the main theorem is given in \cref{sec:proofthm}. For the sake of clarity, the proofs of some auxiliary lemmas are deferred to a later section, namely \cref{sec:auxproofs}. Lastly, \cref{sec:future} concludes the paper and outlines some future work.

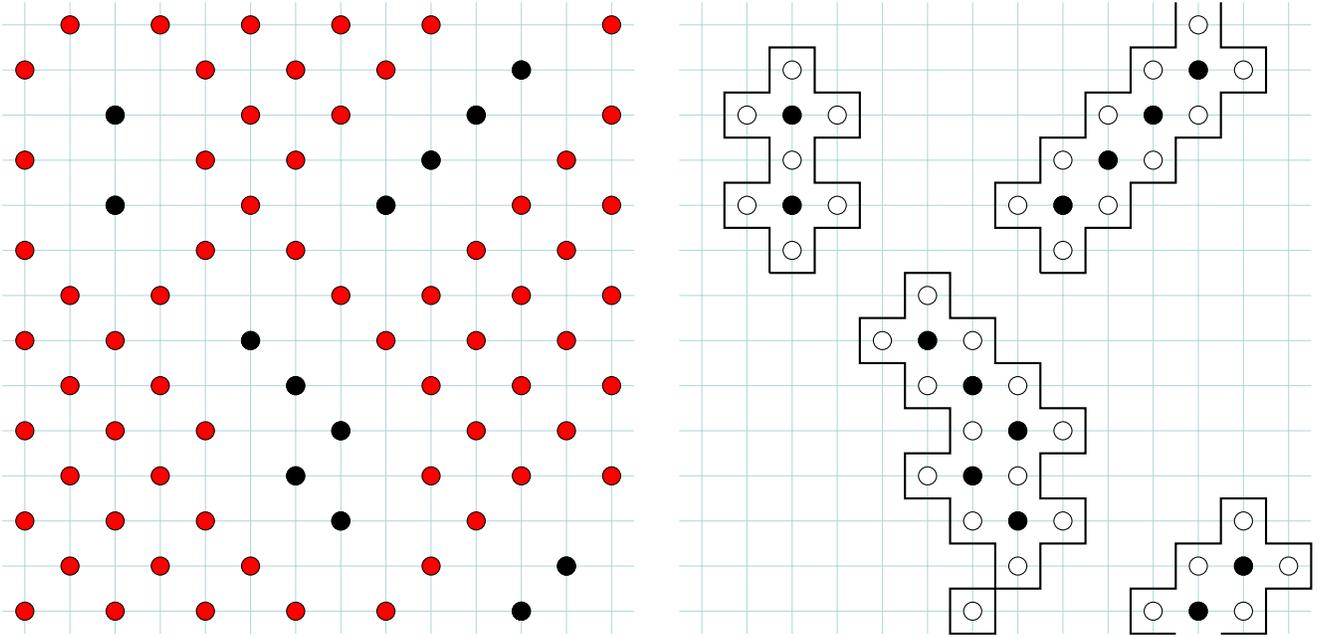
\begin{figure}
\centering
\begin{tikzpicture}[scale=0.6]
\newcommand\radius{0.2}
    \def\trasl{15}
    \newcommand\listodd{(3,-1),(2,0),(3,1),(2,2),(1,3),(4,6),(5,7),(6,8),(-2,6),(-2,8),(7,9),(7,-3),(8,-2), 
    (3+\trasl,-1),(2+\trasl,0),(3+\trasl,1),(2+\trasl,2),(1+\trasl,3),(4+\trasl,6),(5+\trasl,7),(6+\trasl,8),(-2+\trasl,6),(-2+\trasl,8),(7+\trasl,9),(7+\trasl,-3),(8+\trasl,-2)}
    \newcommand\listeven{(2+\trasl,-3),(3+\trasl,-2),(4+\trasl,-1),(3+\trasl,0),(4+\trasl,1),(3+\trasl,2),(2+\trasl,3),(1+\trasl,4),(0+\trasl,3),(1+\trasl,2),(2+\trasl,1),(1+\trasl,0),(2+\trasl,-1),(4+\trasl,5),(5+\trasl,6),(6+\trasl,7),(7+\trasl,8),(6+\trasl,9),(5+\trasl,8),(4+\trasl,7),(3+\trasl,6),(-2+\trasl,5),(-1+\trasl,6),(-2+\trasl,7),(-1+\trasl,8),(-2+\trasl,9),(-3+\trasl,8),(-3+\trasl,6),(8+\trasl,-1),(9+\trasl,-2),(7+\trasl,-2),(8+\trasl,-3),(6+\trasl,-3),(7+\trasl,10),(8+\trasl,9)}
    \draw[help lines] (-4.5,-3.5) grid (9.5,10.5);
\newcommand\listred{(-4,-3), (-4,-1), (-4,1), (-4,3), (-4,5), (-4,7),(-4,9), (-3,10), (-3,4),(-3,2),(-3,0),(-3,-2),(-2,-3),(-2,-1), (-2,1), (-2,3), (-1,10), (-1,4),(-1,2), (-1,0),(-1,-2), (0,-3),(0,-1),(0,1),(0,5),(0,7),(0,9),(1,10), (1,8),(1,6),(1,-2), (2,-3),(2,5),(2,7),(2,9),(3,10),(3,8),(3,4),(4,-3),(4,3),(4,9),(5,10),(5,4),(5,2),(5,0),(5,-2),(6,-1),(6,1),(6,3),(6,5),(7,6),(7,4),(7,2),(7,0),(8,1),(8,3),(8,5),(8,7),(9,10),(9,8),(9,6),(9,4),(9,2),(9,0)};

    \foreach \c in \listodd
        \draw[fill=black] \c circle(\radius);
    \foreach \c in \listred
       \draw[fill=red] \c circle(\radius);

    \draw[help lines] (-4.5+\trasl,-3.5) grid (9.5+\trasl,10.5);

    \draw[thick] (1.5+\trasl,-3.5) rectangle (2.5+\trasl,-2.5);

    \draw[thick] (2.5+\trasl,-2.5)--(3.5+\trasl,-2.5)--(3.5+\trasl,-1.5)--(4.5+\trasl,-1.5)--(4.5+\trasl,-0.5)--(3.5+\trasl,-0.5)--(3.5+\trasl,0.5)--(4.5+\trasl,0.5)--(4.5+\trasl,1.5)--(3.5+\trasl,1.5)--(3.5+\trasl,2.5)--(2.5+\trasl,2.5)--(2.5+\trasl,3.5)--(1.5+\trasl,3.5)--(1.5+\trasl,4.5)--(0.5+\trasl,4.5)--(0.5+\trasl,3.5)--(-0.5+\trasl,3.5)--(-0.5+\trasl,2.5)--(0.5+\trasl,2.5)--(0.5+\trasl,1.5)--(1.5+\trasl,1.5)--(1.5+\trasl,0.5)--(0.5+\trasl,0.5)--(0.5+\trasl,-0.5)--(1.5+\trasl,-0.5)--(1.5+\trasl,-1.5)--(2.5+\trasl,-1.5)--(2.5+\trasl,-2.5);

    \draw[thick] (3.5+\trasl,4.5)--(4.5+\trasl,4.5)--(4.5+\trasl,5.5)--(5.5+\trasl,5.5)--(5.5+\trasl,6.5)--(6.5+\trasl,6.5)--(6.5+\trasl,7.5)--(7.5+\trasl,7.5)--(7.5+\trasl,8.5)--(8.5+\trasl,8.5)--(8.5+\trasl,9.5)--(7.5+\trasl,9.5)--(7.5+\trasl,10.5);
    
    \draw[thick] (6.5+\trasl,10.5)--(6.5+\trasl,9.5)--(5.5+\trasl,9.5)--(5.5+\trasl,8.5)--(4.5+\trasl,8.5)--(4.5+\trasl,7.5)--(3.5+\trasl,7.5)--(3.5+\trasl,6.5)--(2.5+\trasl,6.5)--(2.5+\trasl,5.5)--(3.5+\trasl,5.5)--(3.5+\trasl,4.5);

    \draw[thick] (6.5+\trasl,-3.5)--(5.5+\trasl,-3.5)--(5.5+\trasl,-2.5)--(6.5+\trasl,-2.5)--(6.5+\trasl,-1.5)--(7.5+\trasl,-1.5)--(7.5+\trasl,-0.5)--(8.5+\trasl,-0.5)--(8.5+\trasl,-1.5)--(9.5+\trasl,-1.5)--(9.5+\trasl,-2.5)--(8.5+\trasl,-2.5)--(8.5+\trasl,-3.5)--(7.5+\trasl,-3.5);

    \draw[thick] (-2.5+\trasl,4.5)--(-1.5+\trasl,4.5)--(-1.5+\trasl,5.5)--(-0.5+\trasl,5.5)--(-0.5+\trasl,6.5)--(-1.5+\trasl,6.5)--(-1.5+\trasl,7.5)--(-0.5+\trasl,7.5)--(-0.5+\trasl,8.5)--(-1.5+\trasl,8.5)--(-1.5+\trasl,9.5)--(-2.5+\trasl,9.5)--(-2.5+\trasl,8.5)--(-3.5+\trasl,8.5)--(-3.5+\trasl,7.5)--(-2.5+\trasl,7.5)--(-2.5+\trasl,6.5)--(-3.5+\trasl,6.5)--(-3.5+\trasl,5.5)--(-2.5+\trasl,5.5)--(-2.5+\trasl,4.5);
    
    \foreach \c in \listodd
        \draw[fill=black] \c circle(\radius);
   \foreach \c in \listeven
        \draw[fill=white] \c circle(\radius);
\end{tikzpicture}
\caption{Example of a hard-core configuration on the $14\times 14$ square grid with periodic boundary conditions. On the left, the occupied sites in $\Vo$ (resp.\ in $\Ve$) are highlighted in black (resp.\ in red). On the right, we depict the same configuration using a different visual convention, in which we highlight the odd clusters that the configuration has by drawing only the empty sites in $\Ve$ (in white), the occupied sites in $\Vo$ (in black), and a black line around each odd cluster representing its contour.}
\label{fig:configurazione}
\end{figure}

\section{Model description and main results}
\label{sec:model}
We consider the stochastic evolution of the hard-core model on finite two-dimensional square lattices. More precisely, given an integer $L \geq 2$ we consider the $L\times L$ square grid graph $\L=(V,E)$ with periodic boundary conditions, which we will refer to as $L \times L$ \textit{toric grid graph}. We denote by $E$ the edge set of the grid graph $\L$ and by $V$ the collection of its $N=L^2$ sites. We identify each site $v \in \L$ by its coordinates $(v_1, v_2)$, that is we take as set of sites $V:=\{0,\dots,L-1\} \times \{0,\dots,L-1\}$. In the rest of the paper, we will assume that $L$ is an even integer, which guarantees that $\L$ is a bipartite graph, and that $L\geq6$, to avoid pathological trivial cases.

A particle configuration on $\L$ is described by associating a variable $\s(v)\in \{0,1\}$ to each site $v \in \L$, indicating the absence ($0$) or the presence ($1$) of a particle on that site. Let $\cX \subset \{0,1\}^{N}$ be the collection of \textit{hard-core configurations on $\L$}, \ie
\begin{equation}
\label{eq:statespace}
    \cX:=\{ \s \in \{0,1\}^N ~|~ \s(v) \s(w)=0, \, \, \forall \, (v,w) \in E\},
\end{equation}
i.e., the particle configurations on $\L$ with no particles residing on neighboring sites. 

A site of $\L$ is called \textit{even} (respectively \textit{odd}) if the sum of its two coordinates is even (respectively odd) and we denote by $\Ve$ and $\Vo$ the collection of even sites and that of odd sites of $\L$. Clearly $|\Ve| =|\Vo| = L^2/2$. We denote by $\ee$ ($\oo$, respectively) the particle configuration on $\L$ with particles at each site in $\Ve$ ($\Vo$, respectively), \ie
\[
    \ee(v) :=
    \begin{cases}
        1 & \text{ if } v \in \Ve,\\
        0 & \text{ if } v \in \Vo,
    \end{cases}
    \quad \text{ and } \quad 
    \oo(v) :=
    \begin{cases}
        0 & \text{ if } v \in \Ve,\\
        1 & \text{ if } v \in \Vo.
    \end{cases}
\]
Both $\ee$ and $\oo$ are hard-core configurations thanks to the assumption that $L$ is even. 

\cref{fig:configurazione} shows an example of a hard-core configuration. Throughout the paper, all figures are drawn using the following conventions. They all depict hard-core configurations on a $14\times 14$ grid with periodic boundary conditions. The occupied (empty) sites in $\Vo$ ($\Ve$, respectively) are shown in black (white) and we draw a black line around each odd cluster representing its contour. We tacitly assume that all the even (odd) sites outside the odd region are occupied (empty, respectively) but they are not displayed to avoid cluttering the figures. See \cref{sec:oddcluster} for more precise definitions of odd clusters and odd regions.

Consider the Gibbs measure on $\cX$ given by
\begin{equation}
\label{eq:gibbs}
    \mu_\b(\s):= \frac{e^{-\b H(\s)}}{Z_{\b,\L}}, \qquad \sigma \in \cX,
\end{equation}
where $H$ is the Hamiltonian $H: \cX \to \R$ that is taken to be proportional to the number of present particles, namely
\begin{equation}
\label{eq:energyhardcore}
    H(\s) := - \sum_{v \in V} \s(v),
\end{equation}
with $Z_{\b,\L}:=\sum_{\s \in \cX} e^{-\b H(\s)}$ being the normalizing constant.
The two hard-core configurations on the $L \times L$ toric grid graph $\Lambda$ introduced above have energy equal to
\[
    H(\ee) = H(\oo) = - \frac{L^2}{2},
\]
which is the minimum value the Hamiltonian can take on $\cX$~\cite{NZB15}.

We assume the interacting particle system described evolves according to stochastic Glauber-type dynamics described by a single-step update Markov chain $\smash{\xtbb}$ on $\cX$ with transition probabilities between any pair of configurations $\s,\s' \in \cX$ given by
\begin{equation}
\label{eq:metropolistransitionprobabilities}
    P_\b(\s,\s'):=
    \begin{cases}
        q(\s,\s') e^{-\b [H(\s')-H(\s)]^+}, 	& \text{ if } \s \neq \s',\\
        1-\sum_{\h \neq \s} P_\b(\s,\h), 	& \text{ if } \s=\s',
    \end{cases}
\end{equation}
where $[\cdot]^+=\max\{\cdot,0\}$ and $q$ is the \textit{connectivity matrix} that allows only single-step updates, \ie for every $\s,\s' \in \cX$ we set
\begin{equation}
\label{eq:connectivityfunction}
q(\s,\s'):=
\begin{cases}
    \frac{1}{N},    & \text{if } \left |\{v \in V : \s(v)\neq \s'(v)\} \right |=1,\\
    0,  & \text{if } \left |\{v \in V : \s(v)\neq \s'(v)\} \right |>1.\\
    1 - \sum_{\h \neq \s} q(\s,\h), & \text{if } \s=\s'.
\end{cases}
\end{equation}
The resulting dynamics $P_\b$ is reversible with respect to the Gibbs measure $\mu_\b$ given in~\eqref{eq:gibbs}. One usually refers to the triplet $(\cX, H, q)$ as \textit{energy landscape} and to~\eqref{eq:metropolistransitionprobabilities} as \textit{Metropolis transition probabilities}.

The connectivity matrix $q$ given in~\eqref{eq:connectivityfunction} is irreducible, \ie for any pair of configurations $\s,\s'\in \cX$, $\s \neq \s'$, there exists a finite sequence $\o$ of configurations $\o_1,\dots,\o_n \in \cX$ such that $\o_1=\s$, $\o_n=\s'$ and $q(\o_i,\o_{i+1})>0$, for $i=1,\dots, n-1$. We will refer to such a sequence as a \textit{path} from $\s$ to $\s'$ and denote it by $\o: \s \to \s'$. Given a path $\o=(\o_1,\dots,\o_n)$, we define its \textit{height} $\Phi_\o$ as
\begin{equation}
\label{eq:pathheight}
    \Phi_\o:= \max_{i=1,\dots,n} H(\o_i).
\end{equation}
The \textit{communication energy} between two configurations $\s,\s' \in\cX$ is the minimum value that has to be reached by the energy in every path $\o: \s \to \s'$, \ie
\begin{equation}
\label{eq:ch}
    \Phi(\s,\s') := \min_{\o : \s \to \s'} \Phi_\o = \min_{\o : \s \to \s'} \max_{\h \in\o} H(\h). 
\end{equation}
 
Let $\ss \subset \cX$ denote the set of global minima of the Hamiltonian $H$ on $\cX$, to which we will refer to as \textit{stable states}. 
In~\cite{NZB15} it has been proved that for the hard-core model on a finite $L \times L$ square grid graph the following statements hold:
\begin{itemize}
    \item[(i)] There are exactly two stable states
\begin{equation}
\label{eq:stablestates}
    \ss =\{\ee,\oo\};
\end{equation}
\item[(ii)] The communication energy between the two stable states is equal to
\begin{equation}
\label{eq:cheeoo}
    \Phi(\ee,\oo) - H(\ee)=L+1;
\end{equation}
\item[(iii)] The corresponding energy landscape has no deep wells, \ie
\begin{equation}
\label{eq:maxsl}
    \max_{\s \in \cX} [\Phi(\s, \{\ee,\oo\}) - H(\s)] \leq L < \Phi(\ee,\oo) - H(\ee).
\end{equation}
\end{itemize}

\newpage
\subsection{Essential saddle characterization}
\label{sec:result}
Our results give insight into the way the transitions between $\ee$ and $\oo$ most likely occur in the low--temperature regime. This is usually described by identifying the optimal paths, saddles, and essential saddles that we define as follows. 
\begin{itemize}
    \item $ \cS (\ee, \oo)$ is the \textit{communication level set} between $\ee$ and $\oo$ defined by
    \[
         \cS (\ee, \oo) := \{ \s \in \cX ~|~ \exists \, \o \in (\ee \to \oo)_{\mathrm{opt}},  ~:~ \s\in \o \text{ and } H(\s) = \Phi_\o = \Phi(\ee,\oo)\},
    \]
    where $(\ee \to \oo)_{\mathrm{opt}}$ is the set of {\it optimal paths} from $\ee$ to $\oo$ realizing the minimax in $\Phi(\ee, \oo)$, \ie
    \[
         (\ee \to \oo)_{\mathrm{opt}} := \{ \o: \ee \to \oo ~|~ \Phi_\o = \Phi(\ee,\oo)\}.
    \]
    \item The configurations in $\cS(\ee,\oo)$ are called \textit{saddles}. Given an optimal path $\o \in  (\ee \to \oo)_{\mathrm{opt}}$, we define the set of its saddles $S(\o)$ as 
    $
        S(\o):=\{ \s \in \o ~|~ H(\s) = \Phi_\o = \Phi(\ee,\oo)\}.
    $
    A saddle $\s \in \cS(\ee,\oo)$ is called \textit{essential} if either
    \begin{itemize}
    \item[(i)] $\exists \, \o \in (\ee \to \oo)_{\mathrm{opt}}$ such that $S(\o) = \{\s\}$, or
    \item[(ii)] $\exists \, \o \in (\ee \to \oo)_{\mathrm{opt}}$ such that $\s \in S(\o)$ and
        $S(\o') \not \subseteq S(\o) \setminus \{\s\} \quad \forall \, \o' \in (\ee \to \oo)_{\mathrm{opt}}$. 
    \end{itemize}
    A saddle $\s\in\cS(\ee,\oo)$ that is not essential is called \textit{unessential saddle} or \textit{dead-end}, i.e., for any $\o \in (\ee \to \oo)_{\mathrm{opt}}$ such that $\o\cap\{\s\}\neq\emptyset$ we have that $S(\o)\setminus\{\s\}\neq\emptyset$ and there exists $\o' \in (\ee \to \oo)_{\mathrm{opt}}$ such that $S(\o')\subseteq S(\o)\setminus\{\s\}$.
    \item The \textit{essential gate} $\cG(\ee, \oo)\subset \cX$ is the collection of essential saddles for the transition $\ee \to \oo$. 
\end{itemize}
The aim of the present paper is to accurately identify the set $\cG(\ee,\oo)$ of the essential saddles for the transition from $\ee$ to $\oo$ for the Metropolis dynamics of the hard-core model on a $L \times L$ grid with periodic boundary conditions. 
The set $\cG(\ee,\oo)$ will be described as the union of six disjoint sets, each characterized by configurations with specific geometrical features. While we refer the reader to~\cref{sec:proofthm} for a precise definition of these sets (cf.~Definitions~\ref{def:essential*}--\ref{def:essentialnuove2}), we provide here some intuitive descriptions of the geometrical features of the configurations in these sets. We denote by 
\begin{itemize}
    \item $\cC_{ir}(\ee,\oo)$, $\cC_{gr}(\ee,\oo)$, and $\cC_{cr}(\ee,\oo)$ the collections of configurations with a unique cluster of particles in odd sites of rhomboidal shape with exactly two adjacent even empty sites as in~\cref{fig:selle1} and~\cref{fig:selle2} (left). Roughly speaking, $\cC_{ir}(\ee,\oo)$ contains the configurations with $(\frac{L}{2}-1)^2$ occupied odd particles and $L^2+2$ empty even sites; 
    $\cC_{gr}(\ee,\oo)$ (resp.\ $\cC_{cr}(\ee,\oo)$) contains the configurations obtained from $\cC_{ir}(\ee,\oo)$ (resp.\ $\cC_{gr}(\ee,\oo)$) by removing some occupied even sites attached to the rhombus and by growing along one (resp.\ the longest) side by adding some particles in the nearest odd sites of the rhombus.

    \item $\cC_{sb}(\ee,\oo)$, $\cC_{mb}(\ee,\oo)$, and $\cC_{ib}(\ee,\oo)$ the collections of configurations with a unique cluster of particles in odd sites with at most two additional empty even sites as in~\cref{fig:selle2} (right) and in~\cref{fig:selle3}. In particular, $\cC_{sb}(\ee,\oo)$ contains the configurations with $\frac{L}{2}-1$ particles arranged in an odd column with further two empty even sites;
    $\cC_{mb}(\ee,\oo)$ contains the configurations obtained from $\cC_{sb}(\ee,\oo)$ such that there is at least one column or row with $\frac{L}{2}$ particles arranged in odd sites.
    $\cC_{ib}(\ee,\oo)$ contains the configurations obtained from $\cC_{sb}(\ee,\oo)$ without having column or row with $\frac{L}{2}$ particles.
\end{itemize}

\begin{figure}[!h]
\centering
    \begin{tikzpicture}[scale=0.55]
    \newcommand\radius{0.2}
    \newcommand\listodd{(3,-1),(4,0),(5,1),(6,2),(7,3),(8,4),(2,0),(3,1),(4,2),(5,3),(6,4),(7,5),(1,1),(2,2),(3,3),(4,4),(5,5),(6,6),(0,2),(1,3),(2,4),(3,5),(4,6),(5,7),(-1,3),(0,4),(1,5),(2,6),(3,7),(4,8),(-2,4),(-1,5),(0,6),(1,7),(2,8),(3,9)}
    \newcommand\listeven{(3,-2),(4,-1),(5,0),(6,1),(7,2),(8,3),(9,4),(2,-1),(3,0),(4,1),(5,2),(6,3),(7,4),(8,5),(1,0),(2,1),(3,2),(4,3),(5,4),(6,5),(7,6),(0,1),(1,2),(2,3),(3,4),(4,5),(5,6),(6,7),(-1,2),(0,3),(1,4),(2,5),(3,6),(4,7),(5,8),(-2,3),(-1,4),(0,5),(1,6),(2,7),(3,8),(4,9),(-3,4),(-2,5),(-1,6),(0,7),(1,8),(2,9),(3,10),(-4,5),(-3,6)}
    
    \draw[help lines] (-4.5,-3.5) grid (9.5,10.5);

     \draw[thick] (2.5,-2.5)--(3.5,-2.5)--(3.5,-1.5)--(4.5,-1.5)--(4.5,-0.5)--(5.5,-0.5)--(5.5,0.5)--(6.5,0.5)--(6.5,1.5)--(7.5,1.5)--(7.5,2.5)--(8.5,2.5)--(8.5,3.5)--(9.5,3.5)--(9.5,4.5)--(8.5,4.5)--(8.5,5.5)--(7.5,5.5)--(7.5,6.5)--(6.5,6.5)--(6.5,7.5)--(5.5,7.5)--(5.5,8.5)--(4.5,8.5)--(4.5,9.5)--(3.5,9.5)--(3.5,10.5)--(2.5,10.5)--(2.5,9.5)--(1.5,9.5)--(1.5,8.5)--(0.5,8.5)--(0.5,7.5)--(-0.5,7.5)--(-0.5,6.5)--(-1.5,6.5)--(-1.5,5.5)--(-2.5,5.5)--(-2.5,4.5)--(-3.5,4.5)--(-3.5,3.5)--(-2.5,3.5)--(-2.5,2.5)--(-1.5,2.5)--(-1.5,1.5)--(-0.5,1.5)--(-0.5,0.5)--(0.5,0.5)--(0.5,-0.5)--(1.5,-0.5)--(1.5,-1.5)--(2.5,-1.5)--(2.5,-2.5);
     
     \draw[thick] (-2.5,5.5) rectangle (-3.5,6.5);
     \draw[thick] (-4.5,5.5) rectangle (-3.5,4.5);
   
    \foreach \c in \listodd
        \draw[fill=black] \c circle(\radius);
    \foreach \c in \listeven
        \draw[fill=white] \c circle(\radius);

\end{tikzpicture}
\hspace{1cm}
\begin{tikzpicture}[scale=0.55]
 \newcommand\radius{0.2}
    \newcommand\listodd{(3,-1),(4,0),(5,1),(6,2),(7,3),(8,4),(2,0),(3,1),(4,2),(5,3),(6,4),(7,5),(1,1),(2,2),(3,3),(4,4),(5,5),(6,6),(0,2),(1,3),(2,4),(3,5),(4,6),(5,7),(-1,3),(0,4),(1,5),(2,6),(3,7),(4,8),(-2,4),(-1,5),(0,6),(1,7),(2,8),(3,9),(-3,5),(-2,6),(-1,7)}
    \newcommand\listeven{(3,-2),(4,-1),(5,0),(6,1),(7,2),(8,3),(9,4),(2,-1),(3,0),(4,1),(5,2),(6,3),(7,4),(8,5),(1,0),(2,1),(3,2),(4,3),(5,4),(6,5),(7,6),(0,1),(1,2),(2,3),(3,4),(4,5),(5,6),(6,7),(-1,2),(0,3),(1,4),(2,5),(3,6),(4,7),(5,8),(-2,3),(-1,4),(0,5),(1,6),(2,7),(3,8),(4,9),(-3,4),(-2,5),(-1,6),(0,7),(1,8),(2,9),(3,10),(-4,5),(-3,6),(-2,7),(-1,8),(0,9)}
    
    \draw[help lines] (-4.5,-3.5) grid (9.5,10.5);
        
    \draw[thick] (2.5,-2.5)--(3.5,-2.5)--(3.5,-1.5)--(4.5,-1.5)--(4.5,-0.5)--(5.5,-0.5)--(5.5,0.5)--(6.5,0.5)--(6.5,1.5)--(7.5,1.5)--(7.5,2.5)--(8.5,2.5)--(8.5,3.5)--(9.5,3.5)--(9.5,4.5)--(8.5,4.5)--(8.5,5.5)--(7.5,5.5)--(7.5,6.5)--(6.5,6.5)--(6.5,7.5)--(5.5,7.5)--(5.5,8.5)--(4.5,8.5)--(4.5,9.5)--(3.5,9.5)--(3.5,10.5)--(2.5,10.5)--(2.5,9.5)--(1.5,9.5)--(1.5,8.5)--(0.5,8.5)--(0.5,7.5)--(-0.5,7.5)--(-0.5,8.5)--(-1.5,8.5)--(-1.5,7.5)--(-2.5,7.5)--(-2.5,6.5)--(-3.5,6.5)--(-3.5,5.5)--(-4.5,5.5)--(-4.5,4.5)--(-3.5,4.5)--(-3.5,3.5)--(-2.5,3.5)--(-2.5,2.5)--(-1.5,2.5)--(-1.5,1.5)--(-0.5,1.5)--(-0.5,0.5)--(0.5,0.5)--(0.5,-0.5)--(1.5,-0.5)--(1.5,-1.5)--(2.5,-1.5)--(2.5,-2.5);

    \draw[thick] (-0.5,8.5) rectangle (0.5,9.5);
        
    \foreach \c in \listodd
        \draw[fill=black] \c circle(\radius);
    \foreach \c in \listeven
        \draw[fill=white] \c circle(\radius);
\end{tikzpicture}
\caption{An example of a configuration in $\cC_{ir}(\ee,\oo)$ (on the left) and one in $\cC_{gr}(\ee,\oo)$ (on the right).}
\label{fig:selle1}
\end{figure}
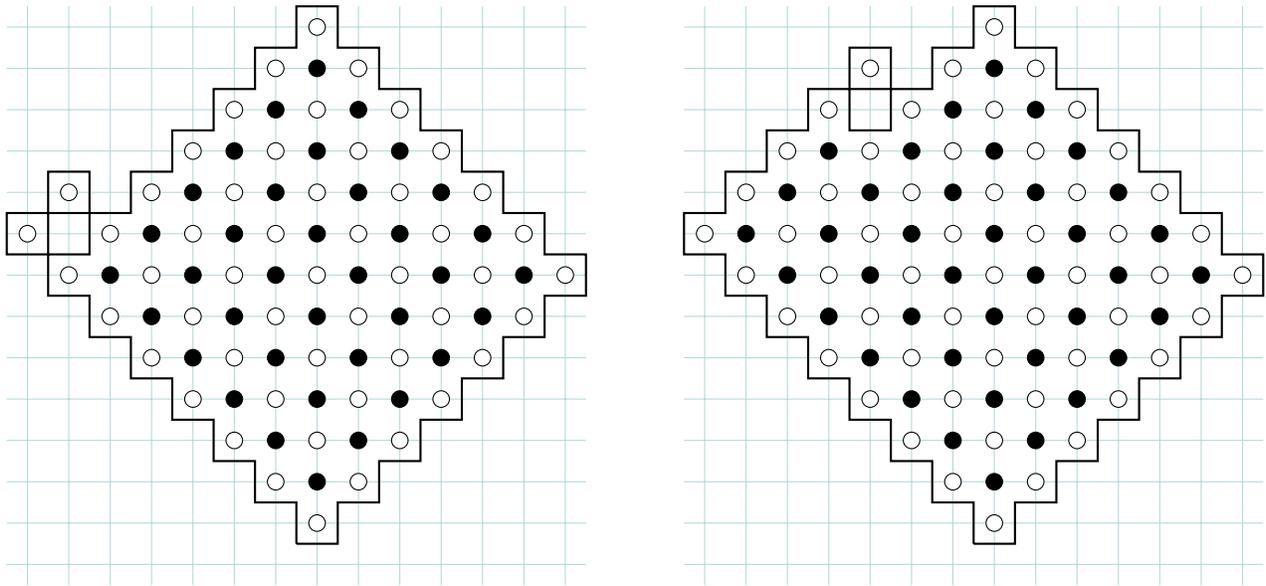

\begin{figure}[h!]
\centering
\begin{tikzpicture}[scale=0.55]
 \newcommand\radius{0.2}
    \newcommand\listodd{(3,-1),(4,0),(5,1),(6,2),(7,3),(8,4),(2,0),(3,1),(4,2),(5,3),(6,4),(7,5),(1,1),(2,2),(3,3),(4,4),(5,5),(6,6),(0,2),(1,3),(2,4),(3,5),(4,6),(5,7),(-1,3),(0,4),(1,5),(2,6),(3,7),(4,8),(-2,4),(-1,5),(0,6),(1,7),(2,8),(3,9),(-3,5),(-2,6),(-1,7),(0,8),(1,9),(2,10),(2,-2),(1,-1),(0,0),(-1,1)}
    \newcommand\listeven{(3,-2),(4,-1),(5,0),(6,1),(7,2),(8,3),(9,4),(2,-1),(3,0),(4,1),(5,2),(6,3),(7,4),(8,5),(1,0),(2,1),(3,2),(4,3),(5,4),(6,5),(7,6),(0,1),(1,2),(2,3),(3,4),(4,5),(5,6),(6,7),(-1,2),(0,3),(1,4),(2,5),(3,6),(4,7),(5,8),(-2,3),(-1,4),(0,5),(1,6),(2,7),(3,8),(4,9),(-3,4),(-2,5),(-1,6),(0,7),(1,8),(2,9),(3,10),(-4,5),(-3,6),(-2,7),(-1,8),(0,9),(1,10),(2,-3),(1,-2),(0,-1),(-1,0),(-2,1),(-3,2)}
    
    \draw[help lines] (-4.5,-3.5) grid (9.5,10.5);
        
    \draw[thick] (2.5,-3.5)--(2.5,-2.5)--(3.5,-2.5)--(3.5,-1.5)--(4.5,-1.5)--(4.5,-0.5)--(5.5,-0.5)--(5.5,0.5)--(6.5,0.5)--(6.5,1.5)--(7.5,1.5)--(7.5,2.5)--(8.5,2.5)--(8.5,3.5)--(9.5,3.5)--(9.5,4.5)--(8.5,4.5)--(8.5,5.5)--(7.5,5.5)--(7.5,6.5)--(6.5,6.5)--(6.5,7.5)--(5.5,7.5)--(5.5,8.5)--(4.5,8.5)--(4.5,9.5)--(3.5,9.5)--(3.5,10.5)--(2.5,10.5);

    \draw[thick] (1.5,10.5)--(0.5,10.5)--(0.5,9.5)--(-0.5,9.5)--(-0.5,8.5)--(-1.5,8.5)--(-1.5,7.5)--(-2.5,7.5)--(-2.5,6.5)--(-3.5,6.5)--(-3.5,5.5)--(-4.5,5.5)--(-4.5,4.5)--(-3.5,4.5)--(-3.5,3.5)--(-2.5,3.5)--(-2.5,2.5)--(-1.5,2.5)--(-1.5,1.5)--(-2.5,1.5)--(-2.5,0.5)--(-1.5,0.5)--(-1.5,-0.5)--(-0.5,-0.5)--(-0.5,-1.5)--(0.5,-1.5)--(0.5,-2.5)--(1.5,-2.5)--(1.5,-3.5);

    \draw[thick] (-2.5,2.5) rectangle (-3.5,1.5);
        
    \foreach \c in \listodd
        \draw[fill=black] \c circle(\radius);
    \foreach \c in \listeven
        \draw[fill=white] \c circle(\radius);
\end{tikzpicture}
\hspace{1cm}
    \begin{tikzpicture}[scale=0.55]
    \newcommand\radius{0.2}
    \newcommand\listodd{(3,-3),(3,1),(3,3),(3,5),(3,7),(3,9)}
    \newcommand\listeven{(2,-3),(4,-3),(3,-2),(4,-1),(2,-1),(3,0),(4,1),(2,1),(3,2),(4,3),(2,3),(3,4),(4,5),(2,5),(3,6),(4,7),(2,7),(3,8),(4,9),(2,9),(3,10)}
    
    \draw[help lines] (-4.5,-3.5) grid (9.5,10.5);

    \draw[thick] (1.5,-3.5)--(2.5,-3.5);
    \draw[thick] (3.5,-1.5)--(2.5,-1.5);
    \draw[thick] (3.5,-0.5)--(2.5,-0.5);
    \draw[thick] (2.5,-2.5)--(1.5,-2.5)--(1.5,-3.5);
    \draw[thick] (3.5,-3.5)--(4.5,-3.5)--(4.5,-2.5)--(3.5,-2.5)--(3.5,-1.5)--(4.5,-1.5)--(4.5,-0.5)--(3.5,-0.5)--(3.5,0.5)--(4.5,0.5)--(4.5,1.5)--(3.5,1.5)--(3.5,2.5)--(4.5,2.5)--(4.5,3.5)--(3.5,3.5)--(3.5,4.5)--(4.5,4.5)--(4.5,5.5)--(3.5,5.5)--(3.5,6.5)--(4.5,6.5)--(4.5,7.5)--(3.5,7.5)--(3.5,8.5)--(4.5,8.5)--(4.5,9.5)--(3.5,9.5)--(3.5,10.5);
    \draw[thick] (2.5,10.5)--(2.5,9.5)--(1.5,9.5)--(1.5,8.5)--(2.5,8.5)--(2.5,7.5)--(1.5,7.5)--(1.5,6.5)--(2.5,6.5)--(2.5,5.5)--(1.5,5.5)--(1.5,4.5)--(2.5,4.5)--(2.5,3.5)--(1.5,3.5)--(1.5,2.5)--(2.5,2.5)--(2.5,1.5)--(1.5,1.5)--(1.5,0.5)--(2.5,0.5)--(2.5,-0.5)--(1.5,-0.5)--(1.5,-1.5)--(2.5,-1.5)--(2.5,-2.5);

    \draw[thick] (1.5,-1.5) rectangle (2.5,-0.5);
    \draw[thick] (3.5,-1.5) rectangle (4.5,-0 .5);

    \foreach \c in \listodd
        \draw[fill=black] \c circle(\radius);
    \foreach \c in \listeven
        \draw[fill=white] \c circle(\radius);
\end{tikzpicture}
\caption{An example of a configuration in $\cC_{cr}(\ee,\oo)$ (on the left) and one in $\cC_{sb}(\ee,\oo)$ (on the right).}
\label{fig:selle2}
\end{figure}
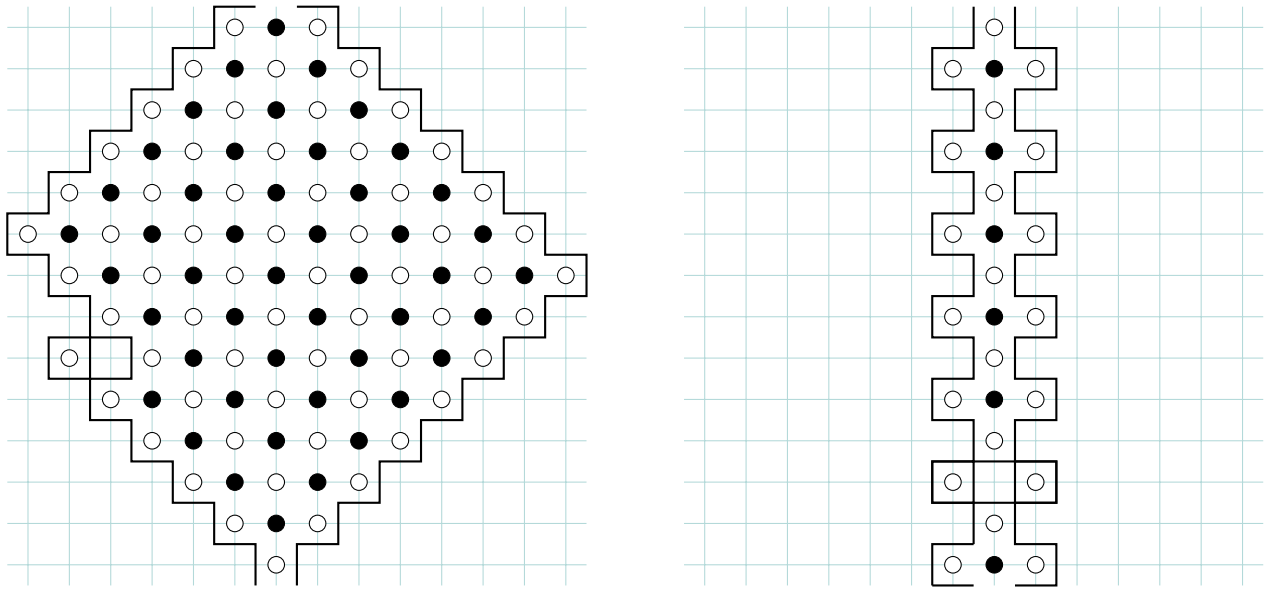

\begin{figure}[h!]
\centering
\begin{tikzpicture}[scale=0.55]
    \newcommand\radius{0.2}
    \newcommand\listodd{(3,-3),(3,-1),(3,1),(3,3),(3,5),(3,7),(3,9),(2,8),(2,6),(2,4),(2,0),(2,-2),(1,-1),(1,5),(1,7),(4,8),(4,6),(4,4),(4,2),(4,0),(5,1),(5,5),(5,7)}
    \newcommand\listeven{(2,-3),(4,-3),(3,-2),(4,-1),(2,-1),(3,0),(4,1),(2,1),(3,2),(4,3),(2,3),(3,4),(4,5),(2,5),(3,6),(4,7),(2,7),(3,8),(4,9),(2,9),(3,10),(1,8),(0,7),(-1,6),(1,6),(0,5),(1,4),(1,0),(0,-1),(1,-2),(5,8),(6,7),(5,6),(6,5),(5,4),(5,2),(6,1),(5,0)}

    \draw[thick] (3.5,-3.5)--(4.5,-3.5)--(4.5,-2.5)--(3.5,-2.5)--(3.5,-1.5)--(4.5,-1.5)--(4.5,-0.5)--(5.5,-0.5)--(5.5,0.5)--(6.5,0.5)--(6.5,1.5)--(5.5,1.5)--(5.5,2.5)--(4.5,2.5)--(4.5,3.5)--(5.5,3.5)--(5.5,4.5)--(6.5,4.5)--(6.5,5.5)--(5.5,5.5)--(5.5,6.5)--(6.5,6.5)--(6.5,7.5)--(5.5,7.5)--(5.5,8.5)--(4.5,8.5)--(4.5,9.5)--(3.5,9.5)--(3.5,10.5);

    \draw[thick] (2.5,10.5)--(2.5,9.5)--(1.5,9.5)--(1.5,8.5)--(0.5,8.5)--(0.5,7.5)--(-0.5,7.5)--(-0.5,6.5)--(0.5,6.5)--(0.5,5.5)--(-0.5,5.5)--(-0.5,4.5)--(0.5,4.5)--(0.5,3.5)--(1.5,3.5)--(1.5,2.5)--(2.5,2.5)--(2.5,1.5)--(1.5,1.5)--(1.5,0.5)--(0.5,0.5)--(0.5,-0.5)--(-0.5,-0.5)--(-0.5,-1.5)--(0.5,-1.5)--(0.5,-2.5)--(1.5,-2.5)--(1.5,-3.5)--(2.5,-3.5);

    \draw[thick] (-0.5,6.5) rectangle (-1.5,5.5);

    \draw[help lines] (-4.5,-3.5) grid (9.5,10.5);

    \foreach \c in \listodd
        \draw[fill=black] \c circle(\radius);
    \foreach \c in \listeven
        \draw[fill=white] \c circle(\radius);
\end{tikzpicture}
\hspace{1cm}
    \begin{tikzpicture}[scale=0.55]
    \newcommand\radius{0.2}
    \newcommand\listodd{(3,-1),(3,1),(3,3),(3,5),(3,7),(2,8),(2,6),(2,4),(2,2),(2,0),(2,-2),(1,-1),(1,1),(1,3),(1,5),(1,7),(0,6),(0,4),(0,2),(0,0),(-1,1),(-1,3),(-1,5),(-2,4),(-2,2),(-3,3),(4,6),(4,4),(4,2),(4,0),(3,9)}
    \newcommand\listeven{(2,-3),(3,-2),(4,-1),(2,-1),(3,0),(4,1),(2,1),(3,2),(4,3),(2,3),(3,4),(4,5),(2,5),(3,6),(2,7),(3,8),(2,9),(1,8),(1,6),(1,4),(1,2),(1,0),(1,-2),(0,-1),(0,1),(0,3),(0,5),(0,7),(-1,6),(-1,4),(-1,2),(-1,0),(-2,1),(-2,3),(-2,5),(-3,4),(-3,2),(-4,3),(5,6),(5,4),(5,2),(5,0),(6,5),(4,7),(4,9),(3,10)}
    
    \draw[help lines] (-4.5,-3.5) grid (9.5,10.5);
    \draw[thick] (5.5,4.5) rectangle (6.5,5.5);

    \draw[thick] (1.5,-3.5)--(2.5,-3.5)--(2.5,-2.5)--(3.5,-2.5)--(3.5,-1.5)--(4.5,-1.5)--(4.5,-0.5)--(5.5,-0.5)--(5.5,0.5)--(4.5,0.5)--(4.5,1.5)--(5.5,1.5)--(5.5,2.5)--(4.5,2.5)--(4.5,3.5)--(5.5,3.5)--(5.5,4.5)--(4.5,4.5)--(4.5,5.5)--(5.5,5.5)--(5.5,6.5)--(4.5,6.5)--(4.5,7.5)--(3.5,7.5)--(3.5,8.5)--(4.5,8.5)--(4.5,9.5)--(3.5,9.5)--(3.5,10.5)--(2.5,10.5)--(2.5,9.5)--(1.5,9.5)--(1.5,8.5)--(0.5,8.5)--(0.5,7.5)--(-0.5,7.5)--(-0.5,6.5)--(-1.5,6.5)--(-1.5,5.5)--(-2.5,5.5)--(-2.5,4.5)--(-3.5,4.5)--(-3.5,3.5)--(-4.5,3.5)--(-4.5,2.5)--(-3.5,2.5)--(-3.5,1.5)--(-2.5,1.5)--(-2.5,0.5)--(-1.5,0.5)--(-1.5,-0.5)--(-0.5,-0.5)--(-0.5,-1.5)--(0.5,-1.5)--(0.5,-2.5)--(1.5,-2.5)--(1.5,-3.5);

    \foreach \c in \listodd
        \draw[fill=black] \c circle(\radius);
    \foreach \c in \listeven
        \draw[fill=white] \c circle(\radius);
\end{tikzpicture}
\caption{An example of a configuration in $\cC_{mb}(\ee,\oo)$ (on the left) and one in $\cC_{ib}(\ee,\oo)$ (on the right).}
\label{fig:selle3}
\end{figure}
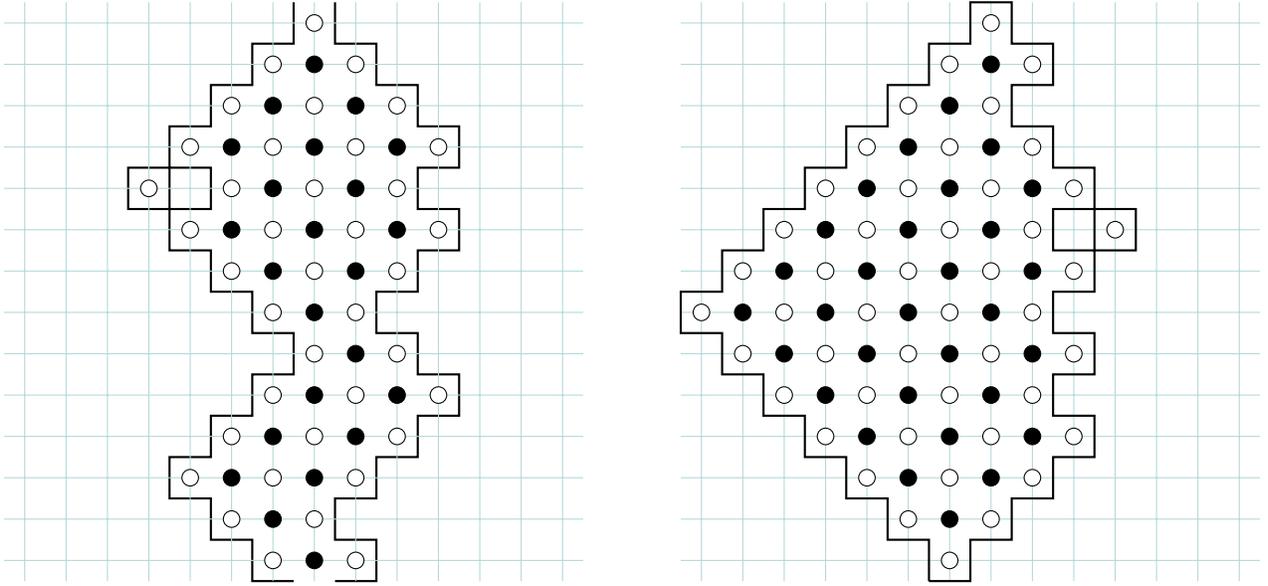

\begin{figure}[h!]
    \includegraphics[width=\textwidth]{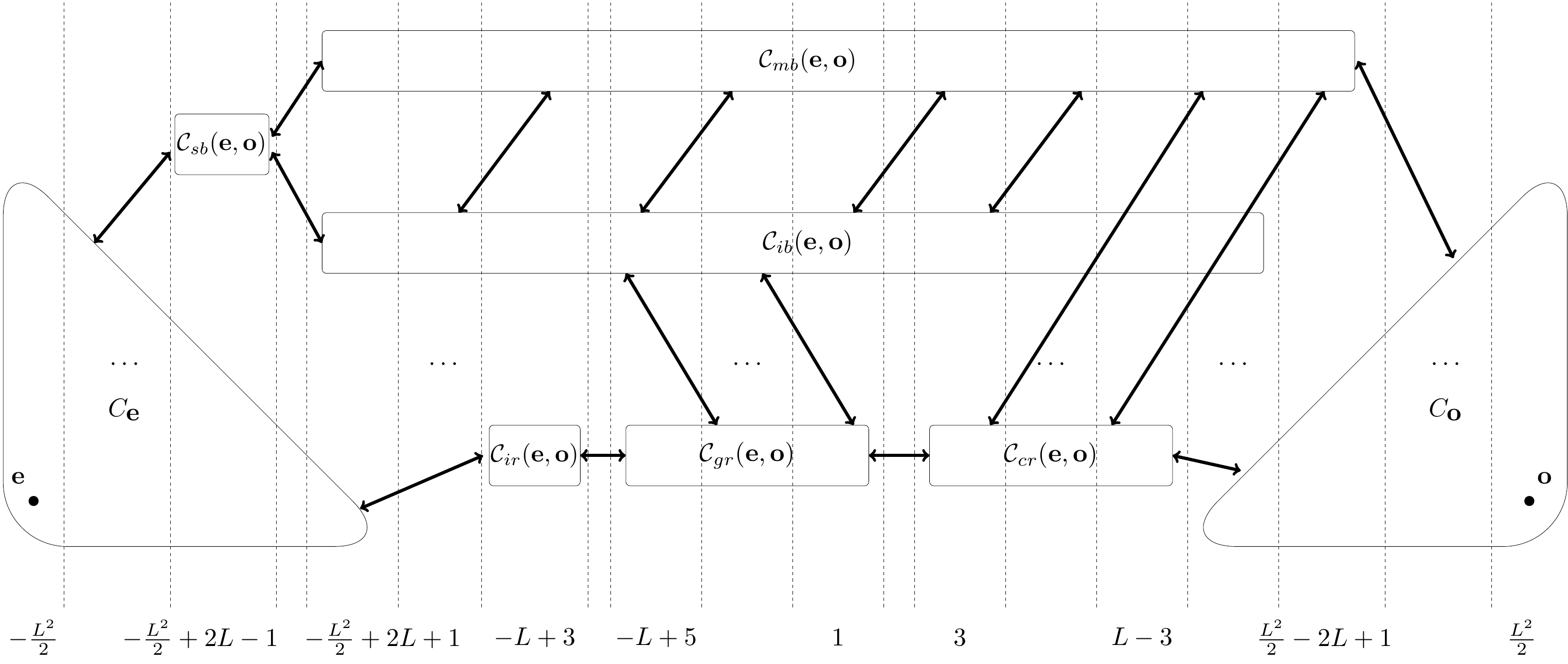}
    \caption{Schematic representation of the set of essential saddles, where we highlight with arrows between the set pairs that communicate at energy not higher than $-\frac{L^2}{2}+L+1$ and the initial cycles $\cC_{\ee}$ and $\cC_{\oo}$, see \cref{sec:defaux}. The vertical lines represent the partition of $\cX$ in manifolds, see \eqref{eq:foliation}.}
\label{fig:panorama}
\end{figure}

The following theorem characterizes the essential gate for the transition from $\ee$ to $\oo$.
\begin{thm}[Essential saddles]\label{thm:saddles}
Define the set
\[
    \cC^*(\ee,\oo):=\cC_{ir}(\ee,\oo)\cup\cC_{gr}(\ee,\oo)\cup\cC_{cr}(\ee,\oo)\cup\cC_{sb}(\ee,\oo)\cup\cC_{mb}(\ee,\oo)\cup\cC_{ib}(\ee,\oo).
\]
The essential saddles for the transition from $\ee$ to $\oo$ of the hard-core model on a $L \times L$ toric grid graph $\Lambda$ are all and only the configurations in $\mathcal C^*(\ee,\oo)$, i.e.,
\[
    \mathcal G(\ee,\oo)=\mathcal C^*(\ee,\oo).
\]
Furthermore, the possible transitions at energy not higher than $-\frac{L^2}{2}+L+1$ among the six subsets forming $\mathcal C^*(\ee,\oo)$ are as detailed in~\cref{fig:panorama}.
\end{thm}

\section{Definitions and auxiliary results}
\label{sec:defaux}
The main goal of this section is to introduce the notion of odd clusters, which are the basis of the geometrical description of the configurations, and to inspect the relation between their shape and perimeter.

In \cref{sec:oddcluster} we define a geometrical representation of clusters associated with the occupied odd sites, and in \cref{sec:oddrhombi} we introduce the notion of rhombi, which turns out to be crucial in the description of the essential saddles. In \cref{sec:filling} we present two algorithms that, combined together, return a path whose last configuration has a rhomboidal shape and such that the energy along it never increases. We will use them to deduce that there exists a downhill path from the configurations without a rhomboidal cluster towards $\ee$ or $\oo$. 

Along the lines of~\cite[eq.~(2.7)]{Manzo2004}, we define $C_{\ee}:=\{ \z \in \cX ~|~ \Phi(\z,\ee) < \Phi(\ee,\oo)\}$ to be the \textit{initial cycle} of $\ee$, that is the maximal cycle that includes $\ee$ but does not include $\oo$, namely, it contains all the configurations that can be reached by $\ee$ by spending strictly less energy than the one needed for the transition between $\ee$ and $\oo$, i.e., the communication height $\Phi(\ee,\oo)$. The corresponding initial cycle of $\oo$ is defined analogously and denoted by $C_{\oo}$.

Given a configuration $\sigma \in \cX$, denote by $\Delta H(\sigma)$ the energy difference with respect to either one of the stable states, i.e.,
\begin{equation}\label{eq:pomodorina}
    \Delta H(\sigma):=H(\sigma)-H(\ee).
\end{equation}

\subsection{Odd clusters and regions}
\label{sec:oddcluster}
For any subset of sites $S \subseteq V$ we define the \textit{complement} of $S$ as $S^c:=V\setminus S$, the \textit{external boundary} $\pap S$ as the subset of sites in $S^c$ that are adjacent to a site in $S$, \ie
\[
    \pap S := \{ v \in S^c ~|~ \exists \, w \in S ~:~ (v,w) \in E \},
\]
and $\nabla S$ as the subset of edges connecting the sites in $S$ with those in $\pap S$, \ie
\[
    \nabla S:=\{(v,w) \in E ~|~ v \in S, \, w \in \pap S\}.
\]
A (connected) \textit{odd cluster} $C \subseteq V $ is a subset of sites that satisfies both the following conditions:
\begin{enumerate}
    \item If an odd site $v \in \Vo$ belongs to $C$, then so do the four neighboring even sites, \ie $\pap \{v\} \subset C$;
    \item $C \cap \Ve$ is connected as a sub-graph of the graph $(\Ve, E^*)$, with  $E^*:=\{(v,w) \in \Ve \times \Ve ~|~ d(v,w)=2\}$, where $d(\cdot,\cdot)$ denotes the usual graph distance on $\L$.
\end{enumerate}
We denote by $\mathrm{C}_\oo(\L)$ the collection of the odd clusters on $\L$.

Consider the dual graph $\L'=(V',E')$ of the graph $\L$, which is a discrete torus of the same size. Given an odd cluster $C$, consider the edge set $\nabla C$ that disconnects $C$ from its complement $C^c$. We associate with $\nabla C$ the edge set $\g(C) \subset E'$ on the dual graph $\L'$ which consists of all the edges of $\L'$ orthogonal to edges in $\nabla C$. Such a set, to which we will refer as the \textit{contour} of the cluster $C$, consists of one or more piecewise linear closed curves and, by construction $|\g(C)|=|\nabla C|$. Leveraging this fact, we define the \textit{perimeter} $P(C)$ of the odd cluster $C$ as the total length of the contour $\gamma(C)$, \ie
\begin{equation}
\label{eq:pC}
    P(C):=|\g(C)|.
\end{equation}
As proved in~\cite{ZBvLN13}, the perimeter of the odd cluster $C$ satisfies the following identity:
\begin{equation}
\label{eq:identitypC}
    P(C)=4(|C \cap \Ve|-|C \cap \Vo|). 
\end{equation}
We call \textit{area} of an odd cluster $C$ the number of odd occupied sites it comprises. We say that an odd cluster is degenerate if it has area 0 and non--degenerate otherwise.

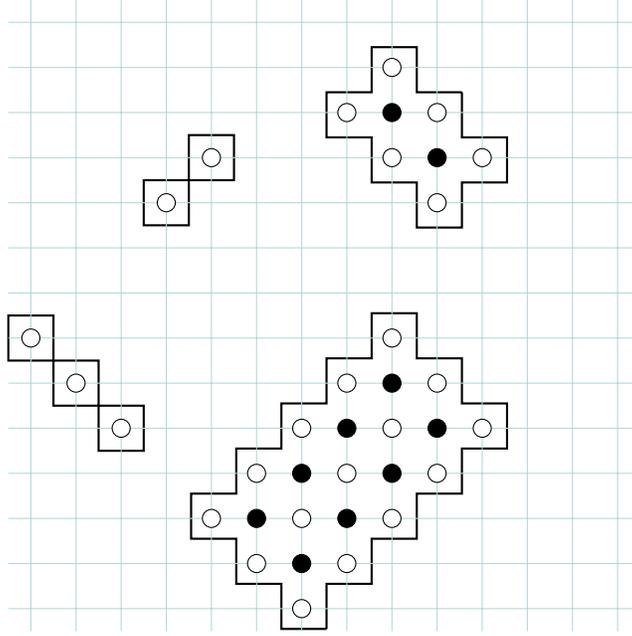
\begin{figure}[!h]
\begin{center}
\begin{tikzpicture}[scale=0.6]
    \newcommand\radius{0.2}
    \newcommand\listodd{(3,-1),(3,1),(2,0),(2,-2),(1,-1),(4,2),(4,0),(5,1),(5,7),(4,8)}
    \newcommand\listeven{(2,-3),(3,-2),(4,-1),(2,-1),(3,0),(4,1),(2,1),(3,2),(4,3),(1,0),(0,-1),(1,-2),(5,2),(6,1),(5,0),(5,8),(4,7),(4,9),(3,8),(6,7),(5,6),(0,7),(-1,6),(-2,1),(-3,2),(-4,3)}
    \draw[thick] (2.55,-3.45)--(2.55,-2.45)--(3.55,-2.45)--(3.55,-1.45)--(4.55,-1.45)--(4.55,-0.45)--(5.55,-0.45)--(5.55,0.55)--(6.55,0.55)--(6.55,1.55)--(5.55,1.55)--(5.55,2.55)--(4.55,2.55)--(4.55,3.55)--(3.55,3.55)--(3.55,2.55)--(2.55,2.55)--(2.55,1.55)--(1.55,1.55)--(1.55,0.55)--(0.55,0.55)--(0.55,-0.45)--(-0.45,-0.45)--(-0.45,-1.45)--(0.55,-1.45)--(0.55,-2.45)--(1.55,-2.45)--(1.55,-3.45)--(2.55,-3.45);
    \draw[thick] (5.55,8.45)--(4.55,8.45)--(4.55,9.45)--(3.55,9.45)--(3.55,8.45)--(2.55,8.45)--(2.55,7.45)--(3.55,7.45)--(3.55,6.45)--(4.55,6.45)--(4.55,5.45)--(5.55,5.45)--(5.55,6.45)--(6.55,6.45)--(6.55,7.45)--(5.55,7.45)--(5.55,8.45);
    \draw[thick] (-0.5,6.5) rectangle (-1.5,5.5);
    \draw[thick] (-0.5,6.5) rectangle (0.5,7.5);
    \draw[thick] (-3.5,1.5) rectangle (-2.5,2.5);
    \draw[thick] (-2.5,0.5) rectangle (-1.5,1.5);
    \draw[thick] (-4.5,2.5) rectangle (-3.5,3.5);
    \draw[help lines] (-4.5,-3.5) grid (9.5,10.5);
    \foreach \c in \listodd
        \draw[fill=black] \c circle(\radius);
    \foreach \c in \listeven
        \draw[fill=white] \c circle(\radius);
\end{tikzpicture}
 \caption{Example of four different rhombi, namely $\cR_{1,2}$, $\cR_{4,2}$, $\cR_{0,2}$ and $\cR_{1,0}$ (in clockwise order from the top-right corner).}
\label{fig:rombi}
\end{center}
\end{figure}

We introduce a mapping $\cO: \cX \to 2^V$ that associates to a given hard-core configuration $\s \in \cX$ the subset $\cO(\s) \subseteq V$ defined as
\begin{equation}
\label{eq:cC}
    \cO(\s):=\{ v \in V_\oo ~|~ \s(v) =1 \} \cup \{ v \in V_\ee ~|~ \s(v) =0 \}.
\end{equation}
In other words, $\cO(\s)$ is the subset comprising all the occupied odd sites and the empty even sites of the configuration $\s$. It is immediate to check that $\cO$ is an injective mapping and we will refer to the image $\cO(\s)$ of a configuration $\s$ as its \textit{odd region}. 

The odd region $\cO(\s)$ of a configuration $\s \in \cX$ can be partitioned into its connected components,
say $C_1(\s), \dots, C_m(\s) \in \mathrm{C}_\oo(\L)$, for some $m \in \N$, which are, by definition, odd clusters, that is
\begin{equation}
\label{eq:unionclusters}
    \cO(\s) = \bigsqcup_{i=1}^m C_i(\s).
\end{equation}
Using the partition~\eqref{eq:unionclusters} of the odd region $\cO(\s)$ into odd clusters, the definitions of contour and perimeter can be extended to the whole odd region in an obvious way, so that we can ultimately define the \textit{contour} $\g(\s)$ of a configuration $\s \in \cX$ as
\begin{equation}
\label{eg:gammas}
    \g(\s):=\bigsqcup_{i=1}^m \g(C_i(\s)),
\end{equation}
and its \textit{perimeter} $P(\s)$ as
\begin{equation}
\label{eq:ps}
    P(\s):=\sum_{i=1}^m P(C_i(\s)). 
\end{equation}
As shown in~\cite{ZBvLN13}, starting from~\eqref{eq:identitypC}, a double counting argument yields the following identity that relates the perimeter $P(\s)$ of a hard-core configuration $ \s \in \cX$ with its energy $H(\s)$ (recall \eqref{eq:pomodorina}-\eqref{eq:pC})
\begin{equation}
\label{eq:contour}
    P(\s)= 4 \, \D H(\s).
\end{equation}

Given a configuration $\sigma\in\cX$, we define the odd non-degenerate region $\cO^{nd}(\sigma)$ as a subset of $\cO(\sigma)$ containing only odd non-degenerate clusters. See~\cref{fig:oddregion} for an example of an odd region and an odd non-degenerate region. 

\begin{figure}[h!]
\begin{center}
\begin{tikzpicture}[scale=0.6]
    \newcommand\radius{0.2}
    \newcommand\listodd{(3,-1),(3,1),(2,0),(2,-2),(1,-1),(4,2),(4,0),(5,1),(5,7),(4,8),(5,9),(3,7)}
    \newcommand\listeven{(2,-3),(3,-2),(4,-1),(2,-1),(3,0),(4,1),(2,1),(3,2),(4,3),(1,0),(0,-1),(1,-2),(5,2),(6,1),(5,0),(5,8),(4,7),(4,9),(3,8),(6,7),(5,6),(0,7),(-1,6),(-2,1),(-3,2),(-4,3),(5,10),(6,9),(2,7),(3,6),(1,6),(-2,7),(0,5),(8,-1)}
    \draw[thick] (2.55,-3.45)--(2.55,-2.45)--(3.55,-2.45)--(3.55,-1.45)--(4.55,-1.45)--(4.55,-0.45)--(5.55,-0.45)--(5.55,0.55)--(6.55,0.55)--(6.55,1.55)--(5.55,1.55)--(5.55,2.55)--(4.55,2.55)--(4.55,3.55)--(3.55,3.55)--(3.55,2.55)--(2.55,2.55)--(2.55,1.55)--(1.55,1.55)--(1.55,0.55)--(0.55,0.55)--(0.55,-0.45)--(-0.45,-0.45)--(-0.45,-1.45)--(0.55,-1.45)--(0.55,-2.45)--(1.55,-2.45)--(1.55,-3.45)--(2.55,-3.45);
    \draw[thick] (3.5,9.5)--(3.5,8.5)--(2.5,8.5)--(2.5,7.5)--(1.5,7.5)--(1.5,6.5)--(2.5,6.5)--(2.5,5.5)--(3.5,5.5)--(3.5,6.5)--(4.5,6.5)--(4.5,5.5)--(5.5,5.5)--(5.5,6.5)--(6.5,6.5)--(6.5,7.5)--(5.5,7.5)--(5.5,8.5)--(6.5,8.5)--(6.5,9.5)--(5.5,9.5)--(5.5,10.5)--(4.5,10.5)--(4.5,9.5)--(3.5,9.5);
    \draw[thick,red] (-1.5,6.5) rectangle (-0.5,5.5);
    \draw[thick,red] (-0.5,6.5) rectangle (0.5,7.5);
    \draw[thick,red] (-3.5,1.5) rectangle (-2.5,2.5);
    \draw[thick,red] (-2.5,0.5) rectangle (-1.5,1.5);
    \draw[thick,red] (-4.5,2.5) rectangle (-3.5,3.5);
    \draw[thick,red] (0.5,5.5) rectangle (1.5,6.5);
    \draw[thick,red] (-2.5,6.5) rectangle (-1.5,7.5);
    \draw[thick,red] (-0.5,4.5) rectangle (0.5,5.5);
    \draw[thick,red] (7.5,-1.5) rectangle (8.5,-0.5);
    \draw[help lines] (-4.5,-3.5) grid (9.5,10.5);
    \foreach \c in \listodd
        \draw[fill=black] \c circle(\radius);
    \foreach \c in \listeven
        \draw[fill=white] \c circle(\radius);
\end{tikzpicture}
\caption{Example of a configuration $\sigma$, in which the contour of the non-degenerate (degenerate) odd clusters is highlighted in black (red, respectively). The contour $\gamma(\s)$ of the configuration $\s$ is of the odd region $\cO(\sigma)$ is the union of black lines (corresponding to $\cO^{nd}(\s)$) and red lines.}
\label{fig:oddregion}
\end{center}
\end{figure}
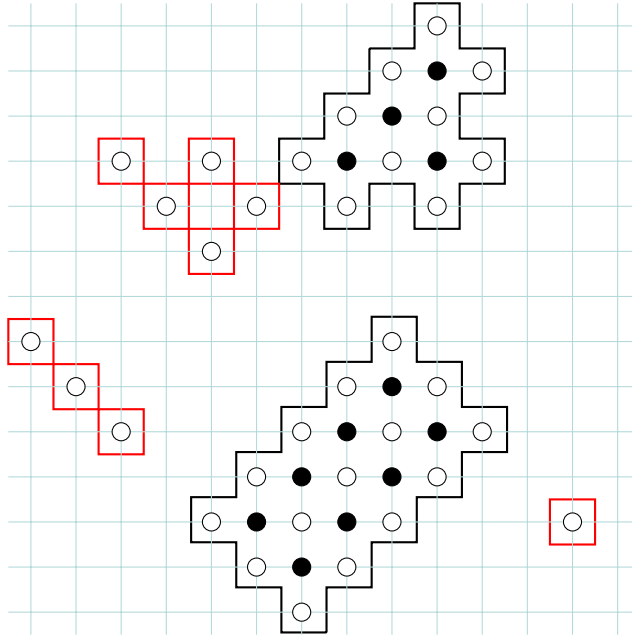

\subsection{Odd rhombi}
\label{sec:oddrhombi}
Given an odd site $\h=(\h_1,\h_2) \in V_\oo$ and two positive integers $\ell_1,\ell_2\leq L$, the \textit{odd rhombus} $\cR_{\ell_1,\ell_2}(\eta)$ with reference site $\h$ and lengths $\ell_1$ and $\ell_2$ is the odd cluster defined as
\begin{equation}
\label{eq:diamond}
    \cR_{\ell_1,\ell_2}(\h):= S_{\ell_1,\ell_2}(\h) \cup \pap S_{\ell_1,\ell_2}(\h),
\end{equation}
where $S_{\ell_1,\ell_2}(\h) \subseteq V_\oo$ is the subset of odd sites given by
\begin{align}\label{def:rectangle_odd_points}
    S_{\ell_1,\ell_2}(\h)
    &:= \bigcup_{0 \leq k \leq \ell_1-1, \, 0 \leq j \leq \ell_2-1} \{ (\h_1 + k + j, \h_2 + k - j )\} \nonumber \\
    &=  \{ v=(v_1,v_2) \in V ~|~ \exists \, k \in [[0,\ell_1]], \, j \in [[0,\ell_2]] ~:~ v_1 = \h_1 + k + j, \, v_2 = \h_2 + k - j\}.
\end{align}
In the latter definition, the coordinates sums and subtractions are taken modulo $L$. In the case $\ell_1\ell_2=0$, we can take $\h\in V_\ee$ and define the \emph{degenerate} rhombus $\cR_{\ell_1,\ell_2}(\eta)$ as the odd cluster 
\[
\cR_{\ell_1,\ell_2}(\h)=
\begin{cases}
\bigcup_{0 \leq j \leq \ell_2} \{ (\h_1 + j, \h_2 - j )\} &\hbox{if } \ell_1=0 \hbox{ and } \ell_2\neq0, \\
\bigcup_{0 \leq k \leq \ell_1} \{ (\h_1 + k, \h_2 - k )\} &\hbox{if } \ell_1\neq0 \hbox{ and } \ell_2=0, \\
(\eta_1,\eta_2) &\hbox{if } \ell_1=\ell_2=0.
\end{cases}
\]
Note that, in this case, $\cR_{\ell_1,\ell_2}(\h)$ is a subset of even sites. The area of $\cR_{\ell_1,\ell_2}(\h)$ is the cardinality of $S_{\ell_1,\ell_2}(\h)$ in the non-degenerate case, whereas in the degenerate case is equal to zero. 
Some example of rhombi and degenerate rhombi are shown in \cref{fig:rombi}. We observe that the non-degenerate rhombus $\cR_{\ell_1,\ell_2}$ has $\ell_1$ diagonals of length $\ell_2$ and $\ell_2$ diagonals of length $\ell_1$ in the opposite direction, which we will refer to as {\it complete diagonals}. We denote by $\mathrm{R}_\oo(\L) \subset \mathrm{C}_\oo(\L)$ the collection of all odd rhombi on $\L$ including the degenerate ones. For every odd cluster $C \in  \mathrm{C}_\oo(\L)$, we define the \textit{surrounding rhombus} $\cR(C)$ as the minimal rhombus (by inclusion) in $\mathrm{R}_\oo(\L) $ such that $C \subseteq \cR(C)$; see \cref{fig:cluster} for an example.

Most of the results for odd rhombi that will be proved are translation-invariant, the reason why we will often refer to the rhombus $\cR_{\ell_1,\ell_2}(\h)$ simply as $\cR_{\ell_1,\ell_2}$, without explicitly specifying the reference site $\h$. 
The next two lemmas, Lemmas \ref{lem:diff} and \ref{lem:complrombo},  concern properties of rhombi on a square $L \times L$ grid with periodic boundary conditions. Their proofs, being involved but not particularly insightful, are deferred to Appendix \ref{sec:app}. 

\begin{lem}[Set of sites winds around the torus]\label{lem:diff} 
Given $\eta=(\eta_1,\eta_2)\in V_\oo$ and two non-negative integers $\ell_1,\ell_2\leq L$ such that $\ell_1 \leq \ell_2$ and $\ell_1\geq L/2$, the following statements hold:
\begin{itemize}
\item[(i)] If $\ell_2\leq L-2$, then
\begin{align*}
\displaystyle\bigcup_{\substack{0\leq k\leq \ell_1 \\ \ell_2+1\leq j\leq L-1}}\{(\eta_1+k+j-1,\eta_2+k-j)\}\cup 
\bigcup_{\substack{\ell_1+1\leq k\leq L-1 \\ 0\leq j\leq \ell_2}} &\{(\eta_1+k+j-1,\eta_2+k-j)\} \\
&\subseteq \displaystyle\bigcup_{\substack{0\leq k\leq \ell_1 \\ 0\leq j\leq \ell_2}}\{(\eta_1+k+j-1,\eta_2+k-j)\}.
\end{align*}
\item[(ii)] If $\ell_2=L-1$, then 
\begin{align}\label{eq:stima1}
\displaystyle\bigcup_{\substack{\ell_1+1\leq k\leq L-1 \\ 0\leq j\leq L-1}}\{(\eta_1+k+j-1,\eta_2+k-j)\}
\subseteq \displaystyle\bigcup_{\substack{0\leq k\leq \ell_1 \\ 0\leq j\leq L-1}}\{(\eta_1+k+j-1,\eta_2+k-j)\}
\end{align}
and
\begin{align}\label{eq:stima2}
\displaystyle\bigcup_{\substack{L/2\leq k\leq \ell_1-1 \\ 1\leq j\leq L-2}}\{(\eta_1+k+j,\eta_2+k-j)\}
\subseteq \displaystyle\bigcup_{\substack{0\leq k\leq L/2-1 \\ 0\leq j\leq L-2}}\{(\eta_1+k+j,\eta_2+k-j)\}.
\end{align}
\end{itemize}
\end{lem}

For any subset of sites $A\subseteq V$, we define the \textit{complement} of $A$ as the complementary set of $A$ in $V$, i.e., as $V \setminus A$. 

\begin{lem}[Properties of rhombi]\label{lem:complrombo}
Given $\eta\in V_\oo$ and two non negative integers $\ell_1,\ell_2\leq L$, the following statements hold:
\begin{itemize}
    \item[(i)] If $\max\{\ell_1,\ell_2\}\leq L-2$ and $\min\{\ell_1,\ell_2\}\geq L/2$, then the complement of the rhombus $\cR_{\ell_1,\ell_2}(\eta)$ is a rhombus $R_{L-\ell_1-1,L-\ell_2-1}(\hat\eta)$ for some $\hat\eta\in V_\ee$.
    \item[(ii)] If $\max\{\ell_1,\ell_2\}=L-1$ and $\min\{\ell_1,\ell_2\}\geq L/2$, then the complement of the rhombus $\cR_{\ell_1,\ell_2}(\eta)$ is the disjoint union of $L-\min\{\ell_1,\ell_2\}$ odd sites.
    \item[(iii)] If $\max\{\ell_1,\ell_2\}= L$ and $\min\{\ell_1,\ell_2\}< L/2$, then the rhombus $\cR_{\ell_1,\ell_2}(\eta)$ contains $L\, \ell_1$ odd sites and $L(\ell_1+1)$ even sites.
    \item[(iv)] If $\max\{\ell_1,\ell_2\}= L$ and $\min\{\ell_1,\ell_2\}\geq L/2$, then the rhombus $\cR_{\ell_1,\ell_2}(\eta)$ coincides with $V$.
\end{itemize}
\end{lem}

These two lemmas will now be used to prove the next proposition, which gives a formula for the perimeter of a rhombus $\cR_{\ell_1,\ell_2}$. To this end, we will use the fact that a rhombus $\cR_{\ell_1,\ell_2}$ and its complement in $\Lambda$ have the same boundary for any $0\leq \ell_1,\ell_2\leq L$, and, in particular the same perimeter. In addition, we will say that a rhombus $\cR_{\ell_1,\ell_2}$ {\it winds vertically} (resp.\ {\it horizontally}) {\it around the torus} if there exists a set of $\frac{L}{2}$ odd sites $\h_1,...,\h_{L/2}$ in $\cR_{\ell_1,\ell_2}$ all on the same column (resp.\ row).
If the direction is not relevant, we will simply say that the rhombus {\it winds around the torus}. 

\begin{prop}[Formula for rhombus perimeter]\label{conj:rhombus}
Given a $L \times L$ toric grid graph $\Lambda$ and any sizes $0\leq \ell_1,\ell_2 \leq L$, the perimeter of the rhombus $\cR_{\ell_1,\ell_2}$ satisfies the following identity
\begin{equation}\label{eq:perimeter}
	P (\cR_{\ell_1,\ell_2}) =  4 \times
	\begin{cases}
		\ell_1 + \ell_2+1 & \text{ if } \min \{\ell_1,\ell_2\} < L/2 \text{ and } \max\{\ell_1,\ell_2\} < L,\\
		2L-(\ell_1+\ell_2+1) & \text{ if } \min \{\ell_1,\ell_2\} \geq L/2 \text{ and } \max\{\ell_1,\ell_2\} < L,\\
        L & \text{ if } \min \{\ell_1,\ell_2 \} < L/2 \text{ and } \max\{\ell_1,\ell_2\} =L ,\\
		0 & \text{ if } \min \{\ell_1,\ell_2 \} \geq L/2 \text{ and } \max\{\ell_1,\ell_2\} =L.
	\end{cases}
\end{equation}
\end{prop}

\begin{proof}
First of all, we identify which of the conditions in \eqref{eq:perimeter} imply that a rhombus winds around the torus. Consider the rhombus $\cR_{\ell_1,\ell_2}(\eta)$ with $\eta=(\eta_1,\eta_2) \in V_\oo$. Let $\sigma=(\sigma_1,\sigma_2) \in S_{\ell_1,\ell_2}(\eta)$ be such that $\sigma_2=\eta_2$ and $d(\eta_1,\sigma_1)$ is the maximal distance along that horizontal axis. Similarly, let $\xi=(\xi_1,\xi_2) \in S_{\ell_1,\ell_2}(\eta)$ be such that $d(\eta_2,\xi_2)$ is the maximal distance along the vertical axis. Recalling \eqref{def:rectangle_odd_points}, since $\sigma_2=\eta_2$, we have that $k=j$ for any $k,j=1,...,\ell_{\min}-1$, where $\ell_{\min}=\min \{\ell_1,\ell_2\}$. Thus, we obtain 
\begin{equation}
\sigma_1=\eta_1+k+j=\eta_1+2(\ell_{\min}-1),
\end{equation}
where the last equality follows from the fact that the maximal distance along the horizontal axis is precisely the distance between $\sigma_1$ and $\eta_1$. We note that if $d(\eta_1,\sigma_1) <L-2$ then the rhombus does not wind horizontally around the torus. Thus,
\begin{align}
d(\eta_1,\sigma_1)=2(\ell_{\min}-1)<L-2 \iff \ell_{\min}<\frac{L}{2}.
\end{align}
Now, let us consider the distance between $\eta_2$ and $\xi_2$. Let $\ell_{\max}=\max\{\ell_1,\ell_2\}$. In this case, if $d(\eta_2,\xi_2) \leq L-2$ then the rhombus does not wind vertically around the torus. Thus, we have
\begin{equation}
d(\eta_2,\xi_2)=\max_{k,j}|\eta_2-(\eta_2-k+j)|=\ell_{\max}-1 \leq L-2.
\end{equation}
We conclude that if $\ell_{\min}<\frac{L}{2}$ and $\ell_{\max}<L$, then the rhombus does not wind around the torus and the perimeter is the length of its external boundary. In view of \eqref{eq:identitypC}, the claim follows. Otherwise, there are three cases, which will be treated separately:
\begin{itemize}
    \item[(a)] $\ell_{\min} \geq \frac{L}{2}$ and $\ell_{\max} \leq L-2$;
    \item[(b)] $\ell_{\min} \geq \frac{L}{2}$ and $\ell_{\max} > L-2$;
    \item[(c)] $\ell_{\min}<\frac{L}{2}$ and $\ell_{\max}=L$.
\end{itemize}

(a) Consider the complement of the rhombus $\cR_{\ell_1,\ell_2}(\eta)$ for some $\eta$. By virtue of \cref{lem:complrombo}(i), we know that its complement in $V$ is a rhombus with side lengths $\tilde \ell_1=L-\ell_1-1$ and $\tilde \ell_2=L-\ell_1-1$. We claim that this complementary rhombus does not wind around the torus. By using the condition $\ell_{\min} \geq \frac{L}{2}$, we have that the maximal side length of the complementary rhombus is
\begin{equation}
    \max\{\tilde \ell_1, \tilde \ell_2\}=\max\{L-\ell_1-1,L-\ell_2-1\}=L-\ell_{\min}-1 \leq L-\frac{L}{2}-1 <L,
\end{equation}
that is, $\max\{\tilde \ell_1, \tilde \ell_2\} < L$. Moreover, the minimal side length is
\begin{equation}
    \min\{\tilde \ell_1, \tilde \ell_2\}=\min\{L-\ell_1-1,L-\ell_2-1\}=L-\ell_{\max}-1 \leq L-\ell_{\min}-1 \leq L-\frac{L}{2}-1,
\end{equation}
that is, $\min\{\tilde \ell_1, \tilde \ell_2\}<\frac{L}{2}$. Since the perimeter of the rhombus $R_{\ell_1,\ell_2}$ is the same as that of $R_{L-\ell_1-1,L-\ell_2-1}$, the claim follows from \eqref{eq:identitypC}.

(b) The claim follows from \cref{lem:complrombo}(ii)--(iii) and \eqref{eq:identitypC}.

(c) The claim follows from \cref{lem:complrombo}(iv) and \eqref{eq:identitypC}.
\end{proof}

\begin{figure}[h!]
\centering
    \begin{tikzpicture}[scale=0.55]
    \newcommand\radius{0.2}
    \newcommand\listodd{(-2,0),(-1,-1),(1,3),(2,4),(-1,3),(0,4),(2,2),(2,6),(1,7),(0,6),(3,5),(4,4),(1,-1),(2,0),(3,1),(4,2),(3,3),(5,3),(-2,2),(-3,3),(0,-2),(4,6),(3,7),(4,8)}
    \newcommand\listeven{(-1,0),(-2,1),(-3,0),(-2,-1),(0,-1),(-1,-2),(3,0),(4,1),(3,0),(2,3),(1,4),(0,3),(1,2),(3,4),(2,5),(-1,4),(-2,3),(-1,2),(0,5),(3,2),(2,1),(3,6),(2,7),(1,6),(1,8),(0,7),(-1,6),(4,5),(5,4),(4,3),(2,-1),(1,0),(1,-2),(5,2),(6,3),(-3,2),(-3,4),(-4,3),(0,-3),(5,6),(4,7),(3,8),(5,8),(4,9)}
    \newcommand\listantiknob{(5,7),(5,5),(1,5),(-1,5),(-2,4),(0,2),(1,1),(-1,1),(-3,1),(0,0),(2,8)}
    
    \draw[thick,red,fill=red,opacity=0.4] (4.5,6.5) rectangle (5.5,7.5);
    \draw[thick,red,fill=red,opacity=0.4] (4.5,4.5) rectangle (5.5,5.5);
    \draw[thick,red,fill=red,opacity=0.4] (0.5,4.5) rectangle (1.5,5.5);
    \draw[thick,red,fill=red,opacity=0.4] (-1.5,4.5) rectangle (-0.5,5.5);
    \draw[thick,red,fill=red,opacity=0.4] (-2.5,3.5) rectangle (-1.5,4.5);
    \draw[thick,red,fill=red,opacity=0.4] (-0.5,1.5) rectangle (0.5,2.5);
    \draw[thick,red,fill=red,opacity=0.4] (0.5,0.5) rectangle (1.5,1.5);
    \draw[thick,red,fill=red,opacity=0.4] (-1.5,0.5) rectangle (-0.5,1.5);
    \draw[thick,red,fill=red,opacity=0.4] (-3.5,0.5) rectangle (-2.5,1.5);
    \draw[thick,red,fill=red,opacity=0.4] (-0.5,-0.5) rectangle (0.5,0.5);
    \draw[thick,red,fill=red,opacity=0.4] (1.5,7.5) rectangle (2.5,8.5);
        
    \draw[help lines] (-4.5,-3.5) grid (9.5,10.5);
        
    \draw[thick] (-3.5,-0.5)--(-2.5,-0.5)--(-2.5,-1.5)--(-1.5,-1.5)--(-1.5,-2.5)--(-0.5,-2.5)--(-0.5,-3.5) -- (0.5,-3.5)--(0.5,-2.5)--(1.5,-2.5)--(1.5,-1.5)--(2.5,-1.5)--(2.5,-0.5)--(3.5,-0.5)--(3.5,0.5)--(4.5,0.5)--(4.5,1.5)--(5.5,1.5)--(5.5,2.5)--(6.5,2.5)--(6.5,3.5)--(5.5,3.5)--(5.5,4.5)--(4.5,4.5)--(4.5,5.5)--(5.5,5.5)--(5.5,6.5)--(4.5,6.5)--(4.5,7.5)--(5.5,7.5)--(5.5,8.5)--(4.5,8.5)--(4.5,9.5)--(3.5,9.5)--(3.5,8.5)--(2.5,8.5)--(2.5,7.5)--(1.5,7.5)--(1.5,8.5)--(0.5,8.5)--(0.5,7.5)--(-0.5,7.5)--(-0.5,6.5)--(-1.5,6.5)--(-1.5,5.5)--(-0.5,5.5)--(-0.5,4.5)--(-1.5,4.5)--(-1.5,3.5)--(-2.5,3.5)--(-2.5,4.5)--(-3.5,4.5)--(-3.5,3.5)--(-4.5,3.5)--(-4.5,2.5)--(-3.5,2.5)--(-3.5,1.5)--(-2.5,1.5)--(-2.5,0.5)--(-3.5,0.5)--(-3.5,-0.5);
    
    \draw[thick] (-0.5,0.5)--(-0.5,-0.5)--(0.5,-0.5)--(0.5,0.5)--(1.5,0.5)--(1.5,1.5)--(0.5,1.5)--(0.5,2.5)--(-0.5,2.5)--(-0.5,1.5)--(-1.5,1.5)--(-1.5,0.5)--(-0.5,0.5);
    \draw[thick] (0.5,4.5)--(1.5,4.5)--(1.5,5.5)--(0.5,5.5)--(0.5,4.5);

    \draw[thin,blue] (1.75,-1.25)--(1.25,-1.75)--(-3.75,3.25)--(-3.25,3.75)--(1.75,-1.25);
    \draw[thin,blue] (2.75,-0.25)--(2.25,-0.75)--(-1.75,3.25)--(-1.25,3.75)--(2.75,-0.25);
    \draw[thin,blue] (4.75,1.75)--(4.25,1.25)--(-0.75,6.25)--(-0.25,6.75)--(4.75,1.75);

    \foreach \c in \listodd
        \draw[fill=black] \c circle(\radius);
    \foreach \c in \listeven
        \draw[fill=white] \c circle(\radius);
    
\end{tikzpicture}
\hspace{1cm}
\begin{tikzpicture}[scale=0.55]
    \newcommand\radius{0.2}
    \newcommand\listodd{(-2,0),(-1,-1),(1,3),(2,4),(-1,3),(0,4),(2,2),(2,6),(1,7),(0,6),(3,5),(4,4),(1,-1),(2,0),(3,1),(4,2),(3,3),(5,3),(-2,2),(-3,3),(0,-2),(4,6),(3,7),(4,8)}
    \newcommand\listeven{(-1,0),(-2,1),(-3,0),(-2,-1),(0,-1),(-1,-2),(3,0),(4,1),(3,0),(2,3),(1,4),(0,3),(1,2),(3,4),(2,5),(-1,4),(-2,3),(-1,2),(0,5),(3,2),(2,1),(3,6),(2,7),(1,6),(1,8),(0,7),(-1,6),(4,5),(5,4),(4,3),(2,-1),(1,0),(1,-2),(5,2),(6,3),(-3,2),(-3,4),(-4,3),(0,-3),(5,6),(4,7),(3,8),(5,8),(4,9)}
    
    \draw[help lines] (-4.5,-3.5) grid (9.5,10.5);
        
    \draw[thick] (-3.5,-0.5)--(-2.5,-0.5)--(-2.5,-1.5)--(-1.5,-1.5)--(-1.5,-2.5)--(-0.5,-2.5)--(-0.5,-3.5) -- (0.5,-3.5)--(0.5,-2.5)--(1.5,-2.5)--(1.5,-1.5)--(2.5,-1.5)--(2.5,-0.5)--(3.5,-0.5)--(3.5,0.5)--(4.5,0.5)--(4.5,1.5)--(5.5,1.5)--(5.5,2.5)--(6.5,2.5)--(6.5,3.5)--(5.5,3.5)--(5.5,4.5)--(4.5,4.5)--(4.5,5.5)--(5.5,5.5)--(5.5,6.5)--(4.5,6.5)--(4.5,7.5)--(5.5,7.5)--(5.5,8.5)--(4.5,8.5)--(4.5,9.5)--(3.5,9.5)--(3.5,8.5)--(2.5,8.5)--(2.5,7.5)--(1.5,7.5)--(1.5,8.5)--(0.5,8.5)--(0.5,7.5)--(-0.5,7.5)--(-0.5,6.5)--(-1.5,6.5)--(-1.5,5.5)--(-0.5,5.5)--(-0.5,4.5)--(-1.5,4.5)--(-1.5,3.5)--(-2.5,3.5)--(-2.5,4.5)--(-3.5,4.5)--(-3.5,3.5)--(-4.5,3.5)--(-4.5,2.5)--(-3.5,2.5)--(-3.5,1.5)--(-2.5,1.5)--(-2.5,0.5)--(-3.5,0.5)--(-3.5,-0.5);
    
    \draw[thick] (-0.5,0.5)--(-0.5,-0.5)--(0.5,-0.5)--(0.5,0.5)--(1.5,0.5)--(1.5,1.5)--(0.5,1.5)--(0.5,2.5)--(-0.5,2.5)--(-0.5,1.5)--(-1.5,1.5)--(-1.5,0.5)--(-0.5,0.5);
    \draw[thick] (0.5,4.5)--(1.5,4.5)--(1.5,5.5)--(0.5,5.5)--(0.5,4.5);

    \draw[red,thick] (2.5,10.5)--(2.5,9.5)--(1.5,9.5)--(1.5,8.5)--(0.5,8.5)--(0.5,7.5)--(-0.5,7.5)--(-0.5,6.5)--(-1.5,6.5)--(-1.5,5.5)--(-2.5,5.5)--(-2.5,4.5)--(-3.5,4.5)--(-3.5,3.5)--(-4.5,3.5)--(-4.5,2.5);

    \draw[thick,red] (9.5,2.5)--(8.5,2.5)--(8.5,1.5)--(9.5,1.5);

    \draw[thin,blue] (0.75,-2.25)--(0.25,-2.75)--(-2.75,0.25)--(-2.25,0.75)--(0.75,-2.25);
    \draw[thin,blue] (3.75,0.75)--(3.25,0.25)--(-0.75,4.25)--(-0.25,4.75)--(3.75,0.75);
    \draw[thin,dartmouthgreen] (5.25,2.25)--(5.75,2.75)--(0.75,7.75)--(0.25,7.25)--(5.25,2.25);
    \draw[thin,blue] (4.75,5.75)--(4.25,5.25)--(2.25,7.25)--(2.75,7.75)--(4.75,5.75);
    \draw[thin,blue] (4.75,7.75)--(4.25,7.25)--(3.25,8.25)--(3.75,8.75)--(4.75,7.75);

    \draw[thick,red] (-4.5,1.5)--(-4.5,0.5)--(-3.5,0.5)--(-3.5,-0.5)--(-2.5,-0.5)--(-2.5,-1.5)--(-1.5,-1.5)--(-1.5,-2.5)--(-0.5,-2.5)--(-0.5,-3.5)--(0.5,-3.5)--(0.5,-2.5)--(1.5,-2.5)--(1.5,-1.5)--(2.5,-1.5)--(2.5,-0.5)--(3.5,-0.5)--(3.5,0.5)--(4.5,0.5)--(4.5,1.5)--(5.5,1.5)--(5.5,2.5)--(6.5,2.5)--(6.5,3.5)--(7.5,3.5)--(7.5,4.5)--(8.5,4.5)--(8.5,5.5)--(7.5,5.5)--(7.5,6.5)--(6.5,6.5)--(6.5,7.5)--(5.5,7.5)--(5.5,8.5)--(4.5,8.5)--(4.5,9.5)--(3.5,9.5)--(3.5,10.5)--(2.5,10.5);
        
    \foreach \c in \listodd
        \draw[fill=black] \c circle(\radius);
    \foreach \c in \listeven
        \draw[fill=white] \c circle(\radius);
    
\end{tikzpicture}
\caption{Example of a odd cluster $C$ (on the left) and its surrounding rhombus $\cR(C)=\cR_{8,5}$ in red (on the right). On the left, the red squares contain the antiknobs and the decreasing broken diagonals are highlighted with blue rectangles. On the right, we highlight the decreasing shorter (resp.\ complete) diagonals with blue (resp.\ green) rectangles.}
\label{fig:cluster}
\end{figure}
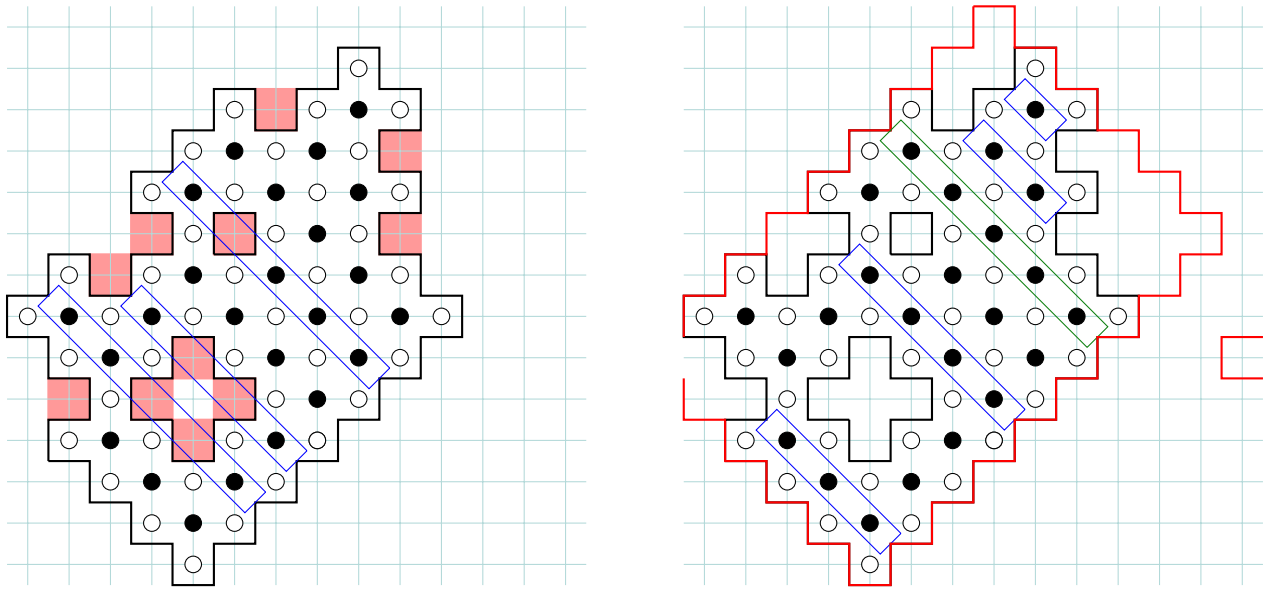
\FloatBarrier

We say that an odd cluster $C$ is \textit{monotone} when its perimeter coincides with that of its surrounding rhombus $\cR(C)$ if it does not wind around the torus, i.e., $P(C)=P(\cR(C))$. Otherwise, we say that an odd cluster $C$ is monotone when its perimeter coincides with that of a bridge, i.e., $P(C)=4L$. Note that it immediately follows that a monotone odd cluster $C$ has no \textit{holes}, i.e., empty odd sites with the four even neighboring sites belonging to $C$. An empty odd site $\h\notin C$ is an \textit{antiknob for the cluster} $C$ if it has at least three neighboring even empty sites that belong to $C$. \cref{fig:cluster} (left) highlights in red the antiknobs of a hard-core configuration.

Given an odd cluster $C\in C_\oo(\Lambda)$ and an integer $k\geq1$, we say that $C$ displays an \textit{increasing} (resp.\ \textit{decreasing}) \textit{diagonal broken in $k$ sites} if there exist a sequence of sites $z_i=(x_i,y_i)\in V_\oo\setminus C$, $i=1,...,k$, such that 
\begin{itemize}
    \item $x_{i+1}=x_i+1$ and $y_{i+1}=y_i+1$ (resp.\ $x_{i+1}=x_i+1$ and $y_{i+1}=y_i-1$) for any $i=1,...,k-1$, and 
    \item the two odd sites $(x_1-1,y_1-1)$ and $(x_k+1,y_k+1)$ (resp.\ $(x_1-1,y_1+1)$ and $(x_k+1,y_k-1)$) belong to the cluster $C$. 
\end{itemize}

By construction of a broken diagonal, the two sites $z_1$ and $z_k$ are always antiknobs. If it does not matter if an increasing or decreasing diagonal is broken, we simply say that a diagonal is broken. Broken diagonals are visualized in blue in~\cref{fig:cluster} (left).
Given an odd cluster $C\in C_\oo(\Lambda)$ and an integer $k\geq1$, we say that $C$ displays an \textit{increasing} (resp.\ \textit{decreasing}) \textit{shorter diagonal} lacking in $k$ sites if there exist a sequence of sites $z_i=(x_i,y_i)\in (V_\oo\cap\cR(C))\setminus C$, $i=1,...,k$, such that 
\begin{itemize}
    \item $x_{i+1}=x_i+1$ and $y_{i+1}=y_i+1$ (resp.\ $x_{i+1}=x_i+1$ and $y_{i+1}=y_i-1$) for any $i=1,...,k-1$, and 
    \item the two odd sites $(x_1-1,y_1-1)$ and $(x_k+1,y_k+1)$ (resp.\ $(x_1-1,y_1+1)$ and $(x_k+1,y_k-1)$) do not belong to $\cR(C)$.
\end{itemize}
\cref{fig:cluster} (right) highlights the shorter diagonals in blue.

\subsection{Expanding an odd cluster: The filling algorithms}
\label{sec:filling}
We now describe an iterative procedure that builds a path $\omega$ in $\cX$ from a configuration $\sigma$ with a unique odd cluster to another configuration $\s'$ that (i) displays a rhombus, and (ii) whose energy $H(\s')$ is equal to or lower than $H(\sigma)$. 

The path $\omega$ can be described as the concatenation of two paths, each obtained by means of a specific \textit{filling algorithm}. The reason behind this name is that, along the generated paths, any incomplete diagonal in the odd cluster of the starting configuration is gradually filled by adding particles in odd sites until a rhombus is obtained. 

The two paths can be intuitively described as follows. The first path, denoted as $\tilde\omega$, starts from a configuration with at least one broken diagonal and, by filling one by one all broken diagonals in lexicographic order, arrives at a configuration with no broken diagonal. Each broken diagonal is progressively filled by removing a particle in the even site at distance 1 from a antiknob that lies on that diagonal and adding a particle in the odd site where the antiknob is. The antiknobs on the same diagonal are processed in lexicographic order. The second path, denoted as $\bar\omega$, starts from a configuration with no broken diagonal (such as any ending configuration of the path $\tilde\omega$) and arrives at a configuration displaying an odd rhombus. The construction of this second path is similar to that of the first path, but in this case the particles are added in odd sites to fill all the shorter diagonals. 

The filling algorithms generating the two paths $\tilde\o$ and $\bar\o$ are designed in such a way that the maximum energy along the resulting path $\omega=\tilde\o\cup\bar\o$, i.e., $\Phi_{\omega}$, is never larger than $H(\sigma)+1$. More specifically, the perimeter of the odd cluster either decreases or does not change along $\tilde\omega$, whereas it can increase and sequentially decrease by the same quantity along $\bar\omega$. \cref{lem:expansion}, whose proof is postponed to \cref{sec:oddproof}, specifies the requirement for the starting configuration and summarized the properties of the path generated by the filling algorithm. 

To formally define these two algorithms, we introduce the following notation. Given two configurations $\s,\s' \in \cX$ and a subset of sites $W \subset \Lambda$, we write $\s_{|W} = \s'_{|W}$ if $\s(v) = \s'(v)$ for every $v \in W$. Given a configuration $\s \in \cX$, we let
\begin{itemize}
    \item $\s^{(v,0)}$ be the configuration $\s' \in \cX$ such that $\s'_{|V \setminus \{v\}} = \s_{|V \setminus \{v\}}$ and $\s'(v)=0$; and
    \item $\s^{(v,1)}$ be the configuration $\s'$ such that $\s'_{|V \setminus \{v\}} = \s_{|V \setminus \{v\}}$ and $\s'(v)=1$. 
\end{itemize}    
In general, $\s^{(v,1)}$ might not be a hard-core configuration in $\cX$, since $\sigma$ may already have a particle residing in one of the four neighboring sites of $v$. 

\cref{alg:fill1} (resp.\ \cref{alg:fill2}) provide the detailed pseudocode for the filling algorithm that yields $\tilde\omega$ (resp.\ $\bar\omega$). 

\begin{algorithm}[!h]
\caption{Filling algorithm to build path $\tilde\omega$} \label{alg:fill1}
\KwData{a configuration $\s\in\cX$ consisting of a unique odd cluster $C$ with $m\geq1$ broken diagonals such that the $j$-th is broken in $k_j$ sites for $j=1,...,m$}
\KwResult{$\tilde\omega:\s\to\tilde\s$, with the configuration $\tilde\s\in\cX$ consisting of a unique odd cluster with no broken diagonals}

$\sigma_0=\sigma$\;
\For{$j=1,...,m$}{
$\sigma_{j}=\sigma_{j-1}$\;
\For{$i=1,...,k_j$}{

Consider the $i$-th antiknob (in lexicographic order) of the $j$-th broken diagonal in $C$ and denote it by $x_{j,i}\in V_\oo$

\eIf{$x_{j,i}$ has a neighboring occupied site $\tilde x_{j,i}\in V_\ee$}{
$\tilde\sigma_{j,i}=\sigma_j^{(\tilde x_{j,i},0)}$;

$\sigma_{j,i}=\tilde\sigma_{j,i}^{(x_{j,i},1)}$;
}
{
$\tilde\sigma_{j,i}=\sigma_j^{(x_{j,i},1)}$;

$\sigma_{j,i}=\tilde\sigma_{j,i}$;
}
$\sigma_{j+1}=\sigma_{j,k_j}$;

$\tilde\omega_{j,i}=(\sigma_j,\tilde\sigma_{j,1},\sigma_{j,1},...,\tilde\sigma_{j,i},\sigma_{j,i})$;
}
$\tilde\omega_j=\tilde\omega_{j,k_j}$;

$\sigma_{j+1}=\sigma_{j,k_j}$;
}
$\tilde\sigma=\sigma_m$;

$\tilde\omega$ is the concatenation of the paths $\tilde\omega_1$,..., $\tilde\omega_m$
\end{algorithm}

\begin{algorithm}[!h]
\caption{Filling algorithm to build path $\bar\omega$} \label{alg:fill2}
\KwData{a configuration $\s\in\cX$ consists of a unique odd cluster $C$ with no broken diagonal and $m\geq0$ increasing shorter diagonals. The quantity $k_j$ is the difference between the length of the shorter diagonal and the corresponding one of the surrounding rhombus for $j=1,...,m$}
\KwResult{$\bar\omega:\s\to\bar\s$, with the configuration $\bar\s\in\cX$ having a unique odd cluster, which is a rhombus}

\eIf{$m=0$}{
$\bar\sigma=\sigma$ and $\bar\omega$ is trivial;
}
{
$\sigma_0=\sigma$\;
\For{$j=1,...,m$}{
$\sigma_{j}=\sigma_{j-1}$\;
\For{$i=1,...,k_j$}{

Consider the $i$-th empty odd site (in lexicographic order) not belonging to the $j$-th shorter diagonal in $C$ but in the corresponding complete diagonal of the surrounding rhombus and denote it by $x_{j,i}\in V_{\oo}$

\eIf{$x_{j,i}$ has a neighboring occupied site $\bar x_{j,i}\in V_\ee$}{
$\bar\sigma_{j,i}=\sigma_j^{(\bar x_{j,i},0)}$;

$\sigma_{j,i}=\bar\sigma_{j,1}^{(x_{j,i},1)}$;
}
{
$\bar\sigma_{j,i}=\sigma_j^{(x_{j,i},1)}$;

$\sigma_{j,i}=\bar\sigma_{j,i}$;
}
$\sigma_{j+1}=\sigma_{j,k_j}$;

$\bar\omega_{j,i}=(\sigma_j,\bar\sigma_{j,1},\sigma_{j,1},...,\bar\sigma_{j,i},\sigma_{j,i})$;
}
$\bar\omega_j=\bar\omega_{j,k_j}$;

$\sigma_{j+1}=\sigma_{j,k_j}$;
}
$\bar\sigma=\sigma_m$;

Obtain $\bar\omega$ as the concatenation of the paths $\bar\omega_1$,..., $\bar\omega_m$
}
\end{algorithm}

\begin{prop}[Odd cluster expansion via filling algorithms] \label{lem:expansion}
Let $\s,\s' \in \cX$ be two hard-core configurations on $\L$, $\s \neq \s'$, and $\cR$ a rhombus such that
\begin{itemize}
    \item[(i)] There exists a connected odd cluster $C \subseteq \cO(\s)$ such that $\cR(C) = \cR$;
	\item[(ii)] $\s_{|\L \setminus \cR} = \s'_{| \L \setminus \cR}$;
	\item[(iii)] $\s'_{|\cR} = \oo_{|\cR}$.
\end{itemize}
Then, there exists a path $\o: \s \to \s'$ such that $\Phi_\o - H(\s) \leq 1$. In addition, if $C$ has at least one broken diagonal then $P(\sigma)>P(\sigma')$, otherwise $P(\sigma)=P(\sigma')$. 
\end{prop}

We note that conditions $(i),(ii)$, and $(iii)$ mean that there is a unique odd cluster in $\sigma$ different from a rhombus, i.e., there exists at least one broken or shorter diagonal. 

Thanks to \cref{lem:expansion}, we are able to characterize the configurations having minimal perimeter for a fixed number of occupied odd sites. This finding is formalized in the following two results, \cref{lem:minP} and \cref{cor:mortadella}, whose proofs are deferred to \cref{sec:oddproof}.

To state the precise results, we first introduce the notion of \textit{bars} as follows. We define a vertical (resp.\ horizontal) bar $B$ of length $k$ as the union of the particles arranged in odd sites $x_1,...,x_k$ belonging to the same column (resp.\ row) such that $d(x_i,x_{i+1})=2$ for any $i=1,...,k-1$. In the case of $k=1$, we will refer to it as \textit{protuberance}. If on the same column (resp.\ row) there are $m$ disjoint vertical (resp.\ horizontal) bars, each of them of length $k_i$, we say that the total length of the bars is $k=k_1+...+k_m$. Similarly, we can define a \textit{diagonal bar} and note that it can correspond to a shorter or complete diagonal. If it does not matter if the bar is vertical, horizontal, or diagonal, we simply refer to it as a \textit{bar}. Finally, we will say that a bar $B$ is attached to a cluster $C$ when all the particles belonging to $B$ are at distance two from $C$. Note that in~\cref{fig:cluster} (right) the shorter diagonals of lengths three and one are diagonal bars. See~\cref{fig:cluster} (left) for examples of vertical bars.

\begin{lem}\label{lmm:pancetta}
For any $n$ positive integer there exist two positive integers $s$ and $k$, with $0\leq k<s$, such that either (i) $n=s(s-1)+k$ or (ii) $n=s^2+k$.
\end{lem}
\begin{proof}
See the first part of \cite[Lemma 6.17]{Cirillo2013}.
\end{proof}

\begin{prop}[Perimeter-Minimal rhombi]\label{lem:minP} 
Consider $n\leq L(L-2)$ and let $s,k$ be the unique integers as in \cref{lmm:pancetta}. The set of odd clusters with area $n$ that have minimal perimeter contains either a rhombus $\cR_{s,s-1}$ or $\cR_{s-1,s}$ with a bar of length $k$ attached to one of its longest sides if $n=s(s-1)+k$ and a rhombus $\cR_{s,s}$ with a bar of length $k$ attached to one of its sides if $n=s^2+k$.
\end{prop}

\begin{cor}[Minimal perimeter]\label{cor:mortadella}
Consider $n\leq L(L-2)$ and let $s,k$ be the unique integers as in \cref{lmm:pancetta}. The perimeter $P$ of an odd cluster with area $n$ satisfies the following inequalities:
\[
    \Bigl(\frac{P}{4}-1\Bigr)^2\geq
    \begin{cases}
         4n &\hbox{ if } s<L/2, \\
         2(L^2-2n) &\hbox{ if } L/2\leq s<L.
    \end{cases} 
\]
In addition, for all $s<L$ we have that
\begin{equation}\label{eq:diavola}
    P\geq 4(2\sqrt{n}+1)
\end{equation}
and for $s<\frac{L}{2}$ the equality holds if and only if the odd cluster is the rhombus $\cR_{s,s}$.
\end{cor}

Lastly, in the next lemma, we derive an isoperimetric inequality assuming the total number of odd and even occupied sites is fixed. To this end, we first define the \textit{real area} of a configuration $\sigma$ as
\begin{equation}\label{eq:prosciuttoefunghi}
\tilde n(\sigma)= o(\sigma)+\tilde e(\sigma),
\end{equation}
where $o(\s)$ (resp.\ $\tilde e(\sigma)$) denotes the number of occupied odd (resp.\ empty even) sites of the configuration $\sigma$. 

\begin{lem}[Perimeter-Minimal rhombi with fixed real area]\label{prop:maialona} 
Given $1\leq \ell<\frac{L}{2}$, the unique odd cluster with real area $\tilde n=2\ell^2+2\ell+1$ and minimal perimeter is the rhombus $\cR_{\ell,\ell}$. In particular, for $\tilde n= \frac{L^2}{2}-L+1$ this rhombus is $\cR_{\frac{L}{2}-1,\frac{L}{2}-1}$. 
\end{lem}
We use this lemma to characterize the critical configurations having an odd cluster with a rhomboidal shape that does not wind around the torus. Indeed, we identify the shape of the protocritical configurations $\s$ with minimal perimeter and fixed real area $\frac{L^2}{2}-L+1$ such that $\cR(\cO^{nd}(\s))$ does not wind around the torus, and we show that if the trajectory visits another type of configuration with such a real area, then the corresponding path would be not optimal. The proof is given in \cref{sec:oddproof}. 

\section{Essential saddles: Proof of the main theorem}
\label{sec:proofthm}
In this section, we first formally introduce in~\cref{sec:preliminariesforproof} the six sets appearing in the statement of~\cref{thm:saddles} and then prove the same theorem in in~\cref{sec:actualproof} by showing that the elements of those six sets are all the essential saddles for the transition from $\ee$ to $\oo$.

\subsection{Preliminaries}
\label{sec:preliminariesforproof}

We say that a configuration $\sigma\in\cX$ has a \textit{odd (resp.\ even) vertical bridge} if there exists a column in which configuration $\sigma$ perfectly agrees with $\oo$ (resp.\ $\ee$). We define \textit{odd (resp.\ even) horizontal bridge} in an analogous way and we say that a configuration $\sigma\in\cX$ has an \textit{odd (resp.\ even) cross} if it has both vertical and horizontal odd (resp.\ even) bridges (see~\cref{fig:bridge}). In addition, we say that a configuration displays an \textit{odd (resp.\ even) vertical $m$-uple bridge}, with $m\geq2$, if there exist $m$ contiguous columns in which the configuration perfectly agrees with $\oo$ (resp.\ $\ee$). Similarly, we can define an \textit{odd (resp.\ even) vertical $m$-uple bridge} (see~\cref{fig:multiplebridges}). We refer to \cite{NZB15} for more details.

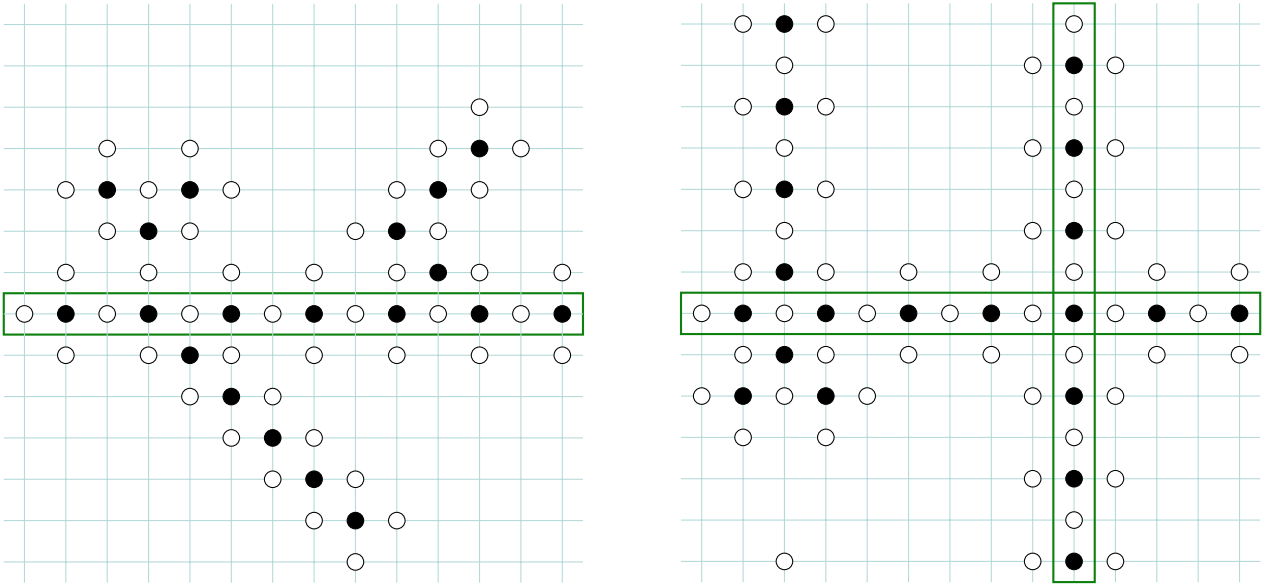
\begin{figure}[h!]
\centering
\begin{tikzpicture}[scale=0.55]
    \newcommand\radius{0.2}
    \newcommand\listodd{(-3,3),(-1,3),(1,3),(3,3),(5,3),(7,3),(9,3),(7,7),(6,6),(5,5),(6,4),(-2,6),(0,6),(-1,5),(0,2),(1,1),(2,0),(3,-1),(4,-2)}
    \newcommand\listeven{(-4,3),(-2,3),(0,3),(4,3),(2,3),(6,3),(8,3),(7,6),(7,8),(6,7),(8,7),(5,6),(6,5),(4,5),(5,4),(7,4),(-3,6),(-1,6),(-2,5),(-2,7),(1,6),(0,5),(0,7),(-1,4),(-1,2),(-3,2),(-3,4),(1,2),(1,4),(3,2),(3,4),(5,2),(7,2),(9,2),(0,1),(2,1),(1,0),(3,0),(2,-1),(4,-1),(3,-2),(5,-2),(4,-3),(9,4)}

    \draw[thick,dartmouthgreen] (-4.5,2.5) rectangle (9.5,3.5);

    \draw[help lines] (-4.5,-3.5) grid (9.5,10.5);

    \foreach \c in \listodd
        \draw[fill=black] \c circle(\radius);
    \foreach \c in \listeven
        \draw[fill=white] \c circle(\radius);
\end{tikzpicture}
\hspace{1cm}
    \begin{tikzpicture}[scale=0.55]
    \newcommand\radius{0.2}
    \newcommand\listodd{(-3,3),(-1,3),(1,3),(3,3),(5,3),(7,3),(9,3),(5,-3),(5,-1),(5,1),(5,5),(5,7),(5,9),(-2,2),(-1,1),(-3,1),(-2,4),(-2,6),(-2,8),(-2,10)}
    \newcommand\listeven{(-4,3),(-2,3),(0,3),(4,3),(2,3),(6,3),(8,3),(-3,4),(-1,4),(1,4),(3,4),(5,4),(7,4),(9,4),(-3,2),(-1,2),(1,2),(3,2),(5,2),(7,2),(9,2),(4,-3),(4,-1),(4,1),(4,5),(4,7),(4,9),(6,-3),(6,-1),(6,1),(6,5),(6,7),(6,9),(5,-2),(5,0),(5,6),(5,8),(5,10),(-4,1),(-3,0),(-2,1),(-1,0),(0,1),(-2,5),(-3,6),(-1,6),(-2,7),(-3,8),(-1,8),(-2,9),(-3,10),(-1,10),(-2,-3)}
    
    \draw[help lines] (-4.5,-3.5) grid (9.5,10.5);

    \draw[thick,dartmouthgreen] (-4.5,2.5) rectangle (9.5,3.5);
    \draw[thick,dartmouthgreen] (4.5,-3.5) rectangle (5.5,10.5);

    \foreach \c in \listodd
        \draw[fill=black] \c circle(\radius);
    \foreach \c in \listeven
        \draw[fill=white] \c circle(\radius);
\end{tikzpicture}
\caption{Examples of configurations displaying an odd horizontal bridge (on the left) and an odd cross (on the right).} 
\label{fig:bridge}
\end{figure}

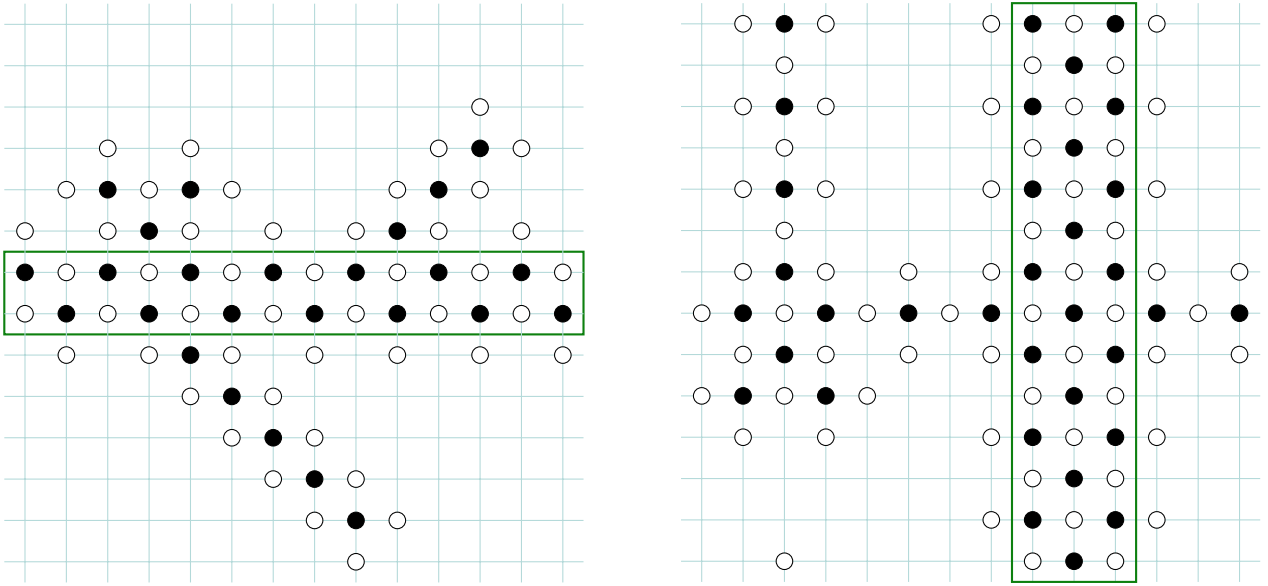
\begin{figure}[!h]
\centering
\begin{tikzpicture}[scale=0.55]
    \newcommand\radius{0.2}
    \newcommand\listodd{(-3,3),(-1,3),(1,3),(3,3),(5,3),(7,3),(9,3),(7,7),(6,6),(5,5),(6,4),(-2,6),(0,6),(-1,5),(0,2),(1,1),(2,0),(3,-1),(4,-2),(-4,4),(-2,4),(0,4),(2,4),(4,4),(8,4)}
    \newcommand\listeven{(-4,3),(-2,3),(0,3),(4,3),(2,3),(6,3),(8,3),(7,6),(7,8),(6,7),(8,7),(5,6),(6,5),(4,5),(5,4),(7,4),(-3,6),(-1,6),(-2,5),(-2,7),(1,6),(0,5),(0,7),(-1,4),(-1,2),(-3,2),(-3,4),(1,2),(1,4),(3,2),(3,4),(5,2),(7,2),(9,2),(0,1),(2,1),(1,0),(3,0),(2,-1),(4,-1),(3,-2),(5,-2),(4,-3),(9,4),(8,5),(2,5),(-4,5)}

    \draw[thick,dartmouthgreen] (-4.5,2.5) rectangle (9.5,4.5);

    \draw[help lines] (-4.5,-3.5) grid (9.5,10.5);

    \foreach \c in \listodd
        \draw[fill=black] \c circle(\radius);
    \foreach \c in \listeven
        \draw[fill=white] \c circle(\radius);
\end{tikzpicture}
\hspace{1cm}
    \begin{tikzpicture}[scale=0.55]
    \newcommand\radius{0.2}
    \newcommand\listodd{(-3,3),(-1,3),(1,3),(3,3),(5,3),(7,3),(9,3),(5,-3),(5,-1),(5,1),(5,5),(5,7),(5,9),(-2,2),(-1,1),(-3,1),(-2,4),(-2,6),(-2,8),(-2,10),(4,-2),(6,-2),(4,0),(6,0),(4,2),(6,2),(4,4),(6,4),(4,6),(6,6),(4,8),(6,8),(4,10),(6,10)}
    \newcommand\listeven{(-4,3),(-2,3),(0,3),(4,3),(2,3),(6,3),(8,3),(-3,4),(-1,4),(1,4),(3,4),(5,4),(7,4),(9,4),(-3,2),(-1,2),(1,2),(3,2),(5,2),(7,2),(9,2),(4,-3),(4,-1),(4,1),(4,5),(4,7),(4,9),(6,-3),(6,-1),(6,1),(6,5),(6,7),(6,9),(5,-2),(5,0),(5,6),(5,8),(5,10),(-4,1),(-3,0),(-2,1),(-1,0),(0,1),(-2,5),(-3,6),(-1,6),(-2,7),(-3,8),(-1,8),(-2,9),(-3,10),(-1,10),(-2,-3),(3,10),(7,10),(3,8),(7,8),(3,6),(7,6),(3,0),(7,0),(3,-2),(7,-2)}
    
    \draw[help lines] (-4.5,-3.5) grid (9.5,10.5);

    \draw[thick,dartmouthgreen] (3.5,-3.5) rectangle (6.5,10.5);

    \foreach \c in \listodd
        \draw[fill=black] \c circle(\radius);
    \foreach \c in \listeven
        \draw[fill=white] \c circle(\radius);
\end{tikzpicture}
\caption{Examples of configurations displaying an odd horizontal double (2-uple) bridge (on the left) and an odd vertical triple (3-uple) bridge (on the right).}
\label{fig:multiplebridges}
\end{figure}

The next lemma states that all configurations $\sigma\in\cX$ with $\Delta H(\sigma)<L$ must belong to one of the two initial cycles.

\begin{lem}[Configurations with $\D H <L$ belong to one of the initial cycles]
\label{lem:reducible}
If a configuration $\s \in \cX$ is such that $\D H(\s) <L$, then there exists a path $\o: \s \to \{\ee,\oo\}$ with $\Phi_\o \leq H(\s)+1$. In particular, either $\s \in C_\ee$ or $\s \in C_\oo$.
\end{lem}
\begin{proof}
If $\D H(\s) <L$, then there exists both a horizontal and a vertical bridge. Due to the hard-core constraints, these $L-1$ particles should all reside on sites of the same parity, hence $\s$ has either an even cross or an odd cross. Then, one can build using the \textit{reduction algorithm} introduced in~\cite{NZB15} a path to either $\ee$ or $\oo$, respectively, with the desired properties.
\end{proof}
\FloatBarrier
\begin{defi}\label{def:essential*} 
A hard-core configuration $\s \in \cX$ on the $L \times L$ toric grid graph $\Lambda$ belongs to $\cC_{ir}(\ee,\oo)$ if the following conditions hold:
\begin{itemize}
	\item[1.] the odd region $\cO(\s)$ contains only a non-degenerate cluster $C$ and a degenerate rhombus $D \in \{\cR_{1,0}, \cR_{0,1}\}$ at distance two from $C$;
	\item[2.] the cluster $C$ is monotone;
	\item[3.] $C$ is a rhombus $\cR_{\frac{L}{2}-1,\frac{L}{2}-1}$; 
	\item[4.] $\s_{| \L \setminus (C\cup D)} = \ee_{| \L \setminus (C \cup D)}$.
\end{itemize}
\end{defi}
See~\cref{fig:selle1} (left) for an example of configurations in $\cC_{ir}(\ee,\oo)$.

\begin{defi}\label{def:essential1} 
A hard-core configuration $\s \in \cX$ on the $L \times L$ toric grid graph $\Lambda$ belongs to $\cC_{gr}(\ee,\oo)$ (resp.\ $\cC_{cr}(\ee,\oo)$) if the following conditions hold:
\begin{itemize}
    \item[1.] the odd region $\cO(\s)$ contains only a non-degenerate cluster $C$ and a degenerate rhombus $D=\cR_{0,0}$ at distance one from an antiknob;
    \item[2.] the cluster $C$ is monotone; 
    \item[3.] $C$ is a rhombus $\cR_{\frac{L}{2}-1,\frac{L}{2}-1}$ (resp.\ $\cR_{\frac{L}{2}-1,\frac{L}{2}}$) with a single bar of length $k$, for $k=1,...,\frac{L}{2}-1$ (resp.\ $k=1,...,\frac{L}{2}-2$), attached to one of its sides; 
    \item[4.] $\s_{| \L \setminus (C\cup D)} = \ee_{| \L \setminus (C \cup D)}$.
\end{itemize}
\end{defi}
See~\cref{fig:selle1} (right) and~\cref{fig:selle2} (left) for an example of configurations in $\cC_{gr}(\ee,\oo)$ and $\cC_{cr}(\ee,\oo)$, respectively. Note that for any $\sigma\in\cC_{ir}(\ee,\oo)\cup\cC_{gr}(\ee,\oo)$ the surrounding rhombus $R(\cO(\s))$ is $R_{\frac{L}{2}-1,\frac{L}{2}}$, while for any $\sigma\in\cC_{cr}(\ee,\oo)$ the surrounding rhombus $R(\cO(\s))$ is $R_{\frac{L}{2},\frac{L}{2}}$.

\begin{defi}\label{def:essential3} 
A hard-core configuration $\s \in \cX$ on the $L \times L$ toric grid graph $\Lambda$ belongs to $\cC_{sb}(\ee,\oo)$ if the following conditions hold:
\begin{itemize}
    \item[1.] the odd region $\cO(\s)$ contains only a non-degenerate cluster $C$ and a degenerate region $D$ consisting of two even sites at distance one from the same antiknob; 
    \item[2.] the cluster $C$ is monotone;
    \item[3.] $C$ is a single column or row of length $\frac{L}{2}-1$;
    \item[4.] $\s_{| \L \setminus (C\cup D)} = \ee_{| \L \setminus (C \cup D)}$. 
\end{itemize}
\end{defi}
Note that the unique possibilities are that (i) $D$ consists of two degenerate rhombi $\cR_{0,0}$ as in~\cref{fig:selle2} (right) or (ii) $D\in\{\cR_{0,1},\cR_{1,0}\}$ as in~\cref{fig:degpos} (left).

\begin{figure}[!h]
\centering
    \begin{tikzpicture}[scale=0.55]
    \newcommand\radius{0.2}
    \newcommand\listodd{(3,-3),(3,1),(3,3),(3,5),(3,7),(3,9)}
    \newcommand\listeven{(2,-3),(4,-3),(3,-2),(4,-1),(3,0),(4,1),(2,1),(3,2),(4,3),(2,3),(3,4),(4,5),(2,5),(3,6),(4,7),(2,7),(3,8),(4,9),(2,9),(3,10),(5,-2)}
    
    \draw[help lines] (-4.5,-3.5) grid (9.5,10.5);

    \draw[thick] (1.5,-3.5)--(2.5,-3.5);
    \draw[thick] (3.5,-1.5)--(2.5,-1.5);
    \draw[thick] (3.5,-0.5)--(2.5,-0.5);
    \draw[thick] (2.5,-2.5)--(1.5,-2.5)--(1.5,-3.5);
    \draw[thick] (3.5,-3.5)--(4.5,-3.5)--(4.5,-2.5)--(3.5,-2.5)--(3.5,-1.5)--(4.5,-1.5)--(4.5,-0.5)--(3.5,-0.5)--(3.5,0.5)--(4.5,0.5)--(4.5,1.5)--(3.5,1.5)--(3.5,2.5)--(4.5,2.5)--(4.5,3.5)--(3.5,3.5)--(3.5,4.5)--(4.5,4.5)--(4.5,5.5)--(3.5,5.5)--(3.5,6.5)--(4.5,6.5)--(4.5,7.5)--(3.5,7.5)--(3.5,8.5)--(4.5,8.5)--(4.5,9.5)--(3.5,9.5)--(3.5,10.5);
    \draw[thick] (2.5,10.5)--(2.5,9.5)--(1.5,9.5)--(1.5,8.5)--(2.5,8.5)--(2.5,7.5)--(1.5,7.5)--(1.5,6.5)--(2.5,6.5)--(2.5,5.5)--(1.5,5.5)--(1.5,4.5)--(2.5,4.5)--(2.5,3.5)--(1.5,3.5)--(1.5,2.5)--(2.5,2.5)--(2.5,1.5)--(1.5,1.5)--(1.5,0.5)--(2.5,0.5)--(2.5,-0.5);
    \draw[thick](2.5,-1.5)--(2.5,-2.5);

    \draw[thick] (4.5,-2.5) rectangle (5.5,-1.5);
    \draw[thick] (3.5,-1.5) rectangle (4.5,-0 .5);

    \foreach \c in \listodd
        \draw[fill=black] \c circle(\radius);
    \foreach \c in \listeven
        \draw[fill=white] \c circle(\radius);
\end{tikzpicture}
\hspace{1cm}
    \begin{tikzpicture}[scale=0.55]
    \newcommand\radius{0.2}
    \newcommand\listodd{(3,-1),(3,1),(3,3),(3,5),(3,7),(2,8),(2,4),(2,6),(2,2),(2,0),(2,-2),(1,-1),(1,1),(1,3),(1,7),(0,2),(0,0),(-1,1),(4,6),(4,4),(4,2),(4,0),(3,9),(4,8),(5,7),(5,3),(5,1)}
    \newcommand\listeven{(2,-3),(3,-2),(4,-1),(2,-1),(3,0),(4,1),(2,1),(3,2),(4,3),(2,3),(3,4),(4,5),(2,5),(3,6),(2,7),(3,8),(2,9),(1,8),(1,6),(1,4),(1,2),(1,0),(1,-2),(0,-1),(0,1),(0,3),(0,7),(-1,2),(-1,0),(-2,1),(5,6),(5,4),(5,2),(5,0),(4,7),(4,9),(3,10),(5,8),(6,7),(6,3),(6,1),(7,2)}
    
    \draw[help lines] (-4.5,-3.5) grid (9.5,10.5);

    \draw[thick] (6.5,1.5) rectangle (7.5,2.5);
    
    \draw[thick] (1.5,-3.5)--(2.5,-3.5)--(2.5,-2.5)--(3.5,-2.5)--(3.5,-1.5)--(4.5,-1.5)--(4.5,-0.5)--(5.5,-0.5)--(5.5,0.5)--(6.5,0.5)--(6.5,1.5)--(5.5,1.5)--(5.5,2.5)--(6.5,2.5)--(6.5,3.5)--(5.5,3.5)--(5.5,4.5)--(4.5,4.5)--(4.5,5.5)--(5.5,5.5)--(5.5,6.5)--(6.5,6.5)--(6.5,7.5)--(5.5,7.5)--(5.5,8.5)--(4.5,8.5)--(4.5,9.5)--(3.5,9.5)--(3.5,10.5)--(2.5,10.5)--(2.5,9.5)--(1.5,9.5)--(1.5,8.5)--(0.5,8.5)--(0.5,7.5)--(-0.5,7.5)--(-0.5,6.5)--(0.5,6.5)--(0.5,5.5)--(1.5,5.5)--(1.5,4.5)--(0.5,4.5)--(0.5,3.5)--(-0.5,3.5)--(-0.5,2.5)--(-1.5,2.5)--(-1.5,1.5)--(-2.5,1.5)--(-2.5,0.5)--(-1.5,0.5)--(-1.5,-0.5)--(-0.5,-0.5)--(-0.5,-1.5)--(0.5,-1.5)--(0.5,-2.5)--(1.5,-2.5)--(1.5,-3.5);

    \foreach \c in \listodd
        \draw[fill=black] \c circle(\radius);
    \foreach \c in \listeven
        \draw[fill=white] \c circle(\radius);
\end{tikzpicture}
\caption{An example of a configuration in $\cC_{sb}(\ee,\oo)$ that communicates with $\cC_{ib}(\ee,\oo)$ and not with $\cC_{mb}(\ee,\oo)$ (on the left) and an example of a configuration in $\cC_{ib}(\ee,\oo)$ (on the right).}
\label{fig:degpos}
\end{figure}
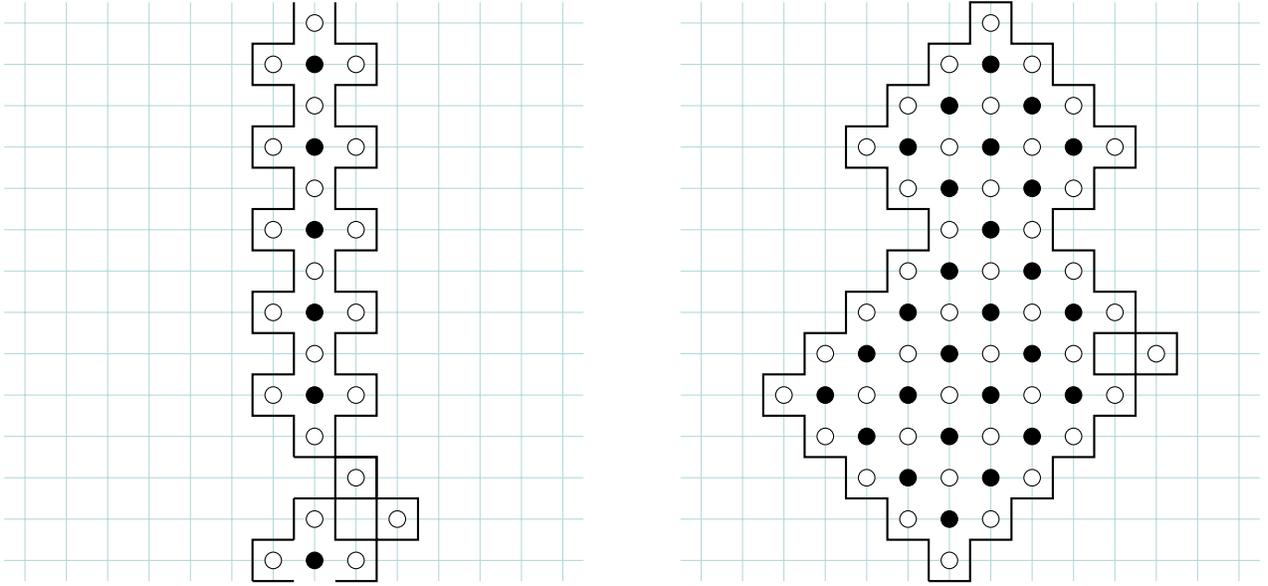

\begin{defi}\label{def:essentialnuove} 
A hard-core configuration $\s \in \cX$ on the $L \times L$ toric grid graph $\Lambda$ belongs to $\cC_{mb}(\ee,\oo)$ if the following conditions hold:
\begin{itemize}
    \item[1.] the odd region $\cO(\s)$ contains only a non-degenerate cluster $C$ and a degenerate rhombus $D=\cR_{0,0}$ at distance one from an antiknob;
    \item[2.] the cluster $C$ is monotone; 
    \item[3.] $C$ is composed either by:
    \begin{itemize}
        \item a odd $(L-3)$-uple bridge, or
        \item an odd $m$-uple bridge, with $2\leq m < L-3$, together with disjoint bars attached to either side of the $m$-uple bridge with total length $k$, with $k=0,...,\frac{L}{2}-1$, or 
        \item an odd bridge, together with disjoint bars attached to the bridge with total length $k$, for $k=0,...,\frac{L}{2}-1$, and $C$ does not contain $\cR_{\frac{L}{2}-1,\frac{L}{2}}$;
    \end{itemize}  
    \item[4.] $\s_{| \L \setminus (C\cup D)} = \ee_{| \L \setminus (C \cup D)}$. 
\end{itemize}
\end{defi}
See~\cref{fig:selle3} (left) for an example of a configuration in $\cC_{mb}(\ee,\oo)$. Note that any configuration $\sigma\in\cC_{mb}(\ee,\oo)$ is such that $\cR(\cO(\sigma))$ winds around the torus. Thus, the assumption that the non-degenerate cluster $C$ is monotone implies that in condition 3 not every choice of the bars is allowed. Indeed, in order for the cluster to be monotone, the bars on the right (resp.\ left) of the $m$-uple bridge can be adjacent only to a longer (resp.\ shorter) bar in lexicographic order. This implies that all the bridges are in contiguous rows or columns. Furthermore, the length of a bar is inversely proportional to its distance from the nearest bar composing the bridge.
\begin{defi}\label{def:essentialnuove2} 
A hard-core configuration $\s \in \cX$ on the $L \times L$ toric grid graph $\Lambda$ belongs to $\cC_{ib}(\ee,\oo)$ if the following conditions hold:
\begin{itemize}
    \item[1.] the odd region $\cO(\s)$ contains only a non-degenerate cluster $C$ and a degenerate rhombus $D=\cR_{0,0}$ at distance one from an antiknob;
    \item[2.] the cluster $C$ is monotone and does not contain $\cR_{\frac{L}{2}-1,\frac{L}{2}-1}$;
    \item[3.] $C$ is composed either by:
    \begin{itemize}
        \item one column (or row) $B$ with $\frac{L}{2}-1$ particles in odd sites or
        \item two neighbouring columns (or row) $B$ with $\frac{L}{2}-1$ particles in odd sites each; 
    \end{itemize}  
    In addition, in the other columns (or rows) there are $k$ particles arranged in odd sites, for $k=0,...,\frac{L}{2}-1-j$, where $j$ is the distance from $B$;
    \item[4.] $\cR(C)$ is not contained in $\cR_{\frac{L}{2}-1,\frac{L}{2}-1}$;
    \item[5.] $\s_{| \L \setminus (C\cup D)} = \ee_{| \L \setminus (C \cup D)}$. 
\end{itemize}
\end{defi}
See~\cref{fig:selle3} (right) and~\cref{fig:degpos} (right) for examples of configurations in $\cC_{ib}(\ee,\oo)$. \medskip

Next, we provide several results investigating the main properties of the optimal paths connecting $\ee$ to $\oo$. To this end, we introduce the following partition in {\it manifolds} of the state space $\cX$: for every $m=-\frac{L^2}{2},...,\frac{L^2}{2}$ we define the {\it manifold} as the subset of configurations where the difference $m(\sigma):=-e(\sigma)+o(\sigma)$ between odd and even occupied sites is equal to $m \in \mathbb{N}$, i.e., 
\begin{equation}\label{eq:foliation}
    \cV_m := \{\sigma\in\cX ~|~ m(\sigma) = m\},
\end{equation}
where $e(\sigma):=\sum_{v \in \Ve} \sigma(v)$ (resp.\ $o(\sigma):=\sum_{v \in \Vo} \sigma(v)$) is the number of the even (resp.\ odd) occupied sites in $\sigma$.

\begin{lem}\label{lmm:manifold}
For any hard-core configuration $\sigma \in \cX$, the following properties hold:
\begin{itemize}
    \item[(a)] The only configurations accessible from $\sigma$ with a single nontrivial step of the dynamics belong to $\cV_{m(\sigma) -1} \cup \cV_{m(\sigma)+1}$. In particular, any path from $\ee$ to $\oo$ must intersect each manifold $\cV_m$ at least once for every $m=-\frac{L^2}{2},...,\frac{L^2}{2}$. 
    \item[(b)] The quantities $m(\sigma)$ and $\Delta H (\sigma)$ always have the same parity, i.e., $m(\sigma) \equiv \Delta H (\sigma) \pmod 2$.
\end{itemize}
\end{lem}
The proof of (a) is immediate by noticing that at every step of the dynamics either $e(\sigma)$ or $o(\sigma)$ can change value and at most by $\pm 1$ and that of (b) follows from the fact that $\Delta H(\sigma)=-e(\sigma)-o(\sigma)+\frac{L^2}{2}$ and that $L$ is even.

A special role in our analysis will be played by the \textit{non-backtracking paths}, i.e., those paths that visit each manifold exactly once.
Lemmas \ref{lmm:cammino*}--\ref{lmm:cammino6} below ensure the existence of an optimal path connecting $\ee$ to $\oo$ and passing through the six sets that define $\cC^*(\ee,\oo)$. In addition, \cref{lmm:nocammino} below shows that some of the sets composing $\cC^*(\ee,\oo)$ do not directly communicate. Later, in \cref{sec:thm2}, combining these lemmas, we will prove that the communication structure of these six sets at energy not higher than $H(\ee)+L+1$ is the one illustrated in~\cref{fig:panorama}. The proof of these lemmas is deferred to \cref{sec:proofrefpath}.

\begin{lem}\label{lmm:cammino*}
The following statements hold.
\begin{itemize}
\item[(i)] For any configuration $\eta\in\cC_{ir}(\ee,\oo)$ there exists a non-backtracking optimal path $\omega:\ee\rightarrow\eta$ such that $\arg\max_{\xi\in\omega}\Delta H(\xi)=\{\eta\}$. 
\item[(ii)] For any configuration $\eta\in\cC_{sb}(\ee,\oo)$ there exists a non-backtracking path $\omega:\ee\rightarrow\eta$ such that $\arg\max_{\xi\in\omega}\Delta H(\xi)=\{\eta\}$. 
\end{itemize}
\end{lem}

\begin{lem}\label{lmm:cammino**}
The following statements hold.
\begin{itemize}
\item[(i)] For any configuration $\eta\in\cC_{ib}(\ee,\oo)$ there exist a configurations $\bar\eta\in\cC_{sb}(\ee,\oo)$ and a non-backtracking path $\omega:\bar\eta\rightarrow\eta$ such that $\Phi_{\omega}-H(\ee)=\Delta H(\eta)=L+1$ and $\arg\max_{\xi\in\omega} H(\xi)\subseteq\cC_{sb}(\ee,\oo)\cup\cC_{ib}(\ee,\oo)$.
\item[(ii)]For any configuration $\eta\in\cC_{gr}(\ee,\oo)$ there exist a configuration $\bar\eta\in\cC_{ir}(\ee,\oo)$, a configuration $\tilde\eta\in\cC_{ib}(\ee,\oo)$ and two non-backtracking paths $\omega:\bar\eta\rightarrow\eta$, $\omega':\tilde\eta\rightarrow\eta$ such that 
\begin{itemize}
\item[-] $\Phi_{\omega}-H(\ee)=L+1$ and $\arg\max_{\xi\in\omega} H(\xi)\subseteq\cC_{gr}(\ee,\oo)\cup\cC_{ir}(\ee,\oo)$;
\item[-] $\Phi_{\omega'}-H(\ee)=L+1$ and $\arg\max_{\xi\in\omega'} H(\xi)\subseteq\cC_{gr}(\ee,\oo)\cup\cC_{ib}(\ee,\oo)$.
\end{itemize}
\item[(iii)] For any configuration $\eta\in\cC_{mb}(\ee,\oo)$ there exist a configuration $\bar\eta\in\cC_{sb}(\ee,\oo)$, a configuration $\tilde\eta\in\cC_{ib}(\ee,\oo)$ and two non-backtracking paths $\omega:\bar\eta\rightarrow\eta$, $\omega':\tilde\eta\rightarrow\eta$ such that 
\begin{itemize}
\item[-] $\Phi_{\omega}-H(\ee)=L+1$ and $\arg\max_{\xi\in\omega} H(\xi)\subseteq\cC_{mb}(\ee,\oo)\cup\cC_{sb}(\ee,\oo)$;
\item[-] $\Phi_{\omega'}-H(\ee)=L+1$ and $\arg\max_{\xi\in\omega'} H(\xi)\subseteq\cC_{mb}(\ee,\oo)\cup\cC_{ib}(\ee,\oo)$.
\end{itemize}
\item[(iv)] For any configuration $\eta\in\cC_{cr}(\ee,\oo)$ there exist a configuration $\bar\eta\in\cC_{gr}(\ee,\oo)$, a configuration $\tilde\eta\in\cC_{mb}(\ee,\oo)$ and two non-backtracking paths $\omega:\bar\eta\rightarrow\eta$, $\omega':\tilde\eta\rightarrow\eta$ such that 
\begin{itemize}
\item[-] $\Phi_{\omega}-H(\ee)=L+1$ and $\arg\max_{\xi\in\omega} H(\xi)\subseteq\cC_{cr}(\ee,\oo)\cup\cC_{gr}(\ee,\oo)$;
\item[-] $\Phi_{\omega'}-H(\ee)=L+1$ and $\arg\max_{\xi\in\omega'} H(\xi)\subseteq\cC_{cr}(\ee,\oo)\cup\cC_{mb}(\ee,\oo)$.
\end{itemize}
\end{itemize}
\end{lem}

\begin{lem}\label{lmm:cammino6}
The following statements hold.
\begin{itemize}
    \item[(i)] For any configuration $\eta\in\cC_{cr}(\ee,\oo)$ there exists a non-backtracking path $\omega:\eta\rightarrow\oo$ such that $\Phi_{\omega}-H(\ee)=L+1$ and $\arg\max_{\xi\in\omega} H(\xi)\subseteq\cC_{cr}(\ee,\oo)$.
    \item[(ii)] For any configuration $\eta\in\cC_{mb}(\ee,\oo)$ there exists a non-backtracking path $\omega:\eta\rightarrow\oo$ such that $\Phi_{\omega}-H(\ee)=L+1$ and $\arg\max_{\xi\in\omega} H(\xi)\subseteq\cC_{mb}(\ee,\oo)$.
\end{itemize}
\end{lem}

\begin{lem}\label{lmm:nocammino}
The following statements hold.
\begin{itemize}
\item[(i)] For any configurations $\eta\in\cC_{ir}(\ee,\oo)$ and $\eta'\in\cC_{ib}(\ee,\oo)$, there is no optimal path $\omega:\eta\to\eta'$ such that $\arg\max_{\xi\in\omega}\subseteq\cC_{ir}(\ee,\oo)\cup\cC_{ib}(\ee,\oo)$.
\item[(ii)] For any configurations $\eta\in\cC_{ib}(\ee,\oo)$ and $\eta'\in\cC_{cr}(\ee,\oo)$, there is no optimal path $\omega:\eta\to\eta'$ such that $\arg\max_{\xi\in\omega}\subseteq\cC_{cr}(\ee,\oo)\cup\cC_{ib}(\ee,\oo)$.
\item[(iii)] For any configurations $\eta\in\cC_{ir}(\ee,\oo)$ and $\eta'\in\cC_{mb}(\ee,\oo)$, there is no optimal path $\omega:\eta\to\eta'$ such that $\arg\max_{\xi\in\omega}\subseteq\cC_{ir}(\ee,\oo)\cup\cC_{mb}(\ee,\oo)$.
\item[(iv)] For any configurations $\eta\in\cC_{gr}(\ee,\oo)$ and $\eta'\in\cC_{mb}(\ee,\oo)$, there is no optimal path $\omega:\eta\to\eta'$ such that $\arg\max_{\xi\in\omega}\subseteq\cC_{gr}(\ee,\oo)\cup\cC_{mb}(\ee,\oo)$.
\end{itemize}
\end{lem}

\subsection{Proof of \cref{thm:saddles}}
\label{sec:actualproof}
This section is entirely devoted to the proof of \cref{thm:saddles}. More specifically, we prove that any essential saddle belongs to the set $\cC^*(\ee,\oo)$ in \cref{sec:inclusione2}, we describe how the transitions between essential gates can take place  in \cref{sec:thm2}, and we prove that all the saddles in $\cC^*(\ee,\oo)$ are essential in \cref{sec:inclusione1}.

\subsubsection{Every essential saddle belongs to $\cC^*(\ee,\oo)$}
\label{sec:inclusione2}

In this subsection, we will show that any essential saddle $\s$ belongs to the subset $\cC^*(\ee,\oo)$. This readily follows from~\cref{prp:vegetariana2} below.

\begin{prop}\label{prp:vegetariana2}
Let $\sigma$ be an essential saddle. Then, the following statements hold:
\begin{itemize}
    \item[(i)] If $\cR(\cO^{nd}(\sigma))$ and $\cR(\cO(\sigma))$ do not wind around the torus, then $\sigma\in\cC_{ir}(\ee,\oo)\cup\cC_{gr}(\ee,\oo)$.
    \item[(ii)] If $\cR(\cO^{nd}(\sigma))$ does not wind around the torus, $\cR(\cO(\sigma))$ does, and $\sigma$ belongs to $\omega\in(\ee\to\oo)_{opt}$ that crosses the set $\cC_{ir}(\ee,\oo)\cup\cC_{gr}(\ee,\oo)$, then $\sigma\in\cC_{cr}(\ee,\oo)$.
    \item[(iii)] If $\cR(\cO^{nd}(\sigma))$ does not wind around the torus, $\cR(\cO(\sigma))$ does, and $\sigma$ belongs to $\omega\in(\ee\to\oo)_{opt}$ that does not cross the set $\cC_{ir}(\ee,\oo)\cup\cC_{gr}(\ee,\oo)$, then $\sigma\in\cC_{sb}(\ee,\oo)\cup\cC_{ib}(\ee,\oo)$.
    \item[(iv)] If $\cR(\cO^{nd}(\sigma))$ winds around the torus, then $\sigma\in\cC_{mb}(\ee,\oo)$. 
\end{itemize}
\end{prop}
We observe that these four cases (i)--(iv) listed in~\cref{prp:vegetariana2} cover all the possibilities and thus form a partition of the set $\cG(\ee,\oo)$ of essential saddles for the transitions $\ee \rightarrow \oo$.

Before presenting the proof, note that given a configuration $\sigma \in \cS(\ee,\oo)$, if we know on which manifold $\cV_m$ it lies, then the quantities $e(\sigma)$ and $o(\sigma)$ are uniquely determined and can be explicitly calculated as 
\begin{align}\label{eq:stimeoe}
    & o(\sigma)=\frac{m}{2}+\frac{L^2}{4}-\frac{L+1}{2} \qquad \text{and} \qquad e(\sigma)=-\frac{m}{2}+\frac{L^2}{4}-\frac{L+1}{2}.
\end{align}
This readily follows from the fact that $\Delta H(\sigma) = L+1=-e(\sigma)-o(\sigma)+\frac{L^2}{2}$ and that $m=-e(\sigma)+o(\sigma)$.
Furthermore, since $\sigma \in \cS(\ee,\oo)$, $\Delta H(\sigma) = L+1$ and $m$ has to be an odd integer by~\cref{lmm:manifold}(b).

To prove \cref{prp:vegetariana2}, we make use of an additional lemma (whose proof is deferred to \cref{sec:proofmanifold}), which characterizes the intersection between any optimal path and a specific manifold, namely $\cV_{m^*}$ with $m^*:=3-L$. 

\begin{lem}[Geometrical properties of the saddles on the manifold $\cV_{m^*}$] \label{lmm:gorgonzola} 
Any non-backtracking optimal path $\omega: \ee \to \oo$ visits a configuration $\s\in\cV_{m^*}$ that satisfy one of the following properties:
\begin{itemize}
\item[(i)] If both $\cR(\cO^{nd}(\sigma))$ and $\cR(\cO(\sigma))$ do not wind around the torus, then $\sigma\in\cC_{ir}(\ee,\oo)$.
\item[(ii)] If $\cR(\cO^{nd}(\sigma))$ does not wind around the torus but $\cR(\cO(\sigma))$ does, then $\sigma\in\cC_{ib}(\ee,\oo)$.
\item[(iii)] If both $\cR(\cO^{nd}(\sigma))$ and $\cR(\cO(\sigma))$ wind around the torus, then $\sigma\in\cC_{mb}(\ee,\oo)$. 
\end{itemize}
\end{lem}

\begin{proof}[Proof of \cref{prp:vegetariana2}]

\noindent
\\
{\bf Case (i).} Let $\sigma$ be an essential saddle such that $\cR(\cO^{nd}(\sigma))$ and $\cR(\cO(\sigma))$ do not wind around the torus. If $\sigma\in\cV_{m^*}$, by \cref{lmm:gorgonzola}(i) we know that $\s\in\cC_{ir}(\ee,\oo)$. Otherwise, we suppose that $\sigma\notin\cV_{m^*}$. 
First, we observe that $\sigma \not \in \cV_{m}$ with $m<m^*$, otherwise the saddle $\sigma$ is not essential. 
Indeed, every optimal path from $\ee$ to $\oo$ has to cross a configuration $\bar\sigma\in\cV_{m^*}$ in $\cC_{ir}(\ee,\oo)$ thanks to \cref{lmm:gorgonzola}(i). Thus, we can write a general optimal path $\omega=(\ee,\omega_1,...,\omega_k,\sigma,...,\bar\sigma,\omega_{k+1},...,\omega_{k+m},\oo)$ and we can define the path $\omega'=(\ee,\tilde\omega_1,...,\tilde\omega_n,\bar\sigma,\omega_{k+1},...,\omega_{k+m},\oo)$, where $\arg\max_{\xi\in\{\ee,\tilde\omega_1,...,\tilde\omega_n,\bar\sigma\}}H(\xi)=\{\bar\sigma\}$. This path $\o'$ exists thanks to \cref{lmm:cammino*}(i). Thus, we are left to analyze the case $\sigma \in \cV_{m}$ with $m> m^*$. We need to show that any essential saddle crossed afterward belongs to the set $\cC_{gr}(\ee,\oo)$. By using again \cref{lmm:gorgonzola}, the path $\omega$ crosses a configuration $\sigma\in\cC_{ir}(\ee,\oo)$ to reach $\oo$. Starting from it, there is a unique possible move to lower the energy towards $\oo$ along the path $\omega$, that is adding a particle in the unique unblocked empty odd site. Afterward, the unique possible move is to remove a particle from an even site. If this site is at a distance greater than one from an antiknob, then there is no more allowed  move. Otherwise, the resulting configuration belongs to the set $\cC_{gr}(\ee,\oo)$. By iterating this pair of moves until the shorter diagonal of the rhombus is completely filled, we obtain that all the saddles that are crossed belong to $\cC_{gr}(\ee,\oo)$. Moreover, from this point onwards, it is only possible to remove a particle from an even site at distance one from the antiknob, obtaining a configuration in $\cC_{gr}(\ee,\oo)\cap\cV_{1}$.

\medskip
\noindent
{\bf Case (ii).} By assumption, the path $\omega$ crosses the set $\cC_{ir}(\ee,\oo)\cup\cC_{gr}(\ee,\oo)$. 
Without loss of generality, we may consider $\omega$ as a non-backtracking path. If this is not the case, we can apply the following argument to the last configuration visited by the path in the manifold $\cV_{m^*}$.
In particular, in view of the properties of the path $\omega$ shown in case (i), we know that the last configuration crossed in $\cC_{gr}(\ee,\oo)$ belongs to $\cV_1$ and it is composed of a unique non-degenerate cluster $\cR_{\frac{L}{2}-1,\frac{L}{2}}$ with a degenerate cluster $\cR_{0,0}$ at distance one from the antiknob. Starting from it, there is a unique possible move to lower the energy towards $\oo$ along the path $\omega$, that is adding a particle in the unique unblocked empty odd site. Afterward, the unique possible move is to remove a particle from an even site. The resulting configuration is in $\cC_{cr}(\ee,\oo)$. By iterating this pair of moves until the shorter diagonal of the rhombus is completely filled, we obtain that all the saddles that are crossed belong to $\cC_{cr}(\ee,\oo)$.

\medskip
\noindent
{\bf Case (iii).} As in case (ii), we assume that the path $\omega$ is non-backtracking. By the assumption, the path $\o$ does not cross the set $\cC_{ir}(\ee,\oo)\cup\cC_{gr}(\ee,\oo)$, thus for \cref{lmm:gorgonzola}(ii) the path crosses the set $\cC_{ib}(\ee,\oo)$, say in the configuration $\bar\eta$. Let $\cV_{\bar m}$ the first manifold containing a configuration in $\cC_{ib}(\ee,\oo)$. 

Consider first the case in which the saddle $\s$ is crossed by the path $\o$ on the manifolds $\cV_m$, with $\bar m\leq m<m^*$. Since $\o$ crosses the set $\cC_{ib}(\ee,\oo)$, we will show that the saddle $\sigma$ belongs to the set $\cC_{sb}(\ee,\oo)\cup\cC_{ib}(\ee,\oo)$. To this end, we need to consider the time-reversal of the path $\omega=(\ee,\omega_1,...,\omega_k,\bar\eta,...,\oo)$, where $\bar\eta\in\cC_{ib}(\ee,\oo)$. Starting from $\bar\eta$, since $\Delta H(\bar\eta)=L+1$, the energy of the configuration $\omega_k$ is less than that of $\bar\eta$. Moreover, the unique possible move is adding a particle in the unique empty even site at distance one from an antiknob. Then, the unique possible move from $\omega_k$ to $\omega_{k-1}$ is removing a particle from an occupied odd site. The configuration $\omega_{k-1}$ belongs to the set $\cC_{sb}(\ee,\oo)$ if it contains at least one bridge, otherwise, it belongs to the set $\cC_{ib}(\ee,\oo)$. By iterating this argument we obtain the desired claim. 

Consider now the case in which the saddle $\s$ is crossed by the path $\o$ on the manifolds $\cV_m$, with $ m>m^*$. Since $\o$ crosses the set $\cC_{ib}(\ee,\oo)$, we will show that the saddle $\sigma$ belongs to the set $\cC_{sb}(\ee,\oo)\cup\cC_{ib}(\ee,\oo)$. Note that the unique admissible moves are the following: add a particle in the antiknob and then remove a particle from an even site at distance one from an antiknob. By iterating this couple of moves we get that the resulting configuration belongs to the set $\cC_{ib}(\ee,\oo)$ since $\cR(\cO^{nd}(\s))$ does not wind around the torus.

It remains to consider the case $\sigma\in\cV_m$, with $m<\bar m$. We will prove that any such $\sigma$ is not essential. Indeed, any non-backtracking optimal path crossing a configuration in $\cC_{ib}(\ee,\oo)$ has crossed a configuration in $\cC_{sb}(\ee,\oo)$ before by using the same argument as in case (iii) of this proof for $\bar m\leq m<m^*$. Thus, we can write $\omega=(\ee,\omega_1,...,\omega_k,\sigma,...,\tilde\s,...,\bar\sigma,\omega_{k+1},...,\omega_{k+m},\oo)$ and we define the path $\omega'=(\ee,\omega_1',...,\omega_n',\tilde\sigma,...,\bar\s,\omega_{k+1},...,\omega_{k+m},\oo)$, where $\tilde\sigma$ is a configuration in $\cC_{sb}(\ee,\oo)$, $\bar\sigma\in\cV_{\bar m}$ is a configuration in $\cC_{ib}(\ee,\oo)$ thanks to \cref{lmm:gorgonzola}(ii) and $\arg\max_{\xi\in\{\ee,\omega_1',...,\omega_n',\tilde\sigma\}}H(\xi)=\{\tilde\sigma\}$. This path $\o'$ exists thanks to \cref{lmm:cammino*}(ii) and this concludes case (iii).

\medskip
\noindent
{\bf Case (iv).} By \cref{lmm:gorgonzola}(iii), we know that any optimal path $\omega\in(\ee\rightarrow\oo)_{\mathrm{opt}}$ crosses the manifold $\cV_{m^*}$ in a configuration $\bar\sigma$ belonging to the set $\cC_{mb}(\ee,\oo)$.

If the essential saddle $\s\in\cV_{m^*}$, then we deduce that $\s\in\cC_{mb}(\ee,\oo)$.

Suppose now that the saddle $\sigma$ belongs to the manifold $\cV_m$, with $m>m^*$. Starting from such a saddle $\bar\sigma$, there is a unique possible move to lower the energy towards $\oo$ along the path $\omega$, that is add a particle in the unique unblocked empty odd site. Afterward, the unique possible move is to remove a particle from an even site. If this site is at a distance greater than one from an antiknob, then there is no more possible move to reach $\oo$ in such a way the path $\o$ is optimal. Otherwise, the resulting configuration belongs to the set $\cC_{mb}(\ee,\oo)$. Indeed, by construction, we deduce that the resulting non-degenerate odd cluster is still monotone due to the properties of the bars attached to each bridge. By iterating this pair of moves until there is a row or a column which is not a bridge, we obtain that all the saddles that are crossed belong to $\cC_{mb}(\ee,\oo)$. Finally, from this point onwards, it is only possible to remove a particle from an even site at distance one from the antiknob, obtaining the last configuration in $\cC_{mb}(\ee,\oo)$. Afterward, the energy only decreases and therefore no more saddles are crossed.

Suppose now that the saddle $\sigma$ belongs to the manifold $\cV_m$, with $m<m^*$. Note that, starting from such $\bar\sigma\in\cC_{mb}(\ee,\oo)$, the unique admissible moves to get $\s$ are the following: add a particle in the unique empty even site at distance one from an antiknob and afterward remove a particle from an occupied odd site. By iterating this couple of moves we get that the resulting configuration belongs to the set $\cC_{mb}(\ee,\oo)$ as long as $\cR(\cO^{nd}(\s))$ winds around the torus, otherwise the saddle $\s$ does not satisfy the properties in the statement. 
\end{proof}

\subsubsection{Communications between essential gates}
\label{sec:thm2}

In this subsection, we show that the six subsets composing the set $\cC^*(\ee,\oo)$ communicate as illustrated in~\cref{fig:panorama}. The next proposition makes it precise.

\begin{prop}\label{cor:entrance}
Any non-backtracking optimal path $\omega:\ee\to\oo$ crosses the set $\cC^*(\ee,\oo)$ in one of the following ways:
\begin{itemize}
    \item[(i)] $\omega$ passes first through $\cC_{ir}(\ee,\oo)$, then through $\cC_{gr}(\ee,\oo)$, and finally through $\cC_{cr}(\ee,\oo)$;
    \item[(ii)] $\omega$ passes first through $\cC_{sb}(\ee,\oo)$, then through $\cC_{ib}(\ee,\oo)$ and afterwards through $\cC_{gr}(\ee,\oo) \cup \cC_{mb}(\ee,\oo)$. If $\omega$ passes $\cC_{gr}(\ee,\oo)$, then it eventually has to visit $\cC_{cr}(\ee,\oo)$, otherwise it does not have to; 
    \item[(iii)] $\omega$ passes first through $\cC_{sb}(\ee,\oo)$ and then through $\cC_{mb}(\ee,\oo)$.
\end{itemize}
\end{prop}
\begin{proof}
Consider a non-backtracking optimal path $\omega:\ee\to\oo$. If $\omega$ visits at least one essential saddle, then we conclude by using \cref{prp:vegetariana2}. Thus, suppose that $\omega$ visits unessential saddles only, say $\s_1,...,\s_n$. By definition of unessential saddle, we know that there exists another optimal path $\omega':\ee\to\oo$ such that $S(\o')\subseteq\{\s_1,...,\s_{n-1}\}$, say $S(\o')\subseteq\{\s_1,...,\s_m\}$ with $m\leq n-1$. Iterating this argument, we deduce that there exists an optimal path $\bar\o:\ee\to\oo$ such that $S(\bar\o)=\{\s_1\}$ and this is a contradiction with the assumption that $\s_1$ is an unessential saddle. Thus, we conclude that any optimal path $\omega:\ee\to\oo$ visits the set $\cC^*(\ee,\oo)$. It remains to prove that the entrance in $\cC^*(\ee,\oo)$ occurs in one of the ways described in (i)--(iii), which easily follows by combining Lemmas \ref{lmm:cammino*}--\ref{lmm:nocammino}.
\end{proof}

\subsubsection{All the saddles in $\cC^*(\ee,\oo)$ are essential}
\label{sec:inclusione1}

In this first part of the proof, we will prove that every $\s \in\mathcal C^*(\ee,\oo)$ is an essential saddle by constructing a non-backtracking optimal path $\o: \ee \to \oo$ that visits $\s$.

Leveraging the fact that $\s \in\mathcal C^*(\ee,\oo)$, we construct the desired non-backtracking path $\o$ as a concatenation of two paths as follows. 
First, using a suitable concatenation of the paths described in Lemmas \ref{lmm:cammino**}--\ref{lmm:cammino6}, we can define a path $\o_1$ that starts from the considered configuration $\s$ to the initial cycle $\mathcal{C}_\ee$.
We construct then another path $\o_2$ that goes from $\s$ to the target cycle $\mathcal{C}_\oo$ as a suitable concatenation of the paths described in Lemmas \ref{lmm:cammino*}--\ref{lmm:cammino**}.
The desired non-backtracking path $\o$ is the time-reversal of $\o_1$ concatenated with $\o_2$ and it is easy to show that it is also optimal. 

Assume now by contradiction that $\s$ is not essential, which means that there must exist another optimal path $\o' \in (\ee \to \oo)_{\textrm{opt}}$ such that $S(\o') \subset S(\o) \setminus \{ \s \}$. Recall that by Lemma~\ref{lmm:manifold}(a), such a path $\o'$ that avoids $\s$ still needs to visit the manifold $\cV_{m(\sigma)}$ where $\s$ lives at least once. Let $\h$ be any such configuration in $\cV_{m(\sigma)}\cap \o'$. We claim that such a configuration $\h$ must satisfy 
\[
    \Delta H(\h)\equiv 1 \pmod 2.
\]
This claim readily follows from Lemma~\ref{lmm:manifold}(b) in combination with the facts that $L$ is even and $\Delta H(\s) = L+1$ by construction.

If $\Delta H(\h) \geq L+3$, then $\o'$ is not an optimal path, since $\Phi_{\o'}- H(\ee) \geq L+3 > L+1 = \Phi(\ee,\oo)-H(\ee)$. 

On the other hand, if $\Delta H(\h)\leq L-1$, then from Lemma~\ref{lem:reducible} it follows that $\h$ belongs to one of the two initial cycles.  
\cref{cor:entrance} ensures that every non-backtracking optimal path crosses $\cC^*(\ee,\oo)$ in one of the three ways (i)--(iii) described therein, so that also the optimal path $\o'$ passing through $\eta$ has to visit $\cC^*(\ee,\oo)$. Since, by assumption, the path $\o'$ has to avoid the saddle $\s$, we deduce that there exists another saddle $\tilde\eta$ obtained starting from $\eta$ that does not belong to $\cV_{m(\sigma)}$. In particular, the two paths cross the set $\cC^*(\ee,\oo)$ in three different ways according to \cref{cor:entrance}(i)--(iii).
Thus, we deduce that $S(\o') \not\subset S(\o) \setminus \{ \s \}$.

Thus, in view of the parity of $\Delta H(\h)$, we must have $\Delta H(\h) = L+1$, but then $\h$ is a saddle and, by construction, it did not belong to $S(\o)$ and thus $S(\o') \not\subset S(\o) \setminus \{ \s \}$.

\section{Proof of auxiliary results}
\label{sec:auxproofs}
In this section, we give the proof of some auxiliary results stated in Sections \ref{sec:defaux}--\ref{sec:proofthm}.

\subsection{Results on the perimeter of an odd region}
\label{sec:oddproof}

\begin{proof}[Proof of \cref{lem:expansion}]
The proof revolves around the simple idea that using the filling algorithms introduced in \cref{sec:filling} we can expand the odd cluster $C$ (\ie progressively increase the number of the occupied odd sites in $C$) in such a way that $\cR$ remains the surrounding rhombus and the energy of all the configurations along such a path never exceeds $H(\s)+1$. Since by assumption $\s \neq \s'$, the odd cluster $C$ cannot coincide with $\cR$, in view of conditions (i), (ii), and (iii). Thus, $C$ contains at least a broken diagonal or a shorter diagonal than those of the surrounding rhombus. 

We can define the desired path $\omega$ as the concatenation of the two paths returned by the filling algorithms $\tilde\omega$ and $\bar\omega$. If $C$ has no broken diagonal, we take $\tilde\omega$ empty. If $C$ has no shorter diagonal, then $C$ has already the shape of a rhombus and therefore we take $\bar\omega$ empty. By the definition of these two paths, the energy increases by one only if an even site has to be emptied, but all these moves are followed by the addition of a particle in an antiknob. Therefore, the energy along the path $\omega$ increases by at most one with respect to the starting configuration $\sigma$. This procedure ends when $C$ coincides with $\cR$, which implies that the resulting configuration is $\sigma'$. 

To conclude the proof, we need to show the properties claimed for the perimeter of $\sigma$. If $C$ contains $m\geq1$ broken diagonals, we argue as follows. 
Since the cluster $C$ is connected, all the empty odd sites in which the diagonals are broken are antiknobs, i.e., they have $n \in \{3,4\}$ neighboring even sites belonging to $C$, see~\cref{fig:ex1}. We distinguish the two following cases. If $n=3$, then we first need to remove a particle from the unique neighboring occupied even site, like the site $v$ represented in~\cref{fig:ex1} on the left. 
\begin{figure}[h!]
\centering
    \begin{tikzpicture}[scale=0.55]
    \newcommand\radius{0.2}
    \newcommand\listodd{(-2,0),(-1,-1),(1,3),(2,4),(-1,3),(0,4),(2,2),(2,6),(1,7),(0,6),(3,5),(4,4),(1,-1),(2,0),(3,1),(4,2),(3,3),(5,3),(-2,2),(-3,3),(0,-2),(4,6),(3,7),(4,8)}
    \newcommand\listeven{(-1,0),(-2,1),(-3,0),(-2,-1),(0,-1),(-1,-2),(3,0),(4,1),(3,0),(2,3),(1,4),(0,3),(1,2),(3,4),(2,5),(-1,4),(-2,3),(-1,2),(0,5),(3,2),(2,1),(3,6),(2,7),(1,6),(1,8),(0,7),(-1,6),(4,5),(5,4),(4,3),(2,-1),(1,0),(1,-2),(5,2),(6,3),(-3,2),(-3,4),(-4,3),(0,-3),(5,6),(4,7),(3,8),(5,8),(4,9)}
    \newcommand\listantiknob{(5,7),(5,5),(1,5),(-1,5),(-2,4),(0,2),(1,1),(-1,1),(-3,1),(0,0),(2,8)}

    \draw[thick,->] (-2.4,6.2)--(-2,5.2);
    \node at (-2.4,6.5) {$v$};
   
    \draw[thick,red,fill=red,opacity=0.4] (-2.5,3.5) rectangle (-1.5,4.5);
        
    \draw[help lines] (-4.5,-3.5) grid (9.5,10.5);
        
    \draw[thick] (-3.5,-0.5)--(-2.5,-0.5)--(-2.5,-1.5)--(-1.5,-1.5)--(-1.5,-2.5)--(-0.5,-2.5)--(-0.5,-3.5) -- (0.5,-3.5)--(0.5,-2.5)--(1.5,-2.5)--(1.5,-1.5)--(2.5,-1.5)--(2.5,-0.5)--(3.5,-0.5)--(3.5,0.5)--(4.5,0.5)--(4.5,1.5)--(5.5,1.5)--(5.5,2.5)--(6.5,2.5)--(6.5,3.5)--(5.5,3.5)--(5.5,4.5)--(4.5,4.5)--(4.5,5.5)--(5.5,5.5)--(5.5,6.5)--(4.5,6.5)--(4.5,7.5)--(5.5,7.5)--(5.5,8.5)--(4.5,8.5)--(4.5,9.5)--(3.5,9.5)--(3.5,8.5)--(2.5,8.5)--(2.5,7.5)--(1.5,7.5)--(1.5,8.5)--(0.5,8.5)--(0.5,7.5)--(-0.5,7.5)--(-0.5,6.5)--(-1.5,6.5)--(-1.5,5.5)--(-0.5,5.5)--(-0.5,4.5)--(-1.5,4.5)--(-1.5,3.5)--(-2.5,3.5)--(-2.5,4.5)--(-3.5,4.5)--(-3.5,3.5)--(-4.5,3.5)--(-4.5,2.5)--(-3.5,2.5)--(-3.5,1.5)--(-2.5,1.5)--(-2.5,0.5)--(-3.5,0.5)--(-3.5,-0.5);
    
    \draw[thick] (-0.5,0.5)--(-0.5,-0.5)--(0.5,-0.5)--(0.5,0.5)--(1.5,0.5)--(1.5,1.5)--(0.5,1.5)--(0.5,2.5)--(-0.5,2.5)--(-0.5,1.5)--(-1.5,1.5)--(-1.5,0.5)--(-0.5,0.5);
    \draw[thick] (0.5,4.5)--(1.5,4.5)--(1.5,5.5)--(0.5,5.5)--(0.5,4.5);

    \foreach \c in \listodd
        \draw[fill=black] \c circle(\radius);
    \foreach \c in \listeven
        \draw[fill=white] \c circle(\radius);
    
\end{tikzpicture}
\hspace{1cm}
\begin{tikzpicture}[scale=0.55]
    \newcommand\radius{0.2}
    \newcommand\listodd{(-2,0),(-1,-1),(1,3),(2,4),(-1,3),(0,4),(2,2),(2,6),(1,7),(0,6),(3,5),(4,4),(1,-1),(2,0),(3,1),(4,2),(3,3),(5,3),(-2,2),(-3,3),(0,-2),(4,6),(3,7),(4,8),(-2,4)}
    \newcommand\listeven{(-1,0),(-2,1),(-3,0),(-2,-1),(0,-1),(-1,-2),(3,0),(4,1),(3,0),(2,3),(1,4),(0,3),(1,2),(3,4),(2,5),(-1,4),(-2,3),(-1,2),(0,5),(3,2),(2,1),(3,6),(2,7),(1,6),(1,8),(0,7),(-1,6),(4,5),(5,4),(4,3),(2,-1),(1,0),(1,-2),(5,2),(6,3),(-3,2),(-3,4),(-4,3),(0,-3),(5,6),(4,7),(3,8),(5,8),(4,9),(-2,5)}
    
    \draw[help lines] (-4.5,-3.5) grid (9.5,10.5);
        
    \draw[thick] (-3.5,-0.5)--(-2.5,-0.5)--(-2.5,-1.5)--(-1.5,-1.5)--(-1.5,-2.5)--(-0.5,-2.5)--(-0.5,-3.5) -- (0.5,-3.5)--(0.5,-2.5)--(1.5,-2.5)--(1.5,-1.5)--(2.5,-1.5)--(2.5,-0.5)--(3.5,-0.5)--(3.5,0.5)--(4.5,0.5)--(4.5,1.5)--(5.5,1.5)--(5.5,2.5)--(6.5,2.5)--(6.5,3.5)--(5.5,3.5)--(5.5,4.5)--(4.5,4.5)--(4.5,5.5)--(5.5,5.5)--(5.5,6.5)--(4.5,6.5)--(4.5,7.5)--(5.5,7.5)--(5.5,8.5)--(4.5,8.5)--(4.5,9.5)--(3.5,9.5)--(3.5,8.5)--(2.5,8.5)--(2.5,7.5)--(1.5,7.5)--(1.5,8.5)--(0.5,8.5)--(0.5,7.5)--(-0.5,7.5)--(-0.5,6.5)--(-1.5,6.5)--(-1.5,5.5)--(-2.5,5.5)--(-2.5,4.5)--(-3.5,4.5)--(-3.5,3.5)--(-4.5,3.5)--(-4.5,2.5)--(-3.5,2.5)--(-3.5,1.5)--(-2.5,1.5)--(-2.5,0.5)--(-3.5,0.5)--(-3.5,-0.5);
    
    \draw[thick] (-0.5,0.5)--(-0.5,-0.5)--(0.5,-0.5)--(0.5,0.5)--(1.5,0.5)--(1.5,1.5)--(0.5,1.5)--(0.5,2.5)--(-0.5,2.5)--(-0.5,1.5)--(-1.5,1.5)--(-1.5,0.5)--(-0.5,0.5);
    \draw[thick] (0.5,4.5)--(1.5,4.5)--(1.5,5.5)--(0.5,5.5)--(0.5,4.5);

    \draw[thick] (-1.5,4.5) rectangle (-0.5,5.5);
    \draw[thick,fill=red,opacity=0.4] (-1.5,4.5) rectangle (-0.5,5.5);

    \draw[thick,->] (-2.4,6.2)--(-2,5.2);
    \node at (-2.4,6.5) {$v$};

    \foreach \c in \listodd
        \draw[fill=black] \c circle(\radius);
    \foreach \c in \listeven
        \draw[fill=white] \c circle(\radius);
    
\end{tikzpicture}
\caption{Example of a configuration $\sigma$ as in the statement of \cref{lem:expansion} (on the left), where we highlight in red the site containing the target antiknob, and the configuration obtained from $\sigma$ by filling it after removing a particle from the even site $v$ (on the right), with highlighted in red the site containing the next target antiknob.}
\label{fig:ex1}
\end{figure}
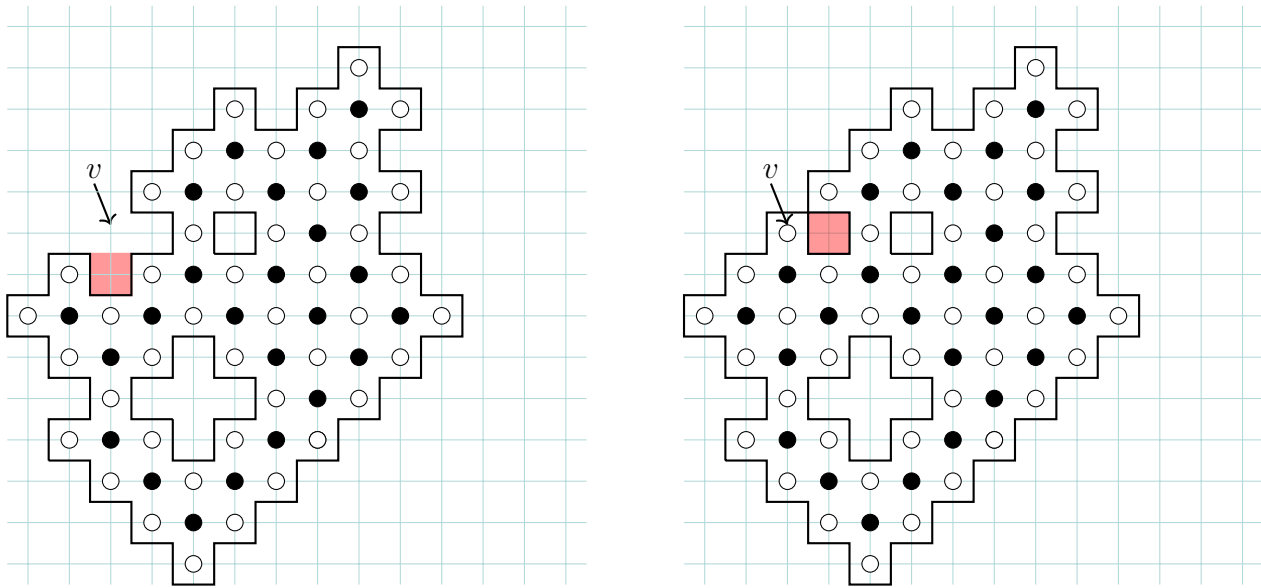
After that move, when the particle is added in the antiknob the perimeter does not change in view of \eqref{eq:identitypC}, otherwise if $n=4$ then the perimeter decreases (see~\cref{fig:ex1} on the right). This occurs also when we add a particle in the target antiknobs except the last antiknob to obtain the complete diagonal, for which by construction we have that the perimeter decreases by 4. By iterating this argument for every broken diagonal, we get 
\[
    P(\sigma)\geq P(\tilde\sigma)+4m>P(\tilde\sigma).
\]
If $C$ does not contain any broken diagonal, we argue as follows. 
By construction, we deduce that the first odd site we will fill, which is the nearest neighbor of the first shorter diagonal, has three neighboring even sites belonging to $C$. Thus, we need to remove a particle from the unique occupied neighboring even site and then, when the first particle is added in the unique possible antiknob, the perimeter does not change thanks to \eqref{eq:identitypC}. This occurs also when we iterate this argument.
\end{proof}

\begin{proof}[Proof of \cref{lem:minP}]
Given any $n$ positive number, let $C$ be an odd cluster with area $n$. If the cluster $C$ is not connected, by \eqref{eq:identitypC} it directly follows that $C$ cannot minimize the perimeter of an odd cluster with $n$ particles. Indeed, when the cluster is not connected the cardinality of $C\cap V_{\ee}$ decreases, while the cardinality of $C\cap V_{\oo}$ is fixed equal to $n$.  
Consider now a connected cluster $C$. 
First, we will show that the minimal perimeter is that of the surrounding rhombus. To this end, we consider separately the following three cases:
\begin{enumerate}
    \item there exists at least one broken diagonal. In this case, we have that $P(C)>P(\cR(C))$ by applying \cref{lem:expansion}.
    \item there is no broken diagonal but there exists at least one shorter diagonal with respect to those of the surrounding rhombus. In this case, we have that $P(C)=P(\cR(C))$ by applying \cref{lem:expansion}.
    \item all the diagonals have the same length as those of the surrounding rhombus, namely, all the diagonals are complete. In this case $C=\cR(C)$, thus it is trivial that $P(C)=P(\cR(C))$.
\end{enumerate}

Lastly, we need to show that the minimizing rhombus is either $\cR_{s-1,s}$ or $\cR_{s,s-1}$ if $n=s(s-1)+k$ and $\cR_{s,s}$ if $n=s^2+k$. We argue by induction over $n$. If $n=1$, then it is trivial that the rhombus $\cR_{1,1}$ minimizes the perimeter and $n=1$ can be represented in the form $s(s-1)+k$ choosing $s=1$ and $k=0$. Suppose now $n>1$ and that the claim holds true for any $m<n$. Suppose that $n-1$ can be written as $s(s-1)+k$. In the other case, when $n=s^2+k$, we can argue in a similar way. If $n-1=s(s-1)+k$, with $0\leq k\leq s-2$ (resp.\ $k=s-1$), then either the rhombus $\cR_{s,s-1}$ or $\cR_{s-1,s}$ (resp.\ the rhombus $\cR_{s,s}$) minimizes the perimeter for an odd cluster with $n$ particles. Indeed, in any other case, the surrounding rhombus could be either $\cR_{s+1,s-1}$ or $\cR_{s-1,s+1}$, which has a strictly greater perimeter in view of \eqref{eq:perimeter}. Note that the assumption $n\leq L(L-2)$ is needed to avoid rhombi with a maximal side equal to $L$, because in view of \eqref{eq:perimeter} we would lose the uniqueness of the minimizing configuration.
\end{proof}

\begin{proof}[Proof of \cref{cor:mortadella}]
We first consider the case in which the area $n$ is of the form $n=s(s-1)+k$. If $s<L/2$, we have that the perimeter of the odd cluster is $P=4(2s+1)$. Since the cluster is contained in a rhombus $\cR_{s,s}$, we have that $n\leq s^2$ and, hence, $\bigl(\frac{P}{4}-1\bigr)^2\geq 4n$. If $L/2\leq s<L$, then the perimeter of the odd cluster is $P=4(2L-2s-1)$. Since the cluster is contained in a rhombus $\cR_{s,s}$ and therefore the complement in $V$ is a rhombus $\cR_{L-s-1,L-s-1}$, we have that $n\geq \frac{L^2}{2}-(L-s-1)^2$. Hence 
\[
    \Bigl(\frac{P}{4}-1\Bigr)^2=4(L-s-1)^2\geq 2(L^2-2n).
\]
Consider now the other case, when the area $n$ is of the form $n=s^2+k$. If $s<L/2$ we have that the perimeter of the odd cluster is $P=4(2s+2)$. Since the cluster is contained in a rhombus either $\cR_{s+1,s}$ or $\cR_{s,s+1}$, we have that $n\leq s^2+s$. Thus, it holds
\[
    \Bigl(\frac{P}{4}-1\Bigr)^2=(2s+1)^2=4(s^2+s)+1\geq 4n+1.
\]
If $L/2\leq s<L$, then the perimeter of the odd cluster is $P=4(2L-2s-2)$. Since the cluster is contained in either a rhombus $\cR_{s+1,s}$ or $\cR_{s,s+1}$ and therefore the complement in $V$ is either a rhombus $\cR_{L-s-2,L-s-1}$ or $\cR_{L-s-1,L-s-2}$, we have that $n\geq \frac{L^2}{2}-(L-s-1)(L-s-2)$. Hence,
\[
    \Bigl(\frac{P}{4}-1\Bigr)^2=(2L-2s-3)^2=4(L-s-1)(L-s-2)+1\geq 4\Bigl(\frac{L^2}{2}-n\Bigr)+1=2(L^2-2n)+1.
\]
By using the condition $n\leq L(L-2)$, inequality~\eqref{eq:diavola} directly follows.
\end{proof}

\begin{proof}[Proof of \cref{prop:maialona}]
Let $\sigma$ be a configuration with real area $\tilde n=2\ell^2+2\ell+1$.
First, we suppose that the set of odd clusters in $\sigma$ is composed only of $j \geq 2$ non-degenerate clusters. Each of them has area $n_i$ and perimeter $p_i$ for $i=1,...,j$. Suppose by contradiction that $\sigma$ has minimal perimeter so that the area of the configuration $\s$ is $n_\sigma=\sum_{i=1}^j n_i$ and its perimeter is $p_\sigma= 4(2\sqrt{n_\sigma}+1)$. By \eqref{eq:diavola}, we have $p_i \geq 4(2\sqrt{n_i}+1)$ for any $i=1,...,j$. Then we obtain that
\begin{equation}
    p_{\sigma}=\sum_{i=1}^j p_i\geq \sum_{i=1}^j 4(2\sqrt{n_i}+1)\geq 8\sqrt{\sum_{i=1}^j n_i}+4j\geq 8\Bigg(\sqrt{\sum_{i=1}^j n_i}+1\Bigg)=4(2\sqrt{n_\sigma}+2),
\end{equation}
that is a contradiction. 

Second, we suppose that the set of odd clusters in $\sigma$ is composed of $k \geq 1$ degenerate clusters and of $j \geq 1$ non-degenerate clusters. Each of these non-degenerate clusters has area $n_i$ and perimeter $p_i$ for $i=1,...,j$, so that $n_\sigma=\sum_{i=1}^j n_i$. We denote by $\tilde p_i$ for $i=1,...,k$ the perimeter of a degenerate cluster. Suppose by contradiction that $\sigma$ has minimal perimeter. 
By \eqref{eq:diavola}, we have $p_i \geq 4(2\sqrt{n_i}+1)$ for any $i=1,...,j$ and $p_\sigma= 4(2\sqrt{n_\sigma}+1)$. Thus, we obtain
\begin{equation}
    p_{\sigma}=\sum_{i=1}^{j} p_i+\sum_{i=1}^{k} \tilde p_i \geq \sum_{i=1}^j 4(2\sqrt{n_i}+1)+4k\geq 8\sqrt{\sum_{i=1}^j n_i}+4j+4k\geq 8\sqrt{\sum_{i=1}^j n_i}+4+4=4(2\sqrt{n_\sigma}+2),
\end{equation}
that is a contradiction. Thus, we obtain that $k=0$ and $j=1$. Since the real area $\tilde n$ of the configuration $\sigma$ is fixed, then also the area $n_\sigma$ is fixed. Thus, given that $\sigma$ contains only one non-degenerate cluster with minimal perimeter and with fixed area, by \cref{cor:mortadella} we obtain that the non-degenerate cluster is the rhombus $\cR_{\ell,\ell}$, which has precisely real area $\tilde n$. 
\end{proof}

\subsection{Results on optimal reference paths}
\label{sec:proofrefpath}

\begin{proof}[Proof of \cref{lmm:cammino*}]
We start by proving (i). The desired path $\omega=(\ee,\omega_1,..., \omega_{k(L)},\eta)$ is obtained as follows. Starting from $\ee$, define the configuration $\omega_1$ as that in which one particle is removed from an empty even site, say $v_1\in V_{\ee}$. Similarly, we define $\omega_2$ as the configuration in which a particle is removed from a site $v_2\in V_{\ee}$ such that $d(v_1,v_2)=2$. Similarly, by removing particles in $v_3,v_4\in V_{\ee}$ in such a way $d(v_i,v_j)=2$ for any $i,j=1,...,4$ and $i\neq j$, we define the configurations $\omega_3$ and $\omega_4$. Note that $\Delta H(\omega_4)=4<L+1$. Thus, we can define the configuration $\omega_5$ as that obtained from $\omega_4$ by adding a particle in the unique unblocked odd site, i.e., the one at distance one from $v_i$ for any $i=1,...,4$. See~\cref{fig:ex2} on the left. 

\begin{figure}[h!]
\centering
\begin{tikzpicture}[scale=0.38]
    \newcommand\radius{0.2}
    \newcommand\listodd{(0,0)}
    \newcommand\listeven{(-1,0),(1,0),(0,1),(0,-1)}

    \draw[help lines] (-4.5,-3.5) grid (9.5,10.5);

    \draw[thick,->] (1,3.3)--(1,2.1);
    \node at (1,3.5) {$v_5$};

    \draw[thick] (-1.5,0.5)--(-1.5,-0.5)--(-0.5,-0.5)--(-0.5,-1.5)--(0.5,-1.5)--(0.5,-0.5)--(1.5,-0.5)--(1.5,0.5)--(0.5,0.5)--(0.5,1.5)--(-0.5,1.5)--(-0.5,0.5)--(-1.5,0.5);

    \foreach \c in \listodd
        \draw[fill=black] \c circle(\radius);
    \foreach \c in \listeven
        \draw[fill=white] \c circle(\radius);
    
\end{tikzpicture}
\hspace{0.5cm}
    \begin{tikzpicture}[scale=0.38]
    \newcommand\radius{0.2}
    \newcommand\listodd{(0,0)}
    \newcommand\listeven{(-1,0),(1,0),(0,1),(0,-1),(1,2),(2,1)}

    \draw[help lines] (-4.5,-3.5) grid (9.5,10.5);

    \draw[thick] (-1.5,0.5)--(-1.5,-0.5)--(-0.5,-0.5)--(-0.5,-1.5)--(0.5,-1.5)--(0.5,-0.5)--(1.5,-0.5)--(1.5,0.5)--(0.5,0.5)--(0.5,1.5)--(-0.5,1.5)--(-0.5,0.5)--(-1.5,0.5);

    \draw[thick] (0.5,1.5) rectangle (1.5,2.5);
    \draw[thick] (1.5,0.5) rectangle (2.5,1.5);
        
    \foreach \c in \listodd
        \draw[fill=black] \c circle(\radius);
    \foreach \c in \listeven
        \draw[fill=white] \c circle(\radius);
    
\end{tikzpicture}
\hspace{0.5cm}
    \begin{tikzpicture}[scale=0.38]
    \newcommand\radius{0.2}
    \newcommand\listodd{(0,0),(1,1),(0,2),(-1,1)}
    \newcommand\listeven{(-1,0),(1,0),(0,1),(0,-1),(1,2),(2,1),(-2,1),(-1,2),(0,3)}

    \draw[help lines] (-4.5,-3.5) grid (9.5,10.5);

     \draw[thick] (-1.5,0.5)--(-1.5,-0.5)--(-0.5,-0.5)--(-0.5,-1.5)--(0.5,-1.5)--(0.5,-0.5)--(1.5,-0.5)--(1.5,0.5)--(2.5,0.5)--(2.5,1.5)--(1.5,1.5)--(1.5,2.5)--(0.5,2.5)--(0.5,3.5)--(-0.5,3.5)--(-0.5,2.5)--(-1.5,2.5)--(-1.5,1.5)--(-2.5,1.5)--(-2.5,0.5)--(-1.5,0.5);

    \foreach \c in \listodd
        \draw[fill=black] \c circle(\radius);
    \foreach \c in \listeven
        \draw[fill=white] \c circle(\radius);
    
\end{tikzpicture}
\caption{Example of the configurations $\omega_5$ (on the left), $\omega_7$ (in the middle) and $\omega_{13}$ (on the right) visited by the path described in the proof of \cref{lmm:cammino*}(i).}
\label{fig:ex2}
\end{figure}
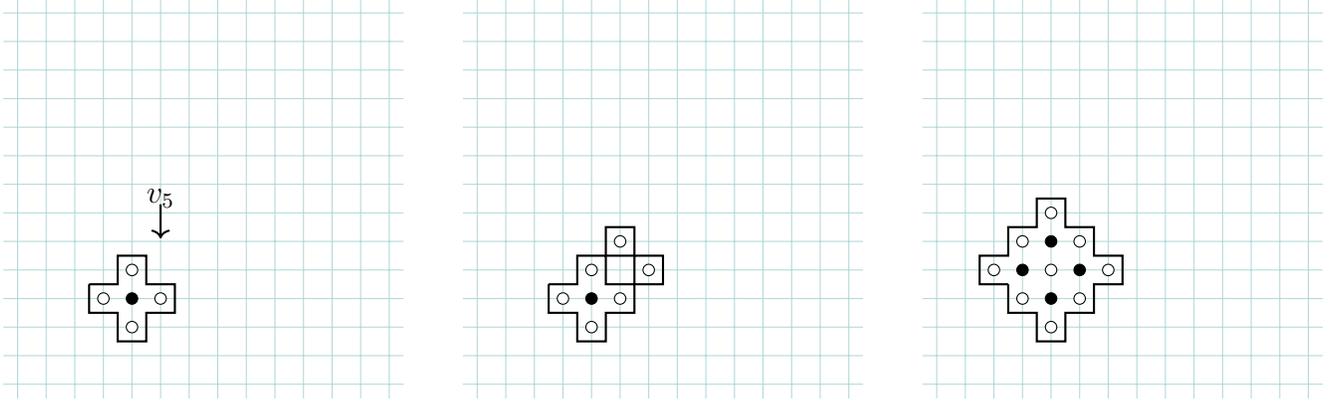

We obtain that $\Delta H(\omega_5)=3<L+1$ and $\omega_5$ is composed of a unique non-degenerate odd cluster, which is $\cR_{1,1}$. To define the configuration $\omega_6$, we remove a particle from a site $v_5\in V_{\ee}$ such that $d(v_i,v_5)=2$ for two indices $i=1,...,4$. Similarly, we define $\omega_7$ in such a way there is an empty odd site with all the neighboring even sites that are empty, see~\cref{fig:ex2} in the middle. Note that $\Delta H(\omega_7)=5<L+1$. Then, we define $\omega_8$ by adding a particle in the unique unblocked odd site, so that $\Delta H(\omega_8)=4<L+1$. Note that $\omega_8$ contains an odd cluster $\cR_{1,2}$. By growing the odd cluster in a spiral fashion, emptying the even sites that are strictly necessary, note that the configuration $\omega_{13}$ contains only the non-degenerate cluster $\cR_{2,2}$ (see~\cref{fig:ex2} on the right). Then, our path visits all the configurations which have a unique non-degenerate odd cluster $C$ such that $C=\cR_{\ell,\ell}$ and $C=\cR_{\ell,\ell+1}$ for any $3\leq \ell\leq \frac{L}{2}-1$ up to the configuration $\omega_{{k(L)}-1}$, whose unique non-degenerate odd cluster is $C=\cR_{\frac{L}{2}-1,\frac{L}{2}-1}$. Then, the configuration $\omega_{k(L)}$ is defined by removing a particle from an even site $w$ at distance two from $C$. Finally, the configuration $\eta$ is obtained by removing a particle from an even site at distance two from $C$ and from $w$. Since the procedure we defined is invariant by translation, given a fixed configuration $\eta\in\cC_{ir}(\ee,\oo)$ is possible to choose the position of the odd clusters in such a way the final configuration of the path we described coincides with the desired $\eta$. It remains to show that $\arg\max_{\xi\in\omega}H(\xi)=\{\eta\}$. To this end, since $\Delta H(\eta)=L+1$ and therefore $\Delta H(\omega_{k(L)})=L$ and $\Delta H(\omega_{{k(L)}-1})=L-1$, we need only to show that 
\[
    \max_{\xi\in\{\ee,\omega_1,...,\omega_{{k(L)}-2}\}}H(\xi)<L+1.
\]
First, we will show that $\Delta H(\h)=3+2(\ell-1)$ by induction over the dimension $\ell=1,...,\frac{L}{2}-1$ of the rhombus $\cR_{\ell,\ell}$ composing the unique odd cluster of the configurations $\eta$ visited by $\omega$. We have already proven the desired property in the case $\ell=1$. Suppose now that the claim holds for $\ell$, with $1\leq \ell\leq\frac{L}{2}-2$, thus we will prove that it holds also for $\ell+1$. To reach the configuration displaying the rhombus $\cR_{\ell,\ell+1}$ starting from $\cR_{\ell,\ell}$, we need first to remove particles from two even sites by increasing the energy by two. Then, we sequentially add a particle in an odd site and remove a particle in an even site until the length of the shorter diagonal is $\ell-1$. Finally, the last move is the addition of a particle in an odd site without the need of removing any particle from an even site. Starting from $\cR_{\ell,\ell+1}$, to obtain $\cR_{\ell+1,\ell+1}$ we follow the same sequence of moves. Thus, for a configuration $\eta$ containing as unique odd cluster a rhombus $\cR_{\ell+1,\ell+1}$ we deduce that $\Delta H(\eta)=3+2(\ell-1)+2$, which proves our claim for $\ell+1$. Along the sequence of moves from $\cR_{\ell,\ell}$ to $\cR_{\ell+1,\ell+1}$, the energy is at most $3+2(\ell-1)+3$, which is strictly less than $L+1$ for $\ell\leq \frac{L}{2}-1$. Note that we do not to consider the case $\ell=\frac{L}{2}-2$, indeed the path $\omega$ stops before reaching the rhombus $\cR_{\frac{L}{2}-1,\frac{L}{2}}$. This concludes the proof of (i).

Now we prove (ii). The desired path $\omega=(\ee,\omega_1,..., \omega_{k(L)},\eta)$ is obtained as follows. Starting from $\ee$, define the configuration $\omega_1$ as that in which one particle is removed from an empty even site, say $v_1\in V_{\ee}$. Similarly, we define $\omega_2$ as the configuration in which a particle is removed from a site $v_2\in V_{\ee}$ such that $d(v_1,v_2)=2$. Similarly, by removing particles in $v_3,v_4\in V_{\ee}$ in such a way $d(v_i,v_j)=2$ for any $i,j=1,...,4$ and $i\neq j$, we define the configurations $\omega_3$ and $\omega_4$. Note that $\Delta H(\omega_4)=4<L+1$. Thus, we can define the configuration $\omega_5$ as the one obtained from $\omega_4$ by adding a particle in the unique unblocked odd site, i.e., the one at distance one from $v_i$ for any $i=1,...,4$. We obtain that $\Delta H(\omega_5)=3<L+1$ and $\omega_5$ is composed of a unique non-degenerate odd cluster, which is $\cR_{1,1}$, see~\cref{fig:ex3} on the left. 

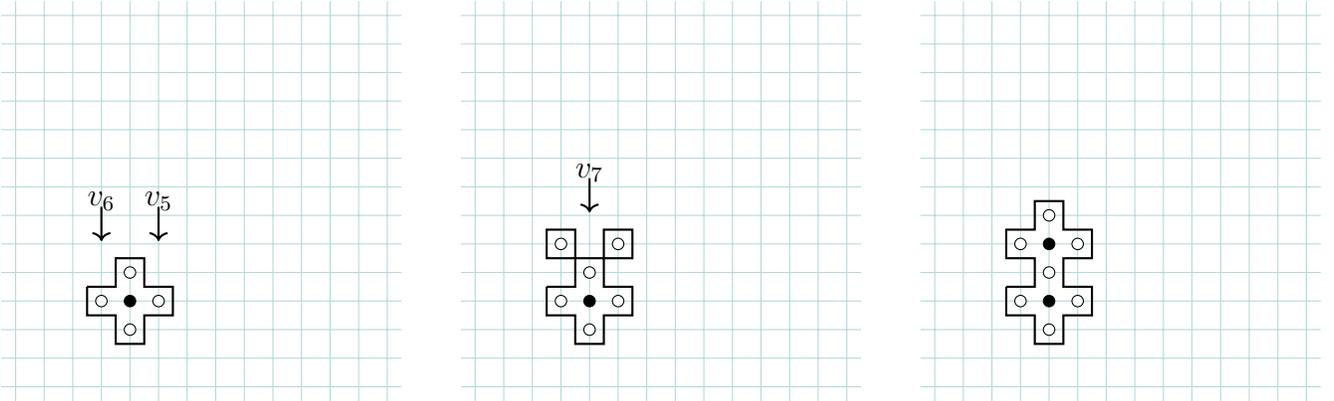
\begin{figure}[h!]
\centering
\begin{tikzpicture}[scale=0.38]
    \newcommand\radius{0.2}
    \newcommand\listodd{(0,0)}
    \newcommand\listeven{(-1,0),(1,0),(0,1),(0,-1)}

    \draw[help lines] (-4.5,-3.5) grid (9.5,10.5);

    \draw[thick,->] (1,3.3)--(1,2.1);
    \node at (1,3.5) {$v_5$};

    \draw[thick,->] (-1,3.3)--(-1,2.1);
    \node at (-1,3.5) {$v_6$};

    \draw[thick] (-1.5,0.5)--(-1.5,-0.5)--(-0.5,-0.5)--(-0.5,-1.5)--(0.5,-1.5)--(0.5,-0.5)--(1.5,-0.5)--(1.5,0.5)--(0.5,0.5)--(0.5,1.5)--(-0.5,1.5)--(-0.5,0.5)--(-1.5,0.5);

    \foreach \c in \listodd
        \draw[fill=black] \c circle(\radius);
    \foreach \c in \listeven
        \draw[fill=white] \c circle(\radius);
    
\end{tikzpicture}
\hspace{0.5cm}
    \begin{tikzpicture}[scale=0.38]
    \newcommand\radius{0.2}
    \newcommand\listodd{(0,0)}
    \newcommand\listeven{(-1,0),(1,0),(0,1),(0,-1),(1,2),(-1,2)}

    \draw[help lines] (-4.5,-3.5) grid (9.5,10.5);

    \draw[thick] (-1.5,0.5)--(-1.5,-0.5)--(-0.5,-0.5)--(-0.5,-1.5)--(0.5,-1.5)--(0.5,-0.5)--(1.5,-0.5)--(1.5,0.5)--(0.5,0.5)--(0.5,1.5)--(-0.5,1.5)--(-0.5,0.5)--(-1.5,0.5);

    \draw[thick] (0.5,1.5) rectangle (1.5,2.5);
    \draw[thick] (-1.5,1.5) rectangle (-0.5,2.5);
        
    \foreach \c in \listodd
        \draw[fill=black] \c circle(\radius);
    \foreach \c in \listeven
        \draw[fill=white] \c circle(\radius);

        \draw[thick,->] (0,4.3)--(0,3.1);
    \node at (0,4.5) {$v_7$};
    
\end{tikzpicture}
\hspace{0.5cm}
    \begin{tikzpicture}[scale=0.38]
    \newcommand\radius{0.2}
     \newcommand\listodd{(0,0),(0,2)}
    \newcommand\listeven{(-1,0),(1,0),(0,1),(0,-1),(1,2),(-1,2),(0,3)}

    \draw[help lines] (-4.5,-3.5) grid (9.5,10.5);

    \draw[thick] (-1.5,0.5)--(-1.5,-0.5)--(-0.5,-0.5)--(-0.5,-1.5)--(0.5,-1.5)--(0.5,-0.5)--(1.5,-0.5)--(1.5,0.5)--(0.5,0.5)--(0.5,1.5)--(1.5,1.5)--(1.5,2.5)--(0.5,2.5)--(0.5,3.5)--(-0.5,3.5)--(-0.5,2.5)--(-1.5,2.5)--(-1.5,1.5)--(-0.5,1.5)--(-0.5,0.5)--(-1.5,0.5);

    \foreach \c in \listodd
        \draw[fill=black] \c circle(\radius);
    \foreach \c in \listeven
        \draw[fill=white] \c circle(\radius);
    
\end{tikzpicture}
\caption{Example of the configurations $\omega_5$ (on the left), $\omega_7$ (in the middle) and $\omega_{9}$ (on the right) visited by the path described in the proof of \cref{lmm:cammino*}(ii).}
\label{fig:ex3}
\end{figure}

Next, we describe four steps that are used in the following iteration. The first step is to define the configuration $\omega_6$ by removing a particle from a site $v_5\in V_{\ee}$ such that $d(v_i,v_5)=2$ for two indices $i=1,...,4$. The second step is to remove a particle from a site $v_6\in V_{\ee}$ in such a way there is an empty odd site with two neighboring empty even sites and the other one is occupied. In this way, we obtain the configuration $\omega_7$ (see~\cref{fig:ex3} in the middle). The third step is to obtain the configuration $\omega_8$ by removing the particle in the site $v_7\in V_{\ee}$ such that $d(v_7,v_5)=d(v_7,v_6)=2$. Note that $\Delta H(\omega_8)=6<L+1$. Then, the last step is to define $\omega_9$ by adding a particle in the unique unblocked odd site, so that $\Delta H(\omega_9)=5<L+1$ and the energy cost of these four steps is $2$. Note that $\omega_9$ is composed of a non-degenerate odd cluster with two odd particles along either the same column or the same row, see~\cref{fig:ex3} on the right. 

From this point onwards, we iterate these four steps for other $\frac{L}{2}-3$ times, until we obtain the configuration $\omega_{k(L)-1}$ in which either a column or a row contains $\frac{L}{2}-1$ odd particles. Note that $\Delta H(\omega_{k(L)-1})=5+2(\frac{L}{2}-3)=L-1$. Then, we repeat the first two steps described above and we reach the configuration $\eta\in\cC_{sb}(\ee,\oo)$ with $\Delta H(\eta)=L+1$. It easy to check that $\arg\max_{\xi\in\omega}=\{\eta\}$. Indeed, we get
\[
    \begin{array}{lll}
    \Delta H (\omega_{i})&=\Delta H (\omega_{i-1})+1 \qquad &\text{ if } 1\leq i\leq 4, \\
    \Delta H (\omega_{2i-1})&=\Delta H (\omega_{2i})-1 \qquad &\text{ if } 3\leq i \leq \frac{k(L)}{2}, \\
    \Delta H (\omega_{2i})&=\Delta H (\omega_{2i-1})+3 \qquad &\text{ if } 3\leq i \leq \frac{k(L)}{2}-1, \\
    \Delta H (\omega_{k(L)})&=\Delta H (\omega_{k(L)-1})+2,
    \end{array}
\]
which concludes the proof of (ii).
\end{proof}

\begin{proof}[Proof of \cref{lmm:cammino**}]
We start by proving (i). Take the path described in \cref{lmm:cammino*}(ii) until it visits for the first time a configuration in $\cC_{sb}(\ee,\oo)$ as in~\cref{fig:degpos} on the left. Starting from such a configuration, add a particle in the unique unblocked odd site. Then, remove a particle from an even site at distance one from an antiknob and finally add a particle in the unique unblocked odd site. Afterward, iterate the sequence of these two moves up to the target configuration $\eta\in\cC_{ib}(\ee,\oo)$. By construction, the resulting path has the desired property.

Let us now focus on case (ii). Consider the path described in \cref{lmm:cammino*}(i) until it visits for the first time a configuration in $\cC_{ir}(\ee,\oo)$. Starting from such configuration, add a particle in the antiknob, obtaining a configuration $\eta'$ that displays a unique non-degenerate odd cluster, which is a rhombus $\cR_{\frac{L}{2}-1,\frac{L}{2}-1}$ with a single protuberance. Starting from $\eta'$, consider the path returned filling algorithm $\bar\omega$ up to the configuration $\eta\in\cC_{gr}(\ee,\oo)$. The described path has the desired property thanks to \cref{lem:expansion}. 

Consider now the concatenation of the paths described in \cref{lmm:cammino*}(ii) and \cref{lmm:cammino**}(i) until it visits for the first time a configuration in $\cC_{ib}(\ee,\oo)$. Starting from such a configuration, by iterating the pair of moves consisting in adding a particle in the unique unblocked odd site and removing a particle at distance one from an antiknob, it is possible to obtain a configuration with a unique non-degenerate cluster that contains $\cR_{\frac{L}{2}-1,\frac{L}{2}-1}$. By construction, this resulting configuration is in $\cC_{gr}(\ee,\oo)$ and it is possible to iterate this couple of moves up to $\eta$. By construction, the resulting path has the desired property.

Consider now case (iii). Take the path described in \cref{lmm:cammino*}(ii) until it visits for the first time the configuration in $\cC_{sb}(\ee,\oo)$ as in~\cref{fig:selle2} on the right. Starting from such a configuration, add a particle in the unique unblocked odd site. Then, remove a particle from an even site at distance one from an antiknob and finally add a particle in the unique unblocked odd site. This configuration is in $\cC_{mb}(\ee,\oo)$. Afterward, iterate the sequence of these two moves up to the target configuration $\eta$. By construction, the resulting path has the desired property. 

Consider now the concatenation of the paths described in \cref{lmm:cammino*}(ii) and \cref{lmm:cammino**}(i) until it visits for the first time a configuration in $\cC_{ib}(\ee,\oo)$. Starting from such configuration, with the same procedure described above it is possible to reach the target configuration $\eta\in\cC_{mb}(\ee,\oo)$ by visiting only saddles in $\cC_{ib}(\ee,\oo)\cup\cC_{mb}(\ee,\oo)$.

Let us now focus on case (iv). Consider the concatenation of the paths described in \cref{lmm:cammino*}(i) and \cref{lmm:cammino**}(ii) until it visits for the first time a configuration in $\cC_{gr}(\ee,\oo)$. Starting from such configuration, add a particle in the antiknob, obtaining a configuration $\eta'$. Note that $\eta'$ is composed of a unique non-degenerate odd cluster, which is a rhombus $\cR_{\frac{L}{2}-1,\frac{L}{2}-1}$ with a single protuberance. Starting from $\eta'$, consider the path returned by filling algorithm $\bar\omega$ up to the configuration $\eta\in\cC_{cr}(\ee,\oo)$. The described path has the desired property thanks to \cref{lem:expansion}. 

Consider now the concatenation of the paths described in \cref{lmm:cammino*}(ii) and \cref{lmm:cammino**}(iii) until it visits for the first time a configuration in $\cC_{mb}(\ee,\oo)$. Starting from such configuration, by iterating the pair of moves consisting in adding a particle in the unique unblocked odd site and removing a particle at distance one from an antiknob, it is possible to obtain a configuration with a unique non-degenerate cluster that contains $\cR_{\frac{L}{2}-1,\frac{L}{2}}$. By construction, this resulting configuration is in $\cC_{cr}(\ee,\oo)$ and it is possible to iterate this couple of moves up to $\eta$. By construction, the resulting path has the desired property.
\end{proof}

\begin{proof}[Proof of \cref{lmm:cammino6}]
We start by proving (i). Arguing as in the proof of \cref{lmm:cammino**}(ii) we can show that there exists a path with the desired property that connects $\eta$ to $\bar\eta\in\cC_{cr}(\ee,\oo)$, where the unique non-degenerate cluster of $\bar\eta$ is $\cR_{\frac{L}{2}-1,\frac{L}{2}}$ with attached a bar of length $\frac{L}{2}-2$ and there is a degenerate rhombus $\cR_{0,0}$ at distance one from an antiknob. Since the configuration displays two antiknobs, it is possible to sequentially add two particles in odd sites. Thus, by proceeding in this way the path reaches $\oo$ without visiting any other saddle and so the described path has the desired property.

Finally, consider case (ii). By arguing as in the proof of \cref{lmm:cammino**}(iii)--(iv) we can show that there exists a path with the desired property that connects $\eta$ to $\bar\eta\in\cC_{gr}(\ee,\oo)\cup\cC_{mb}(\ee,\oo)$. If $\bar\eta\in\cC_{gr}(\ee,\oo)$, the claim follows by the previous argument. Otherwise, the configuration $\bar\eta$ has two antiknobs after arguing as above. In either cases, the described path has the desired property.
\end{proof}

\begin{proof}[Proof of \cref{lmm:nocammino}]

Consider first case (i). By construction, every non-backtracking optimal path from $\ee$ to $\oo$ that crosses a configuration in $\cC_{ir}(\ee,\oo)$ has to visit a configuration in $\cC_{gr}(\ee,\oo)$. Indeed, each move consists in adding a particle in an unblocked odd site or removing a particle from an even site. Thus, it remains to consider backtracking optimal paths only. In order to visit a saddle in $\cC_{ib}(\ee,\oo)$, a column (or row) with precisely $\frac{L}{2}-1$ particles arranged in odd sites needs to be created along these paths. To proceed, since we are considering backtracking optimal paths, the unique possibility is to visit a configuration in $\cC_{sb}(\ee,\oo)$ before reaching $\cC_{ib}(\ee,\oo)$. Thus, case (i) is concluded.

Consider now case (ii). The claim follows after noting that every configuration in $\cC_{cr}(\ee,\oo)$ contains one odd vertical (resp.\ horizontal) bridge $B$, where the two neighboring columns (resp.\ rows) to $B$ contains $\frac{L}{2}-1$ odd particles each. Thus, all the configurations in $\cC_{cr}(\ee,\oo)$ differ from a configuration in $\cC_{ib}(\ee,\oo)$ in at least two odd sites. Then, suppose to have $\eta\in\cC_{ib}(\ee,\oo)$ and $\eta'\in\cC_{cr}(\ee,\oo)$ such that they differ in only two odd sites. Starting from $\eta$, if a bridge is created, then the resulting configuration is in $\cC_{mb}(\ee,\oo)$, while if the resulting configuration contains two neighboring columns (or rows) with exactly $\frac{L}{2}-1$ particles in odd sites, then it belongs to $\cC_{gr}(\ee,\oo)$. If the two configurations differ in at least three sites, we argue as above.

Consider now case (iii). The claim follows after arguing as in case (i) and noting that, in order to create a bridge, the path has to visit the set $\cC_{sb}(\ee,\oo)$. Lastly, in case (iv), the claim follows after arguing as in case (iii).
\end{proof}

\subsection{Results on the saddles lying in the manifold $\cV_{m^*}$}
\label{sec:proofmanifold}

\begin{proof}[Proof of \cref{lmm:gorgonzola}]
We analyze separately the three cases.

\medskip
\noindent
{\bf Case (i).}
Suppose by contradiction that there exists a non-backtracking path $\omega'\in(\ee\rightarrow\oo)_{\mathrm{opt}}$ that crosses $\sigma\in\cV_{m^*}\setminus\cC_{ir}(\ee,\oo)$ such that $R(\cO^{nd}(\s))$ and $R(\cO(\sigma))$ do not wind around the torus.
First, since $\sigma\in\cV_{m^*}$ and $\Delta H(\s)\leq L+1$, we get
\begin{equation}\label{burratina}
\begin{cases}
    e(\sigma)&\geq\frac{L^2}{4}-\frac{3}{2}, \\
    o(\sigma)&\geq\frac{L^2}{4}-L+\frac{3}{2}.
\end{cases}
\end{equation}
Since $o(\sigma)>0$ for any $L$, we deduce that $\s$ cannot contain only degenerate clusters and therefore it contains at least an odd non-degenerate cluster.
Thus, one of the following cases occurs:
\begin{itemize}
    \item[(1)] $\cO(\sigma)$ consists of at least two non-degenerate odd clusters;
    \item[(2)] $\cO(\sigma)$ consists of a single non-degenerate odd cluster different from $\cR_{\frac{L}{2}-1,\frac{L}{2}-1}$ and possibly some degenerate clusters;
    \item[(3)] $\cO(\sigma)$ consists of a single non-degenerate odd rhombus equal to $\cR_{\frac{L}{2}-1,\frac{L}{2}-1}$ and at least one degenerate odd cluster $\cR_{0,0}$ at distance greater than one from the non-degenerate one.
\end{itemize}

We consider the rhombus surrounding the odd non-degenerate region for all the above cases. Due to the isoperimetric inequality of \cref{prop:maialona}, this rhombus $\cR$ has a perimeter $P(\mathcal{R})$ greater than or equal to $P(\cR_{\frac{L}{2}-1,\frac{L}{2}-1})$, i.e., $P(\mathcal{R}) \geq 4(L-1)$. In particular, for cases (1) and (2) we have $P(\mathcal{R}) >4(L-1)$ since $\cO(\s)\neq \cR_{\frac{L}{2}-1,\frac{L}{2}-1}$, and for case (3) we have $P(\mathcal{R})=4(L-1)$. 
Let $P_i$ denotes the perimeter of the $i$-th cluster, with $i=1,...,k$. Note that $P(\s)=\sum_{i=1}^k P_i$. We let $k_d \geq 0$ (resp.\ $k_{nd} \geq 1$) the number of degenerate clusters $\cR_{0,0}$ (resp.\ non-degenerate clusters) such that the total number of clusters is $k=k_{d}+k_{nd} \geq 1$.

\medskip
\noindent
{\bf Subcase (1).} In this case $k_{nd}\geq2$, so that $k \geq 2$. Denote by $\tilde e_i$ (resp.\ $o_i$) the number of empty even sites (resp.\ occupied odd sites) of the $i$-th cluster. Using~\eqref{eq:identitypC}, \eqref{eq:ps} and \eqref{eq:contour}, we obtain
\begin{equation}\label{eq:salsiccina}
    \begin{array}{ll}
    \Delta H(\sigma)&=\displaystyle \sum_{i=1}^{k} (\tilde e_i-o_i)=\sum_{i=1}^{k_d} \tilde e_i+\sum_{i=1}^{k_{nd}} (\tilde e_i-o_i) \geq \tilde e_{d}+\Big(\sum_{i=1}^{k_{nd}}\tilde e_i+2(k_{nd}-1)-\sum_{i=1}^{k_{nd}} o_i\Big)\\
    &=\tilde e_{d}+\Big(\displaystyle\sum_{i=1}^{k_{nd}}\tilde e_i-\sum_{i=1}^{k_{nd}} o_i\Big)+2(k_{nd}-1)> \tilde e_{d}+(L-1)+2(k_{nd}-1)\\
    &\geq \tilde e_{d}+L+1,
    \end{array}
\end{equation}
where $\tilde e_{d}=\sum_{i=1}^{k_{d}}\tilde e_i$, the first inequality follows from the fact that the difference between the perimeter of two disjoint non-degenerate clusters and that of the cluster obtained by attaching the two is at least two, and the last inequality follows from the isoperimetric inequality applied to the cluster obtained by attaching the $k$ clusters forming $\sigma$. Then, from inequalities \eqref{eq:salsiccina} and $\Delta H(\sigma)\leq L+1$, it follows that $\tilde e_d<0$, which is a contradiction since $\tilde e_d \geq 0$. 

\medskip
\noindent
{\bf Subcase (2).} The unique non-degenerate odd cluster of $\sigma$ has a shape different from a rhombus $\cR_{\frac{L}{2}-1,\frac{L}{2}-1}$ by assumption. 
In this case $k_{nd}=1$ and $k_d\geq0$, so that $k=k_d+1$. Thus, we obtain 
\begin{equation}\label{eq:salsiccina2} 
\begin{array}{ll}
    \Delta H(\sigma)&=\displaystyle \sum_{i=1}^{k_d+1} (\tilde e_i-o_i)=\sum_{i=1}^{k_d} \tilde e_i+(\tilde e_{k_d+1}-o_{k_d+1})> \tilde e_{d}+L-1,
\end{array}
\end{equation}
where $\tilde e_{d}=\sum_{i=1}^{k_{d}}\tilde e_i$, the first equality follows from the fact that $\sigma$ contains only one non-degenerate odd cluster, and the last inequality follows from the isoperimetric inequality applied to the non-degenerate cluster. Then by \eqref{eq:salsiccina2} and the fact that $\Delta H(\sigma)\leq L+1$, we find $\tilde e_d \leq 1$ and so $k_d \leq 1$. 

The configuration $\sigma$ thus falls in one of the following three subcases: 
\begin{itemize}
    \item[\textbf{(2a)}] the only non-degenerate odd cluster of $\sigma$ is a rhombus $\cR_{\ell_1,\ell_2} \neq \cR_{\frac{L}{2}-1,\frac{L}{2}-1}$;
    \item[\textbf{(2b)}] the non-degenerate odd cluster of $\sigma$ is a cluster with $m \geq 1$ empty odd sites corresponding to some broken diagonal and $q=0$ shorter diagonals;
    \item[\textbf{(2c)}] the non-degenerate odd cluster of $\sigma$ is a cluster with $m \geq 0$ empty odd sites corresponding to some broken diagonal and $q \geq 1$ shorter diagonals.
\end{itemize}
We can ignore the case $m=0$ and $q=0$ as it is the one corresponding to the rhombus $\cR_{\frac{L}{2}-1,\frac{L}{2}-1}$.

\medskip
{\bf Subcase (2a).} This case is not admissible since, if $R_{\ell_1,\ell_2} \neq R_{\frac{L}{2}-1,\frac{L}{2}-1}$, we have
\begin{align}
    \begin{cases}
        \ell_1+\ell_2+1+k_d&=L+1, \notag \\
        \ell_1+\ell_2+1 &> L-1,
    \end{cases}
\end{align}
and this implies $k_d>2$, which is in contradiction with the assumption $k_d \leq 1$.

\medskip
{\bf Subcase (2b).} From the assumptions on $\sigma$, it follows that
\begin{equation}\label{eq:paninozzo}
    \begin{cases}
        \ell_1+\ell_2+1+k_d+m&=L+1, \\
        \ell_1+\ell_2+1 &> L-1.
    \end{cases}
\end{equation}
These give that $k_d+m< 2$, which, in view of the inequalities $k_d \leq 1$ and $m \geq 1$, then imply that $k_d=0$ and $m=1$. Thus, we deduce that
\begin{align}
    \begin{cases}
        o(\s)&=\ell_1\ell_2-1, \notag \\
        e(\s)&=(\ell_1+1)(\ell_2+1),
    \end{cases}
\end{align}
so that $\ell_1+\ell_2=L-5$ since $\s\in\cV_{m^*}$, but this contradicts the condition $\ell_1+\ell_2+1+k_d+m=L+1$ in \eqref{eq:paninozzo}. 

\medskip
{\bf Subcase (2c).} From the assumptions on $\sigma$, it follows that
\begin{equation}\label{eq:paninozzo2}
    \begin{cases}
        \ell_1+\ell_2+1+k_d+m&=L+1, \\
        2(\ell_1+\ell_2+1)-q & > 2(L-1).
    \end{cases}
\end{equation}
This implies that $2(k_d+m)+q\leq 3$, so we distinguish three subcases: \textbf{(2c-I)} $k_d=m=0$ and $q \in \{1,2,3\}$; \textbf{(2c-II)} $k_d=1$, $m=0$ and $q=1$; and \textbf{(2c-III)} $k_d=0$, $m=1$ and $q=1$. The other subcases are not possible in view of the conditions $k_d \leq 1$, $m \geq 0$, and $q \geq 1$.\\

For subcase (2c-I), $k_d=m=0$ and $q \in \{1,2,3\}$.  
Thus, by letting $s$ be the total number of empty sites needed to be filled in order for the shorter diagonals to become complete, we deduce that 
\begin{align}
    \begin{cases}
        o(\s)&=\ell_1\ell_2-s, \notag \\
        e(\s)&=(\ell_1+1)(\ell_2+1)-s,
    \end{cases}
\end{align}
so that $\ell_1+\ell_2=L-4$ since $\s\in\cV_{m^*}$, but this contradicts identity $\ell_1+\ell_2+1+k_d+m=L+1$ in \eqref{eq:paninozzo2}. 
The claims for (2c-II) and (2c-III) follow by arguing as in (2c-I).

\medskip
{\bf Subcase (3).} First, note that $k_d\leq2$, otherwise $\Delta H(\s)>L+1$ and therefore the path $\o$ would be not optimal. 

If $k_d=2$, by using the non-backtracking property of the path $\o$, it follows that there is no possible move to cross the next manifold towards $\oo$ along an optimal path.

If $k_d=1$, then $\D H(\s)=L$ so that the energy along the path can increase by at most 1 to reach $\oo$. The unique possible move is to remove a particle from an empty site, obtaining a configuration with $k_d=2$. We can then prove the claim by arguing as in the case $k_d=2$.

\medskip
\noindent
{\bf Case (ii).} 
Suppose by contradiction that there exists a non-backtracking path $\omega'\in(\ee\rightarrow\oo)_{\mathrm{opt}}$ that crosses that crosses $\sigma\in\cV_{m^*}\setminus\cC_{ib}(\ee,\oo)$ such that $R(\cO^{nd}(\s))$ does not wind around the torus and $R(\cO(\sigma))$ does.
We observe that \eqref{burratina} holds for the configuration $\sigma$ and, therefore, it contains at least one odd non-degenerate cluster and it cannot contain only degenerate clusters. Thus, we consider the following subcases:
\begin{itemize}
    \item[(1)] $\cO(\s)$ consists of at least two non-degenerate odd clusters;
    \item[(2)] $\cO(\s)$ consists of a single non-degenerate odd cluster different from $\cR_{\frac{L}{2}-1,\frac{L}{2}-1}$ and possibly some degenerate clusters;
    \item[(3)] $\cO(\s)$ consists of a single non-degenerate odd rhombus equal to $\cR_{\frac{L}{2}-1,\frac{L}{2}-1}$ and at least one degenerate odd cluster $\cR_{0,0}$ at distance greater than one from the non-degenerate one.
\end{itemize}

\medskip
{\bf Subcase (1).} Suppose that $\sigma$ contains $k \geq 2$ non-degenerate clusters $C_1(\sigma),...,C_k(\sigma)$. By assumption all the rhombi surrounding $C_i(\s)$ does not wind around the torus for any $i=1,...,k$, thus we can argue as in case \textbf{(i-2c-II)} above. 

\medskip
{\bf Subcase (2).} By arguing as in case \textbf{(i-2)} above, we deduce that this case is possible only when $\cR(\cO^{nd}(\s))$ does not wind around the torus and $\cR(\cO(\s))$ does, so that from \eqref{burratina} we deduce that there exists only one column or row with less than $L/2$ particles, but this contradicts the fact that $\s\notin\cC_{ib}(\ee,\oo)$. 

\medskip
{\bf Subcase (3).} By arguing as in case \textbf{(i-3)} above, we deduce that this case is not possible.

\medskip
\noindent
{\bf Case (iii).}
Suppose by contradiction that there exists a non-backtracking path $\omega'\in(\ee\rightarrow\oo)_{\mathrm{opt}}$ that crosses that crosses $\sigma \in\cV_{m^*}\setminus\cC_{mb}(\ee,\oo)$ such that both $R(\cO^{nd}(\s))$ and $R(\cO(\sigma))$ wind around the torus. 
We observe that \eqref{burratina} holds for the configuration $\sigma$ and therefore it contains at least one odd non-degenerate cluster and it cannot contain only degenerate clusters. Thus, we consider the following subcases:
\begin{itemize}
    \item[(1)] $\cO(\s)$ consists of at least two non-degenerate odd clusters;
    \item[(2)] $\cO(\s)$ consists of a single non-degenerate odd cluster different from $\cR_{\frac{L}{2}-1,\frac{L}{2}-1}$ and possibly some degenerate clusters;
    \item[(3)] $\cO(\s)$ consists of a single non-degenerate odd rhombus equal to $\cR_{\frac{L}{2}-1,\frac{L}{2}-1}$ and at least one degenerate odd cluster $\cR_{0,0}$ at distance greater than one from the non-degenerate one.
\end{itemize}

\medskip
{\bf Subcase (1).} Suppose that $\sigma$ contains $k \geq 2$ non-degenerate clusters $C_1(\sigma),...,C_k(\sigma)$. If the rhombus surrounding $C_i(\s)$ does not wind around the torus for any $i=1,...,k$, we can argue as in case \textbf{(i-1)} above. Otherwise, suppose that there exists an index $i$ such that $\cR(C_i(\s))$ winds around the torus.
Thus, there exists at least one row or column that contains $\frac{L}{2}$ particles, say a column. Since $k\geq2$, this implies that 
there exists a row containing two odd particles which belong to $C_i(\s)$ and another disjoint cluster. Thus, the energy difference contribution $\D H$ along this row or column is at least two. In addition, all the other $L$ rows or columns composing $C_i(\sigma)$ have an energy contribution of at least one. Thus, the total contribution is $\Delta H(\sigma)> 2+L-1=L+1$, where the strict inequality follows from $k\geq2$. This contradicts the assumption $\Delta H(\sigma)\leq L+1$ and therefore this case is not admissible.

\medskip
{\bf Subcase (2).} By arguing as in case \textbf{(i-2)} above, we deduce that this case is possible only when both $\cR(\cO^{nd}(\s)$ and $\cR(\cO(\s)$ wind around the torus, so that from \eqref{burratina} we deduce that there exists at least one column or row with strictly less than $L/2$ particles, but this contradicts the fact that $\s\notin\cC_{mb}(\ee,\oo)$. 

\medskip
{\bf Subcase (3).} By arguing as in case \textbf{(i-3)} above, we deduce that this case is not possible.
\end{proof}

\section{Conclusions and future work}
\label{sec:future}
This work concludes the analysis of the metastable behavior for the hard-core model on a square grid graph with periodic boundary conditions initiated by \cite{NZB15}. In that paper, this interacting particle system was shown to have two stable states and the energy barrier between them was already identified. However, the argument carried out in that paper did not provide any geometrical insight into the typical trajectories in the low-temperature regime and did not characterize the critical configurations. The goal of this paper was precisely to fill this gap. More precisely, we provide the geometrical description of all the essential saddles for this transition and we highlight how these communicate with each other without exceeding the critical energy barrier.

The extension to other types of lattices naturally arises in this context. Indeed, in \cite{Zocca2018bis} the authors investigate the same model on the triangular lattice studying the asymptotic behavior of the transition times between stable states but without providing any information on the critical configurations. The type of analysis carried out in this paper could be useful to tackle that problem, even if it looks more challenging since there are three stable states and isoperimetric inequalities are probably harder to derive. This will be the focus of future work. Another possible direction could be the study of the metastable behavior of the hard-core model on square grid graphs but with different boundary conditions or in the presence of some impurities. However, in this case, we expect the transition between two stable states to most likely occur by heterogeneous nucleation starting from the boundary, hence breaking the intrinsic symmetry and translation-invariance properties of the critical configurations. On the other hand, we expect that the techniques and the machinery developed in this paper to be useful to identify critical configurations also of other interacting particle systems on finite square lattices with similar blocking effects, e.g., the Widom-Rowlinson model \cite{Zocca2018}.

\begin{appendices}
\section{}
\label{sec:app}
\begin{proof}[Proof of \cref{lem:diff}]
We start by proving (i). Given $0\leq k_1^*\leq \ell_1$ and $\ell_2+1\leq j_1^*\leq L-1$, we want to show that there exists $0\leq k_2^*\leq \ell_1$ and $0\leq j_2^*\leq \ell_2$ such that 
\begin{equation}\label{eq:sistcong}
\begin{cases}
    k_1^*+j_1^*=k_2^*+j_2^* &\, (\hbox{mod } L), \\
    k_1^*-j_1^*=k_2^*-j_2^* &\, (\hbox{mod } L).
    \end{cases}
\end{equation}
The choice $j_2^*=j_1^*+L/2 \ \, (\hbox{mod } L)$ and $k_2^*=k_1^*+L/2 \ \, (\hbox{mod } L)$ implies the claim in (i), since the chosen indices $j_2^*$ and $k_2^*$ satisfy \eqref{eq:sistcong} and they are modulo $L$ such that
\[
    \begin{cases}
    j_2^*\geq \ell_2+1+\frac{L}{2} 
    \geq1, \\
    j_2^*\leq L-1+\frac{L}{2} 
    \leq \ell_2-1, \\
    k_2^*\geq \frac{L}{2}, \\ 
    k_2^*\leq \ell_1+\frac{L}{2}  
    \leq\frac{L}{2}-2\leq \ell_1-2.
    \end{cases}
\]
By interchanging the role of $k_1^*$ and $k_2^*$ and arguing in the same way, the proof of (i) is concluded.
The two inclusions stated in (ii) can be proved by arguing in an analogous way.
\end{proof}

\medskip
\begin{proof}[Proof of \cref{lem:complrombo}]
In this proof, all the sums will be tacitly assumed to be taken modulo $L$. Denote $\eta=(\eta_1,\eta_2) \in V_\oo$.  We analyze separately each case. \medskip

\noindent
{\bf Case (i).} Considering $S_{\ell_1,\ell_2}(\eta) \subset V_\oo$ and $\partial^+ S_{\ell_1,\ell_2}(\eta) \subset V_\ee$, we observe that
\begin{align}
    V \setminus \cR_{\ell_1,\ell_2}(\eta) & = V \setminus \{S_{\ell_1,\ell_2}(\eta) \cup \partial ^+ S_{\ell_1,\ell_2}(\eta)\} \notag \\
    & = \{ V_\oo \cup V_\ee \} \setminus \{S_{\ell_1,\ell_2}(\eta) \cup \partial^+S_{\ell_1,\ell_2}(\eta)\} \notag \\
    & = \{V_\oo \setminus \{S_{\ell_1,\ell_2}(\eta) \cup \partial^+S_{\ell_1,\ell_2}(\eta)\}\} \cup \{ V_\ee \setminus \{S_{\ell_1,\ell_2}(\eta) \cup \partial^+S_{\ell_1,\ell_2}(\eta)\} \} \notag \\
    & = \{V_\oo \setminus S_{\ell_1,\ell_2}(\eta) \} \cup \{ V_\ee \setminus \partial^+S_{\ell_1,\ell_2}(\eta)\}.
\end{align}
Thus, we want to show that the two subsets are equal to $S_{\hat{l}_1,\hat{l}_2}(\hat{\eta})$ and $\partial^+S_{\hat{l}_1,\hat{l}_2}(\hat{\eta})$ for some $0 \leq \hat{l}_1,\hat{l}_2 \leq L-1$ and $\hat{\eta} \in V_\ee$. In addition, we may write 
\begin{equation}\label{eq:sitipari}
    V_\ee = \bigcup_{0 \leq k', \, j' \leq L-1 } \{(k'+j',k'-j')\}
\end{equation}
and
\begin{equation}\label{eq:extS}
    \partial^+ S_{\ell_1,\ell_2}(\eta)= \bigcup_{\substack{ 0 \leq k \leq \ell_1 \\ 0 \leq j \leq \ell_2} } \{ (\eta_1+k+j-1, \, \eta_2+k-j) \}.
\end{equation}
Thus, we have
\begin{align}
    V_\ee \setminus \partial^+S_{\ell_1,\ell_2}(\eta) & =\bigcup_{ 0 \leq k',j' \leq L-1 } \{(k'+j'+\eta_1-1,k'-j'+\eta_2)\} \setminus \bigcup_{\substack{ 0 \leq k \leq \ell_1 \\ 0 \leq j \leq \ell_2} } \{ (\eta_1+k+j-1, \, \eta_2+k-j) \} \notag\\
    &= \bigcup_{\substack{ \ell_1+1 \leq \tilde k \leq L-1 \\ \ell_2+1 \leq \tilde j \leq L-1} } \{(\tilde k+ \tilde j + \eta_1-1, \tilde k - \tilde j + \eta_2) \} \notag \\
    &= \bigcup_{\substack{ 0 \leq \hat{k} \leq L-\ell_1-2 \\ 0 \leq \hat{j} \leq L-\ell_2-2} } \{(\hat{k}+\hat{j}+\ell_1+\ell_2+2+\eta_1-1, \hat{k}-\hat{j}+\ell_1-\ell_2+\eta_2\},
\end{align}
where at the second equality we used \cref{lem:diff}(i) and at the last equality we used the change of variables $\hat{k}=\tilde k -(\ell_1+1)$ and $\hat{j}= \tilde j -(\ell_2+1)$. Now, we consider $\hat{\eta}=(\ell_1+\ell_2+1+\eta_1,\ell_1-\ell_2+\eta_2) \in V_\ee$ and we obtain 
\begin{equation}
 V_\ee \setminus \partial^+S_{\ell_1,\ell_2}(\eta)= \bigcup_{\substack{ 0 \leq \hat{k} \leq L-\ell_1-2 \\ 0 \leq \hat{j} \leq L-\ell_2-2} } \{(\hat{k}+\hat{j}+\hat{\eta}_1, \hat{k}-\hat{j}+\hat{\eta}_2) \}
 \end{equation}
 Thus, we have $V_\ee \setminus \partial^+ S_{\ell_1,\ell_2}(\eta)= S_{L-\ell_1-1,L-\ell_2-1}(\hat{\eta})\subseteq V_\ee$. By arguing as above, we prove that $V_\oo \setminus S_{\ell_1,\ell_2}(\eta)= \partial^+ S_{L-\ell_1-1,L-\ell_2-1}(\hat{\eta})\subseteq V_\oo$.

\medskip
\noindent
{\bf Case (ii).} Without loss of generality we may assume $\ell_1=\min\{\ell_1,\ell_2\}$. For $\eta=(\eta_1,\eta_2)\in V_\ee$, by using \eqref{eq:sitipari} and \eqref{eq:stima1}, we have 

\begin{equation}
\begin{array}{ll}
  V_\ee&= \displaystyle\bigcup_{\substack{0\leq k \leq \ell_1 \\ 0\leq j \leq L-1}} \{(k+j+\eta_1,k-j+\eta_2)\} \cup \displaystyle\bigcup_{\substack{\ell_1+1\leq k \leq L-1 \\ 0\leq j \leq L-1}} \{(k+j+\eta_1,k-j+\eta_2)\} \\
 &=\displaystyle\bigcup_{\substack{0\leq k \leq \ell_1 \\ 0\leq j \leq L-1}} \{(k+j+\eta_1,k-j+\eta_2)\} = \partial^+ S_{\ell_1,L-1}(\eta).
\end{array}
\end{equation}
Thus, it follows that all the even sites belong to the rhombus $\cR_{\ell_1,\ell_2}(\eta)$. In addition, by using \eqref{eq:stima2} we obtain
\begin{align}
S_{\ell_1,L-1}(\eta) &=\bigcup_{ \substack{0 \leq k\leq \ell_1-1 \\ 0\leq j \leq L-2} } \{(\eta_1+k'+j',\eta_2+k'-j')\} \notag \\
&=\bigcup_{ \substack{0 \leq k\leq \frac{L}{2}-1 \\ 0\leq j \leq L-2} } \{(\eta_1+k+j,\eta_2+k-j)\} \cup \bigcup_{\substack{\frac{L}{2} \leq k\leq \ell_1-1 \\ 0\leq j \leq L-2}} \{(\hat\eta_1+k+j,\hat\eta_2+k-j)\} \notag \\
&=\bigcup_{ \substack{0 \leq k\leq \frac{L}{2}-1 \\ 0\leq j \leq L-2} } \{(\eta_1+k+j,\eta_2+k-j)\} \cup \bigcup_{L/2 \leq k\leq \ell_1-1} \{(\hat\eta_1+k,\hat\eta_2+k)\}.
\end{align}
Thus, we deduce that the rhombus $\cR_{\ell_1,L-1}(\eta)$ contains $L^2/2-L+\ell_1$ odd sites, which implies that its complement in $V$ contains $L-\ell_1$ odd sites.

\medskip
\noindent
{\bf Case (iii).}  Without loss of generality we may assume $\ell_1=\min\{\ell_1,\ell_2\}$. In this case, after using the same argument we have shown above, for some $\hat\eta=(\hat\eta_1,\hat\eta_2)\in V_\oo$ we obtain that
\begin{equation}
V_\oo\setminus S_{\ell_1,L}(\eta)=\bigcup_{\substack{\ell_1\leq k\leq \frac{L}{2}-1 \\ 0\leq j\leq L-1}} \{(\hat\eta_1+k+j,\hat\eta_2+k-j)\}
\end{equation}
and
\begin{equation}
V_\ee\setminus \partial^+ S_{\ell_1,L}(\eta)=\bigcup_{\substack{\ell_1+1\leq k\leq \frac{L}{2}-1 \\ 0\leq j\leq L-1}} \{(\hat\eta_1+k+j-1,\hat\eta_2+k-j)\}.
\end{equation}
This implies that the rhombus $\cR_{\ell_1,L}(\eta)$ contains $L\, \ell_1$ odd sites and $L(\ell_1+1)$ even sites.

\medskip
\noindent
{\bf Case (iv).} By arguing as above, we can show that the complement of the rhombus $\cR_{\ell_1,\ell_2}(\eta)$ in $V$ has no even and odd sites and therefore $\cR_{\ell_1,\ell_2}\equiv V$.
\end{proof}
\end{appendices}

\printbibliography

\end{document}